\documentclass[11pt,twoside]{article}
\usepackage{fancyhdr}
\usepackage[colorlinks,citecolor=blue,urlcolor=blue,linkcolor=blue,bookmarks=false]{hyperref}
\usepackage{amsfonts,epsfig,graphicx}
\usepackage{afterpage}
\usepackage{amsmath,amssymb,amsthm} 
\usepackage{fullpage}
\usepackage[T1]{fontenc} 
\usepackage{epsf} 
\usepackage{graphics} 
\usepackage{amsfonts,amsmath}
\usepackage[sort,numbers]{natbib} 
\usepackage{psfrag,xspace}
\usepackage{color,etoolbox}

\RequirePackage[OT1]{fontenc}
\RequirePackage{amsthm,amsmath}
\RequirePackage[sort]{natbib}
\RequirePackage[colorlinks,citecolor=blue,urlcolor=blue]{hyperref}

\usepackage{amsfonts,epsfig,graphicx}
\usepackage{afterpage}
\usepackage{amsmath,amssymb,amsthm} 
\usepackage{epsf} 
\usepackage{graphics} 
\usepackage{amsfonts,amsmath}
\usepackage{psfrag,xspace,subcaption}
\usepackage{color,etoolbox}
\usepackage[protrusion=true,expansion=true]{microtype}
\usepackage{enumitem}

\newenvironment{enum}{
\begin{enumerate}
  \setlength{\itemsep}{1pt}
  \setlength{\parskip}{0pt}
  \setlength{\parsep}{0pt}
}{\end{enumerate}}

\newcommand{\poi}{\mathrm{Poisson}}

\let\tilde\widetilde

\newcommand{\prior}{\pi}

\usepackage[ruled]{algorithm2e}

\newcommand{\tfunc}[2]{T_{#1}(#2)}

\newcommand{\nul}{p_0}
\newcommand{\distP}{P}
\newcommand{\densp}{p}
\newcommand{\ppartition}{\mathcal{P}^{\dagger}}

\newcommand{\dimension}{d}
\newcommand{\nulldens}{\ensuremath{p_0}}
\newcommand{\nulldist}{\ensuremath{P_0}}

\newcommand{\lipcons}{\ensuremath{L_n}}

\newcommand{\coninfty}{\|\psi\|_{\infty}}

\newcommand{\intconst}{c_{\text{int}}}

\newcommand{\lipclass}{\ensuremath{\mathcal{L}}}
\newcommand{\multclass}{\ensuremath{\mathcal{M}}}

\newcommand{\critrad}{\ensuremath{\epsilon}}
\newcommand{\critradn}{\ensuremath{\epsilon_n}}

\newcommand{\tv}[2]{\ensuremath{{\rm TV}(#1,#2)}}

\newcommand{\test}{\ensuremath{\phi}}
\newcommand{\domain}{\ensuremath{\mathcal{X}}}
\newcommand{\risk}{\ensuremath{R_n}}
\newcommand{\genclass}{\ensuremath{\mathcal{C}}}

\newcommand{\epstail}[1]{\mathcal{Q}_{#1}(p_0)}
\newcommand{\tailtest}{\phi_{\text{tail}}}
\newcommand{\maxtest}{\phi_{\text{max}}}

\newcommand{\bulk}[1]{\mathcal{B}_{#1}(p_0)}
\newcommand{\truncnorm}[2]{V_{#1}(#2)}
\newcommand{\trunctest}{\phi_{\text{trunc}}}
\newcommand{\truncstat}{T_{\text{trunc}}}

\newcommand{\altclass}{\mathcal{A}(\epsilon,L)}

\newcommand{\finalcon}{\omega_1}
\newcommand{\finalcontwo}{\omega_2}

\newcommand{\contwo}{8}

\newcommand{\bigN}{\widetilde{N}}
\newcommand{\smallN}{N}

\newcommand{\ellone}{\ell_1}
\newcommand{\elltwo}{\ell_2}

\newcommand{\vparam}{\sigma}
\newcommand{\alt}{\mathcal{A}}
\newcommand{\likrat}{W_n}

\newtheorem{theorem}{Theorem}
\newtheorem{newcounterthree}{New Counter}
\newtheorem{newcountertwo}{New Counter Two}
\newtheorem{lemma}[newcounterthree]{Lemma}
\newtheorem{example}[newcountertwo]{Example}

\begin{document}

\title{}

\begin{center} {\Large{\bf{Hypothesis Testing For Densities and 

\vspace{.1cm}

High-Dimensional Multinomials: Sharp Local Minimax Rates}}}
\\

\vspace*{.3in}

{\large{
\begin{tabular}{ccccc}
Sivaraman Balakrishnan$^\dagger$ && Larry Wasserman$^\dagger$ \\
\end{tabular}

\vspace*{.1in}

\begin{tabular}{ccc}
Department of Statistics$^{\dagger}$ \\
\end{tabular}

\begin{tabular}{c}
Carnegie Mellon University \\
Pittsburgh, PA 15213
\end{tabular}

\vspace*{.2in}

\begin{tabular}{c}
{\texttt{$\{$siva,larry$\}$@stat.cmu.edu}}
\end{tabular}
}}

\vspace*{.2in}

\today
\vspace*{.2in}
\begin{abstract}
We consider the goodness-of-fit testing 
problem
of distinguishing whether the data are drawn from a
specified distribution, versus a composite alternative
separated from the null in the total variation metric. In the discrete case, we
consider goodness-of-fit testing when the null distribution has a
possibly growing or unbounded number of categories.  In the
continuous case, we consider testing a Lipschitz density, with
possibly unbounded support, in the low-smoothness regime where the
Lipschitz parameter is not assumed to be constant.  In contrast to
existing results, we show that the minimax rate and critical
testing radius in these settings depend strongly, and in a precise
way, on the null distribution being tested and this motivates the study of the 
(local) minimax rate as a function of the null distribution. 
For multinomials
the local minimax rate was recently studied in the work of
\citet{valiant14}.
We re-visit and extend their results and
develop two modifications to the $\chi^2$-test whose
performance we characterize.  For testing Lipschitz densities, we
show that the usual binning tests are inadequate in the
low-smoothness regime and we design a spatially adaptive
partitioning scheme that forms the basis for our locally minimax
optimal tests. Furthermore, we provide the first local minimax lower
bounds for this problem which yield a sharp characterization of
the dependence of the critical radius on the null hypothesis being
tested. In the low-smoothness regime we also provide adaptive tests,
that adapt to the unknown smoothness parameter. We illustrate our results with a variety of
simulations that demonstrate the practical utility of our proposed
tests.
\end{abstract}
\end{center}

\section{Introduction}



Hypothesis testing is one of the pillars of modern mathematical statistics with a vast array of scientific applications. There is a well-developed theory of 
hypothesis testing starting with the work of \citet{neyman33}, and their framework plays a central role in the theory and practice of statistics.
In this paper we re-visit the classical goodness-of-fit testing problem of distinguishing the hypotheses:
\begin{align}
\label{eqn:test_main}
H_0: Z_1,\ldots, Z_n \sim P_0\ \ \ {\rm versus}\ \ \ 
H_1: Z_1,\ldots, Z_n \sim P \in \alt
\end{align}
for some set of distributions $\alt$.
This fundamental
problem has been widely studied (see for instance \cite{lehmann06} and
references therein). 

A natural choice of the composite
alternative, one that has a clear probabilistic interpretation, excludes a total variation
neighborhood around the null, i.e.  
we take
$\alt = \{P: \tv{P}{P_0}\geq \epsilon/2\}$. 
This is equivalent to
$\alt = \{P: \|P-P_0\|_1 \geq \epsilon\}$, and we use this representation in the rest of this paper.
However, there exist no consistent tests that can distinguish an arbitrary distribution $P_0$ from alternatives separated in $\ell_1$;
see \cite{lecam73,barron89}.
Hence,
we impose structural restrictions on $P_0$ and $\alt$.
We focus on two
cases: 
\begin{enum}
\item {\bf Multinomial testing: } When the null and alternate distributions are multinomials.
\item {\bf Lipschitz testing: } When the null and alternate distributions have Lipschitz densities.
\end{enum}
The problem of goodness-of-fit testing for multinomials
has a rich history in statistics and popular 
approaches are based on the
$\chi^2$-test \cite{pearson00} or the likelihood ratio test
\cite{wilks38,casella02,neyman33};
see, for instance,
\cite{fienberg79,morris75,diaconis06,paninski08,read88} and references therein.
Motivated by connections to property testing \citep{ron08}, there is also a recent literature developing in computer
science; see 
\cite{goldreich11,valiant14,batu01,diakon16}. 
Testing
Lipschitz densities is one of the basic non-parametric hypothesis
testing problems and tests are often based on the Kolmogorov-Smirnov or
Cram\'{e}r-von Mises statistics \cite{smirnov39,cramer28,von28}.
This problem 
was originally studied from the minimax perspective in the work of
Ingster \cite{ingster94,ingster03}.
See \cite{ingster03,gine15,ariascastro16}
for further references.

In the goodness-of-fit testing problem in~\eqref{eqn:test_main}, previous results
use the (global) critical radius as a benchmark.
Roughly, this global critical radius is a measure of the minimal separation between the null 
and alternate hypotheses that ensures distinguishability,
as the null hypothesis is varied over a large class of distributions (for instance over the class of distributions with Lipschitz densities or over the class of all multinomials on $d$ categories). 
Remarkably, 
as shown in the work of \citet{valiant14} for the case of multinomials and as we show in this paper for the case of Lipschitz densities, there is considerable heterogeneity in the critical radius as a function 
of the null distribution $P_0$. In other words, even within the class of Lipschitz densities, testing certain null hypotheses can be much easier than testing others. Consequently, the \emph{local minimax rate} which describes the critical radius for each individual null distribution provides a much more nuanced picture. In this paper, 
we provide (near) matching upper and lower bounds on the critical radii for Lipschitz testing as a function of the null distribution, i.e. we precisely upper and lower bound the critical radius for each individual Lipschitz null hypothesis. Our upper bounds are based on $\chi^2$-type tests, performed 
on a carefully chosen spatially adaptive binning, and highlight the fact that the standard prescriptions of choosing bins with a fixed width \cite{snedecor80} can yield sub-optimal tests. 

The distinction between local and global perspectives is reminiscent
of similar effects that arise in some 
estimation problems, for instance
in shape-constrained inference
\cite{caisinica}, in constrained least-squares problems \cite{chatterjee14} and
in classical Fisher Information-Cram\'{e}r-Rao bounds \cite{lehmann98}.

The remainder of this paper is organized as follows.
In Section~\ref{sec:background} we provide some background on the minimax
perspective on hypothesis testing, and formally describe the local and
global minimax rates. We provide a detailed discussion of the problem
of study and finally provide an overview
of our main results. In Section~\ref{sec:multinomial} we review the results of 
\cite{valiant14} and present a
new globally-minimax test for testing multinomials, as well as a
(nearly) locally-minimax test.  In Section~\ref{sec:lipschitz} we
consider the problem of testing a Lipschitz density against a total
variation neighbourhood.
We present the body of our main technical result in Section~\ref{sec:lipproof}
and defer technical aspects of this proof 
to the Appendix.
In each
of Section~\ref{sec:multinomial} and~\ref{sec:lipschitz} we present
simulation results that demonstrate the superiority of the tests we
propose and their potential practical applicability.  
In the Appendix, we also present several other results
including a brief study of limiting distributions of the test
statistics under the null, as well as tests that are adaptive to
various parameters.


\section{Background and Problem Setup}
\label{sec:background}

We begin with some basic background on hypothesis testing, the testing
risk and minimax rates, before providing a detailed treatment of some
related work.

\subsection{Hypothesis testing and minimax rates}
Our focus in this paper is on the one sample goodness-of-fit testing
problem. We observe samples $Z_1,\ldots, Z_n \in \domain$, where
$\domain \subset \mathbb{R}^{\dimension},$ which are independent and
identically distributed with distribution $\distP$.
In this context, for a fixed distribution $\nulldist$, we want to test the hypotheses:
\begin{align}
  \label{eqn:test}
\begin{split}
H_0&: \distP = \nulldist \ \ \ {\rm versus} \\
H_1&: \|P - P_0\|_1 \geq \critradn.
\end{split}
\end{align}
Throughout this paper we use $P_0$ to denote the null distribution and $P$ to denote an arbitrary alternate distribution.
Throughout the paper, we use
the total variation distance (or equivalently the $\ell_1$ distance) 
between two distributions $P$ and $Q$,
defined by
\begin{equation}
{\rm TV}(P,Q) =  \sup_A |P(A) - Q(A)|
\end{equation}
where the supremum is over all measurable sets. If $P$ and $Q$ have densities $p$ and $q$ with respect to
a common dominating measure $\nu$, then
\begin{equation}
{\rm TV}(P,Q) =  \frac{1}{2}\int |p-q| d\nu = \frac{1}{2}\|p-q\|_1 \equiv \frac{1}{2}\|P-Q\|_1.
\end{equation}
We consider
the total variation distance because it has a clear probabilistic meaning
and because it is invariant under one-to-one transformations 
\cite{devroye85}.
The $\elltwo$ metric is often easier to work with but
in the context of distribution testing its interpretation is less intuitive.
Of course, other metrics (for instance Hellinger, $\chi^2$ or Kullback-Leibler) 
can be used as well but
we focus on TV (or $\ellone$) throughout this paper.
It is well-understood \citep{barron89,lecam73} that without further restrictions
there are no uniformly consistent tests for distinguishing these hypotheses.
Consequently, we focus on two restricted variants of this problem:
\begin{enumerate}
\item Multinomial testing: In the multinomial testing problem, the domain
of the distributions is $\domain = \{1,\ldots,d\}$ 
and the distributions $\nulldist$ and $\distP$
are equivalently characterized by vectors 
$\nulldens, \densp \in \mathbb{R}^d$. Formally, we define,
\begin{align*}
\multclass = \Big\{\densp: p \in \mathbb{R}^d, 
\sum_{i=1}^d p_i = 1, \ p_i \geq 0~~\forall~i \in \{1,\ldots,d\} \Big\},
\end{align*}
and consider the multinomial testing problem of distinguishing:
\begin{align}
\label{test:mult}
H_0: \distP = \nulldist, \nulldist \in \multclass~~~~{\rm versus}~~~~
H_1: \|P - P_0\|_1 \geq \critradn, P \in \multclass.
\end{align}
In contrast to
classical ``fixed-cells'' asymptotic theory \cite{read88}, we focus on
high-dimensional multinomials where $\dimension$ can grow with, and
potentially exceed the sample size $n$.

\item Lipschitz testing: In the Lipschitz density testing problem
the set $\domain \subset \mathbb{R}^d$, and we restrict our
attention to distributions with Lipschitz densities, i.e. letting
$\nulldens$ and $\densp$ denote the densities of $\nulldist$ and
$\distP$ with respect to the Lebesgue measure, we consider the set
of densities:
\begin{align*}
\lipclass(\lipcons) = 
\Bigg\{\densp: \int_{\domain} p(x) dx = 1, p(x) \geq 0~~\forall~x, 
| \densp(x) - \densp(y)| \leq \lipcons \|x - y\|_2~~\forall~x,y \in \mathbb{R}^d \Bigg\},
\end{align*}
and consider the Lipschitz testing problem of distinguishing:
\begin{align}
\label{test:lip}
H_0: \distP = \nulldist, \nulldist \in \lipclass(\lipcons)~~~~{\rm versus}~~~~
H_1: \|P - P_0\|_1 \geq \critradn, P \in \lipclass(\lipcons).
\end{align}
We emphasize, that unlike prior work \citep{ingster94,ariascastro16,gine15} we do
not require $p_0$ to be uniform. We also do  
not restrict the domain of the densities and we consider the
low-smoothness regime where the Lipschitz parameter $\lipcons$ is
allowed to grow with the sample size.
\end{enumerate}

\vspace{0.3cm}

\noindent {\bf Hypothesis testing and risk.}
Returning to the setting
described in~\eqref{eqn:test}, we define a test $\test$ as a
Borel measurable map, $\test: \domain^n \mapsto \{0,1\}.$ For a fixed
null distribution $\nulldist$, 
we define the set of level $\alpha$ tests:
\begin{equation}
\Phi_{n,\alpha} = \Bigl\{ \phi:\ P_0^n(\phi=1)\leq \alpha\Bigr\}.
\end{equation}
The worst-case risk (type II error) of a test $\phi$
over a restricted class $\genclass$ which contains $\nulldist$
is
\begin{align*}
\risk(\test; \nulldist, \critradn, \genclass) = 
\sup \Big\{ \mathbb{E}_{\distP} [1 - \test]: \|\distP - \nulldist\|_1 \geq \critradn, \distP \in \genclass\Big\}.
\end{align*} 
The local minimax risk is\footnote{Although our proofs are explicit in their dependence on $\alpha$, we suppress 
this dependence in our notation and in our main results treating $\alpha$ as a fixed strictly positive universal constant. }:
\begin{align}
\label{eqn:risk}
\risk(\nulldist, \critradn, \genclass) = 
\inf_{\test\in \Phi_{n,\alpha}} \risk(\test; \nulldist, \critradn, \genclass).
\end{align}
It is common to 
study the minimax risk via a coarse lens by studying instead
the critical radius or the minimax
separation. The critical radius is the smallest value $\critradn$ for which a hypothesis
test has non-trivial power to distinguish $\nulldist$ from the set of
alternatives.  Formally, we define the local critical radius as:
\begin{align}
\label{eqn:critradn}
\critradn(\nulldist, \genclass) = \inf \Big\{\critrad: \risk(\nulldist, \critradn, \genclass) \leq 1/2 \Big\}.
\end{align}
The constant 1/2 is arbitrary; we could use any number in $(0,1-\alpha)$.

The local minimax risk and critical radius depend on the
null distribution $\nulldist$. 
A more common quantity of interest is 
the \emph{global} minimax risk
\begin{equation}
\label{eqn:globalcritradn}
\risk(\critradn, \genclass) = \sup_{\nulldist \in \genclass} \risk(\nulldist, \critradn, \genclass).
\end{equation}
The corresponding global critical radius is
\begin{equation}
\critradn(\genclass) = \inf \Big\{\critradn: \risk(\critradn, \genclass) \leq 1/2 \Big\}.
\end{equation}

In typical non-parametric problems, the local minimax risk  and the global minimax risk match up to
constants and this has led researchers in past work
to focus on the global minimax risk.
We show that for the distribution testing problems we consider,
the local critical radius in \eqref{eqn:critradn} can
vary considerably as a function of the null distribution
$\nulldist$. As a result, the global critical radius, provides only a
partial understanding of the intrinsic difficulty of this family
of hypothesis testing problems. In this paper, we focus on producing
tight bounds on the local minimax separation. These bounds yield as a simple
corollary, sharp bounds on the global minimax separation, but are in general
considerably more refined.

\vspace{0.3cm}

\noindent {\bf Poissonization: } 
In constructing upper bounds on the minimax risk---we work under a
simplifying assumption that the sample size is random:
$n_0 \sim \poi(n)$. 
This assumption is standard in the literature
\citep{valiant14,ariascastro16}, and simplifies several
calculations. When the sample size is chosen to be distributed as
$\poi(n)$, it is straightforward to verify that for any fixed set $A,B
\subset \mathcal{X}$ with $A \cap B = \emptyset$, under $\distP$ the
number of samples falling in $A$ and $B$ are distributed independently
as $\poi(n \distP(A))$ and $\poi(n \distP(B))$ respectively.

In the Poissonized setting, we consider the averaged minimax risk,
where we additionally average the risk in~\eqref{eqn:risk}
over the random sample size. The Poisson distribution is tightly
concentrated around its mean and this additional averaging only
affects constant factors in the minimax risk and we ignore this averaging in the
rest of the paper.

\subsection{Overview of our results}
With the basic framework in place we now provide a high-level overview of the main results of this paper.
In the context of testing multinomials, the results of \cite{valiant14} characterize the local and global minimax rates. We provide the following additional results:
\begin{itemize}
\item In Theorem~\ref{thm:truncchisq} we characterize a simple and practical globally minimax test. In Theorem~\ref{thm:max} building on the results of \cite{diakon16} we provide a simple (near) locally minimax test. 
\end{itemize}

In the context of testing Lipschitz densities 
we make advances over classical results \cite{ingster03,gine15}
by eliminating several unnecessary assumptions (uniform null, bounded support, fixed Lipschitz parameter). We provide the first characterization of the local minimax rate for this problem.
In studying the Lipschitz testing problem in its full generality we find that the critical testing radius 
can exhibit a wide range of possible behaviours, based roughly on the tail behaviour of the null hypothesis.

\begin{itemize}
\item In Theorem~\ref{thm:main} we provide a characterization of the local minimax rate for Lipschitz density testing. In Section~\ref{sec:ex}, we consider a variety of concrete examples that demonstrate the rich scaling behaviour exhibited by the critical radius in this problem.
\item Our upper and lower bounds are based on a novel spatially 
adaptive partitioning scheme. We describe this scheme and derive some of its useful properties in Section~\ref{sec:part}.
\end{itemize}

In the Supplementary Material we provide the technical details of the proofs. 
We briefly consider the limiting behaviour of our test statistics under the null in Appendix~\ref{app:limit}.
Our results show that the critical radius is determined by a certain functional of the null hypothesis. In Appendix~\ref{app:tfunc} we study certain important properties of this functional pertaining to its stability.
Finally, we study tests which are adaptive to various parameters in Appendix~\ref{app:adapt}.

\section{Testing high-dimensional multinomials}
\label{sec:multinomial} 

Given a sample $Z_1,\ldots, Z_n \sim P$
define the counts
$X=(X_1,\ldots, X_d)$ 
where $X_j = \sum_{i=1}^n I(Z_i =j)$.
The local minimax critical radii for the multinomial problem
have been found in 
\citet{valiant14}.
We begin by summarizing these results.

Without loss of generality 
we assume that the entries of the null multinomial $p_0$ 
are sorted so that $p_0(1) \geq p_0(2) \geq \ldots 
\geq p_0(d)$. For any $0 \leq \vparam \leq 1$ we denote $\vparam$-tail of the
multinomial by:
\begin{align}
\label{eqn:taildef}
\epstail{\vparam} = \left\{i: \sum_{j = i}^d p_0(j) \leq \vparam \right\}.
\end{align}
The $\vparam$-bulk is defined to be
\begin{align}
\label{eqn:bulkdef}
\bulk{\vparam} = \{i>1:\ i\notin \epstail{\vparam}\}.
\end{align}
Note that $i=1$ is excluded from the $\vparam$-bulk.
The minimax rate depends on the functional:
\begin{align}
\label{eqn:vfunc}
\truncnorm{\vparam}{p_0} = 
\left( \sum_{i \in \bulk{\vparam}} p_0(i)^{2/3} \right)^{3/2}.
\end{align}
For a given multinomial $p_0$, our goal is to upper and lower bound the 
local 
critical
radius $\critradn(\nulldens, \multclass)$ in \eqref{eqn:critradn}. 
We define, $\ell_n$ and $u_n$ to be the solutions to the equations \footnote{These equations always have a unique solution since the right hand side monotonically decreases to $0$ as the left hand side monotonically increases from 0 to 1.}:
\begin{align}
\label{eqn:multrad}
\ell_n(p_0) = \max\left\{ \frac{1}{n}, \sqrt{\frac{ \truncnorm{\ell_n(p_0)}{p_0}}{n}} \right\},\ \ \ 
u_n(p_0) = \max\left\{\frac{1}{n}, \sqrt{\frac{ \truncnorm{u_n(p_0)/16}{p_0}}{n}}\right\}.
\end{align}
With these definitions in place, we are now ready to state the result of \cite{valiant14}. We use $c_1,c_2,C_1,C_2 > 0$ to denote positive universal constants. 
\begin{theorem}[\cite{valiant14}]
\label{thm:val-val}
The local critical radius $\critradn(\nulldens,\multclass)$ 
for multinomial testing is upper and lower bounded as:
\begin{align}
\label{eqn:minimaxcritrad}
c_1 \ell_n(p_0) \leq \critradn(\nulldens,\multclass) \leq C_1 u_n(p_0).
\end{align}
Furthermore, the global critical radius $\critradn(\multclass)$ is bounded as:
\begin{align*}
\frac{c_2 d^{1/4}}{\sqrt{n}} \leq \critradn(\multclass) \leq  \frac{C_2 d^{1/4}}{\sqrt{n}}.
\end{align*}
\end{theorem}
\noindent {\bf Remarks: } 
\begin{itemize}
\item The local critical radius is roughly determined by the (truncated) 2/3-rd norm of the multinomial $p_0$. This norm is maximized when $p_0$ is uniform and is small when $p_0$ is sparse, and at a high-level captures the ``effective sparsity'' of $p_0$.

\item The global critical radius can shrink to zero even when $d \gg n$. 
When $d \asymp \sqrt{n}$ almost all categories of the multinomial are unobserved
but it is still possible to reliably distinguish any $p_0$ from an $\ellone$-neighborhood. This phenomenon 
is noted for instance in the work of \cite{paninski08}.
We also note the work of \citet{barron89} that shows that when $d = \omega(n)$, no test can have power that approaches $1$ at an exponential rate.

\item The local critical radius can be much smaller than the global minimax radius. If the multinomial $p_0$ is nearly (or exactly) $s$-sparse then the critical radius is upper and lower bounded up to constants by $s^{1/4}/\sqrt{n}$. Furthermore, these results also show that it is
possible to design consistent tests for sufficiently structured null hypotheses:
in cases when 
$\sqrt{d} \gg n,$ and even in cases when $d$ is infinite.

\item Except for certain 
pathological multinomials, the upper and lower critical radii match up to constants. We revisit this issue in Appendix~\ref{app:tfunc}, in the context of Lipschitz densities, where we present examples where the solutions to critical equations similar to~\eqref{eqn:multrad} are stable and examples where they are unstable.

\end{itemize}

\noindent 
In the remainder of this section we consider a variety of tests, including the test presented
in \cite{valiant14} and several alternatives. 
The test of \cite{valiant14} is a composite
test that requires knowledge of $\critradn$ and the analysis of their test is quite intricate.
We present an alternative, simple test that is globally minimax, and then present an alternative
composite test that is locally minimax but simpler to analyze. Finally, we present a few illustrative simulations.

\subsection{The truncated $\chi^2$ test}
We begin with a simple globally minimax test. 
From a practical standpoint, the most popular test for multinomials is Pearson's $\chi^2$ test. However, in the high-dimensional regime where the dimension of the multinomial 
$d$ is not treated as fixed the $\chi^2$ test can have bad power due to the fact that 
the variance of the $\chi^2$ statistic is dominated by small entries of the multinomial
(see \cite{valiant14,mariott15}).

 A natural
thought then is to truncate the normalization factors of
the $\chi^2$ statistic in order to limit the contribution to the
variance from each cell of the multinomial. 
Recalling that $(X_1,\ldots,X_d)$ denote the observed counts, 
we propose
the test statistic:
\begin{align}
\label{eqn:truncstat}
\truncstat = \sum_{i=1}^d \frac{(X_i - np_0(i))^2 - X_i}{ \max\{1/d, p_0(i) \}} :=  
\sum_{i=1}^d \frac{(X_i - np_0(i))^2 - X_i}{ \theta_i}
\end{align}
and the corresponding test,
\begin{align}
\label{eqn:trunctest}
\trunctest = \mathbb{I}\left(\truncstat >  n \sqrt{ \frac{2}{\alpha} \sum_{i=1}^d \frac{ p_0(i)^2 }{ \theta_i^2}}  \right).
\end{align}
This test statistic truncates the 
usual normalization factor for the $\chi^2$ test for any entry which falls below $1/d$, and thus 
ensures that very small entries in $p_0$ do not have a large effect on the variance of the statistic.
We emphasize the simplicity and practicality of this test.
We have the following result which bounds the power and size of the truncated $\chi^2$ test. We use
$C > 0$ to denote a positive universal constant. 
\begin{theorem}
\label{thm:truncchisq}
Consider the testing problem in~\eqref{test:mult}. The truncated $\chi^2$ test has size at most $\alpha$, i.e.
$P_0(\trunctest = 1) \leq \alpha.$ Furthermore, there is a universal constant $C > 0$ such that if for any $0 < \zeta \leq 1$ we have that,
\begin{align}
\label{eqn:trunccritrad}
\critradn^2  \geq \frac{C \sqrt{d}}{n} \left[ \frac{1}{\sqrt{\alpha}}+\frac{1}{\zeta} \right], 
\end{align}
then the Type II error of the test is bounded by $\zeta$, i.e. $P(\trunctest = 0) \leq \zeta.$
\end{theorem}

%

\noindent {\bf Remarks:}
\begin{itemize}
\item A straightforward consequence of this result together with the result in Theorem~\ref{thm:val-val}
is that the truncated $\chi^2$ test is globally minimax optimal. 
\item The classical $\chi^2$ and likelihood ratio tests are not generally consistent (and thus not globally minimax optimal) in the high-dimensional regime (see also, Figure~\ref{fig:power}).
\item At a high-level the proof follows by verifying that when the alternate hypothesis is true,
under the condition on the critical
radius in~\eqref{eqn:trunccritrad}, the test statistic is larger than the threshold in~\eqref{eqn:trunctest}.
To verify this, we lower bound the mean and upper bound the variance of the test statistic under the alternate and then use standard concentration results. We defer the details to the Supplementary Material (Appendix~\ref{app:mult}).
\end{itemize}

\subsection{The $2/3$-rd + tail test}
The truncated $\chi^2$ test described in the previous section, although globally minimax, is not locally adaptive. The test from \cite{valiant14}, achieves the local minimax upper bound in Theorem~\ref{thm:val-val}. We refer to this as the $2/3$-rd + tail test. We use a slightly modified version of their test when testing Lipschitz goodness-of-fit in Section~\ref{sec:lipschitz}, and provide a description here.

The test is a composite two-stage test, and has a tuning parameter $\sigma$. Recalling the definitions
of $\bulk{\vparam}$ and $\epstail{\vparam}$ (see~\eqref{eqn:taildef}), we define two test statistics $T_1,T_2$ 
and corresponding test thresholds $t_1,t_2$:
\begin{align*}
T_1(\sigma) &= \sum_{j \in \epstail{\sigma}} (X_j - np_0(j)),~~~~~~~~~~~~~~~t_{1}(\alpha,\sigma) =  \sqrt{ \frac{n P_0(\epstail{\sigma}) }{\alpha}}, \\
T_2(\sigma) &= \sum_{j\in \bulk{\sigma}} \frac{(X_j - n p_0(j))^2 - X_j}{p_j^{2/3}},~~~~~t_{2}(\alpha,\sigma) = \sqrt{ \frac{\sum_{j \in {\cal B}_\sigma} 2n^2p_0(j)^{2/3}}{\alpha}}. 
\end{align*}
%
We define two tests: 
\begin{enumerate}
\item The tail test: $\tailtest(\sigma,\alpha) = \mathbb{I}(T_1(\sigma) > t_{1}( \alpha,\sigma)).$
\item The 2/3-test: $\phi_{2/3}(\sigma,\alpha) = \mathbb{I}(T_2(\sigma) > t_2(\alpha,\sigma))$.
\end{enumerate}
The composite test $\phi_V(\sigma,\alpha)$ is then given as:
\begin{align}
\label{eqn:valtest}
 \phi_V(\sigma,\alpha) &= \max\{\tailtest(\sigma,\alpha/2), \phi_{2/3}(\sigma,\alpha/2)\}.
\end{align}
With these definitions in place, the following result is essentially from
the work of \cite{valiant14}. We use $C > 0$ to denote a positive universal constant.
\begin{theorem}
\label{thm:valup}
Consider the testing problem in~\eqref{test:mult}. The composite test $\phi_V(\sigma,\alpha)$ 
has size at most $\alpha$, i.e.
$P_0(\phi_V = 1) \leq \alpha.$ Furthermore, 
if
we choose $\sigma = \critradn(p_0,\mathcal{M})/8$, and $u_n(p_0)$ as in~\eqref{eqn:multrad}, then 
for any $0 < \zeta \leq 1$, if
\begin{align}
\critradn(p_0, \mathcal{M}) \geq C u_n(p_0) \max\{1/\alpha, 1/\zeta\},  
\end{align}
then the Type II error of the test is bounded by $\zeta$, i.e. $P(\phi_V = 0) \leq \zeta.$
\end{theorem}

\noindent{\bf Remarks:} 
\begin{itemize}
\item The test $\phi_V$ is also motivated by deficiencies of the $\chi^2$ test. In particular, the test includes two main modifications to the $\chi^2$ test which limit the contribution of the small entries of $p_0$: some of the small entries of $p_0$ are dealt with via a separate tail test and further the normalization of the $\chi^2$ test is changed from $p_0(i)$ 
to $p_0(i)^{2/3}$. 
\item This result provides the upper bound of Theorem~\ref{thm:val-val}. It requires that the tuning parameter $\sigma$ is chosen as $\critradn(p_0,\mathcal{M})/8$. 
In the Supplementary Material (Appendix~\ref{app:adapt}) we discuss adaptive choices for $\sigma$.
\item The proof essentially follows from the paper of \cite{valiant14}, but we maintain an explicit bound on the power and size of the test, which we use in later sections. We provide the details in Appendix~\ref{app:mult}. 
\end{itemize}
\noindent While the 2/3-rd norm test is locally minimax optimal its analysis is quite challenging. In the next section, we build on results from a recent paper of \citet{diakon16} to provide an alternative (nearly) locally minimax test with a simpler analysis.

\subsection{The Max Test}
An important insight, one that is seen for instance in Figure~\ref{fig:unif}, is that many simple tests 
are optimal when $p_0$ is uniform and that careful modifications to the $\chi^2$ test 
are required only when $p_0$ is far from uniform.
This suggests  the following strategy: partition the multinomial into nearly uniform groups,
apply a simple test within each group and combine the tests with an appropriate Bonferroni correction.
We refer to this as the max test. Such a strategy was used by \citet{diakon16}, but their test 
is quite complicated and involves many constants.
Furthermore,
it is not clear how to ensure that their test controls the Type I error at level $\alpha$.
Motivated by their approach,
we present a simple test that controls the type I error as required and is (nearly)
locally minimax.

As with the test in the previous section,
the test has to be combined with the tail test.
The test is defined to be
\begin{align*}
\phi_{\max}(\sigma, \alpha) = \max\{\phi_{\rm tail}(\sigma,\alpha/2) ,\phi_{M}(\sigma,\alpha/2)\},
\end{align*}
where 
$\phi_{M}$ is defined as follows.
We partition $\bulk{\sigma}$ into sets $S_j$ for $j \geq 1$, where
\begin{align*}
S_j = \Bigl\{ t:\ \frac{p_0(2)}{2^{j}} < p_0(t) \leq \frac{p_0(2)}{2^{j-1}} \Bigr\}.
\end{align*}
We can bound the total number of sets $S_j$ by noting that for any $i \in \bulk{\sigma}$, we have
that $p_0(i) \geq \sigma/d$, so that the number of sets $k$ is bounded by $\lceil\log_2 (d/\sigma)\rceil$. Within each set we use a modified $\ell_2$ statistic.
Let
\begin{equation}
\label{eqn:elltwo}
T_j = \sum_{t\in S_j}[(X_t - n p_0(t))^2 - X_t]
\end{equation}
for $j \geq 1$. Unlike the traditional $\ell_2$ statistic, each term 
in this statistic is centered around $X_t$.
As observed in \cite{valiant14}, this results in
the statistic having smaller variance in the $n \ll d$ regime.
Let
\begin{equation}
\label{eqn:maxtest}
\phi_M(\sigma,\alpha) = \bigvee_j \mathbb{I}(T_j > t_j),
\end{equation}
where
\begin{align}
\label{eqn:thresh}
t_j =   \sqrt{ \frac{2k n^2 \left[ \sum_{t \in S_j} p_0(t)^2 \right]  }{\alpha}}.
\end{align}
\begin{theorem}
\label{thm:max}
Consider the testing problem in~\eqref{test:mult}. Suppose we choose $\sigma = \critradn(p_0,\mathcal{M})/8$, then the composite test $\maxtest(\sigma,\alpha)$ 
has size at most $\alpha$, i.e.
$P_0(\maxtest = 1) \leq \alpha.$ Furthermore, there is a universal constant $C > 0$, such that 
for $u_n(p_0)$ as in~\eqref{eqn:multrad}, if for any $0 < \zeta \leq 1$ we have that,
\begin{align}
\label{eqn:maxcritrad}
\critradn(p_0,\mathcal{M})  \geq C k^2 u_n(p_0) \max\left\{\frac{\sqrt{k}}{\alpha}, \frac{1}{\zeta}\right\},  
\end{align}
then the Type II error of the test is bounded by $\zeta$, i.e. $P(\maxtest = 0) \leq \zeta.$
\end{theorem}
\noindent {\bf Remarks: }
\begin{itemize}
\item Comparing the critical radii in Equations~\eqref{eqn:maxcritrad} and~\eqref{eqn:minimaxcritrad}, and noting that $k \leq \lceil\log_2 (8d/\critradn)\rceil,$ we conclude that the max test is 
locally minimax optimal, up to a logarithmic factor.
\item In contrast to the analysis of the 2/3-rd + tail test in \cite{valiant14},
the analysis of the max test involves mostly elementary calculations. We provide the details in Appendix~\ref{app:mult}.
As emphasized in the work of~\cite{diakon16}, the 
reduction of testing problems to simpler testing problems (in this case, testing uniformity) 
is a more broadly useful idea. Our upper bound for the Lipschitz testing problem (in Section~\ref{sec:lipschitz}) proceeds
by reducing it to a multinomial testing problem through a spatially adaptive binning scheme.
\end{itemize}

\subsection{Simulations}
In this section, we report some simulation results that demonstrate
the practicality of the proposed tests. We focus on two simulation
scenarios and compare the globally-minimax truncated $\chi^2$ test,
and the 2/3rd + tail test \cite{valiant14} with the
classical $\chi^2$-test and the likelihood ratio test.
The $\chi^2$ statistic is,
\begin{align*}
T_{\chi^2} = \sum_{i=1}^d \frac{  (X_i - np_0(i))^2 - np_0(i)}{np_0(i)},
\end{align*}
and the likelihood ratio test statistic is
\begin{align*}
T_{\text{LRT}}= 2 \sum_{i=1}^d X_i \log \left(\frac{X_i}{np_0(i)} \right).
\end{align*}
In Appendix~\ref{app:extra_sims}, we consider 
a few additional simulations as well as a comparison with statistics based on the
$\ell_1$ and $\ell_2$ distances.

In each setting described below, we set the $\alpha$ level threshold via simulation (by sampling from the null 1000 times) and we calculate the power under particular alternatives by averaging over a 1000 trials. 
\begin{center}
\begin{figure}
\begin{tabular}{cc}
\includegraphics[scale=0.4]{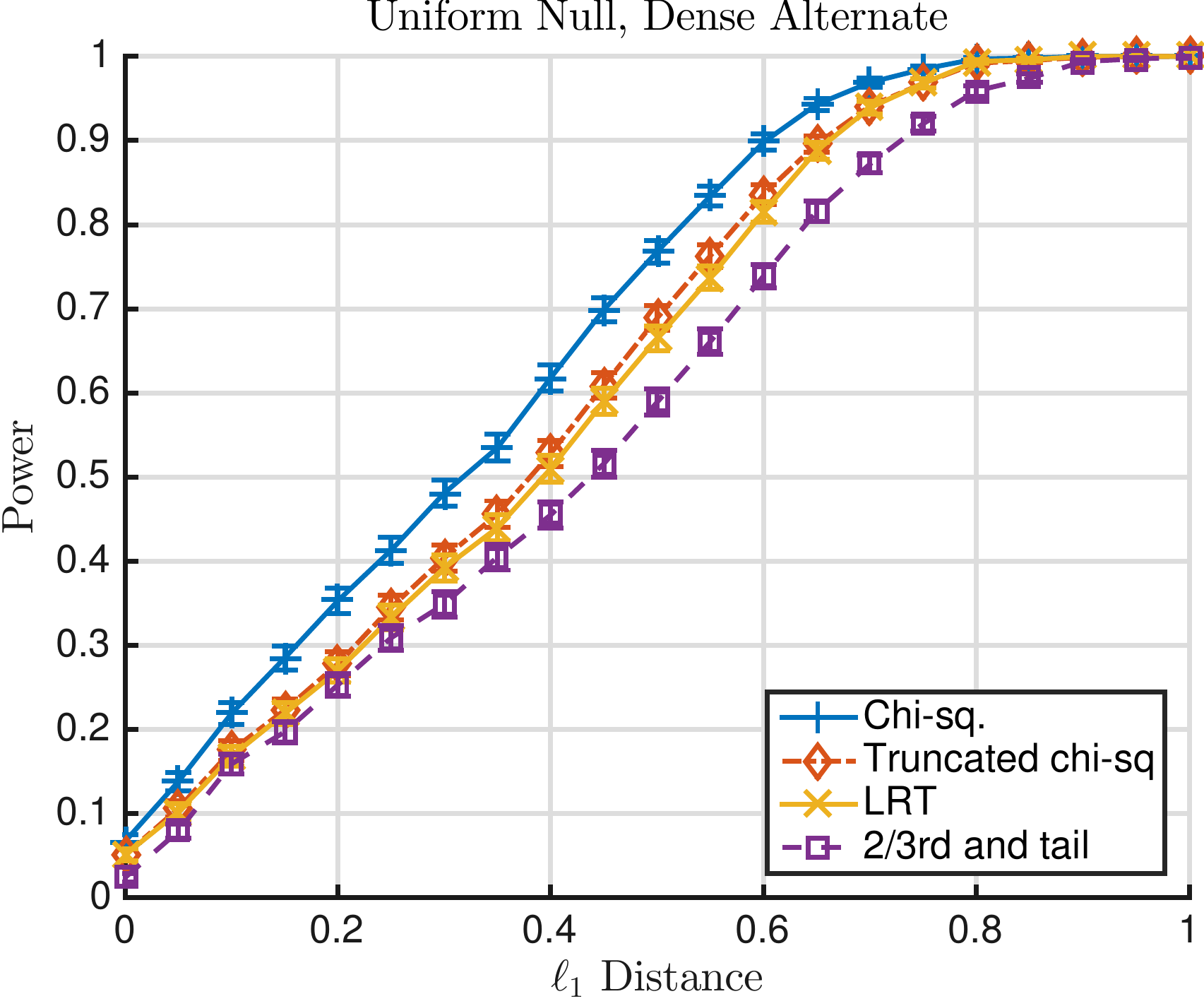} & ~~~~~\includegraphics[scale=0.4]{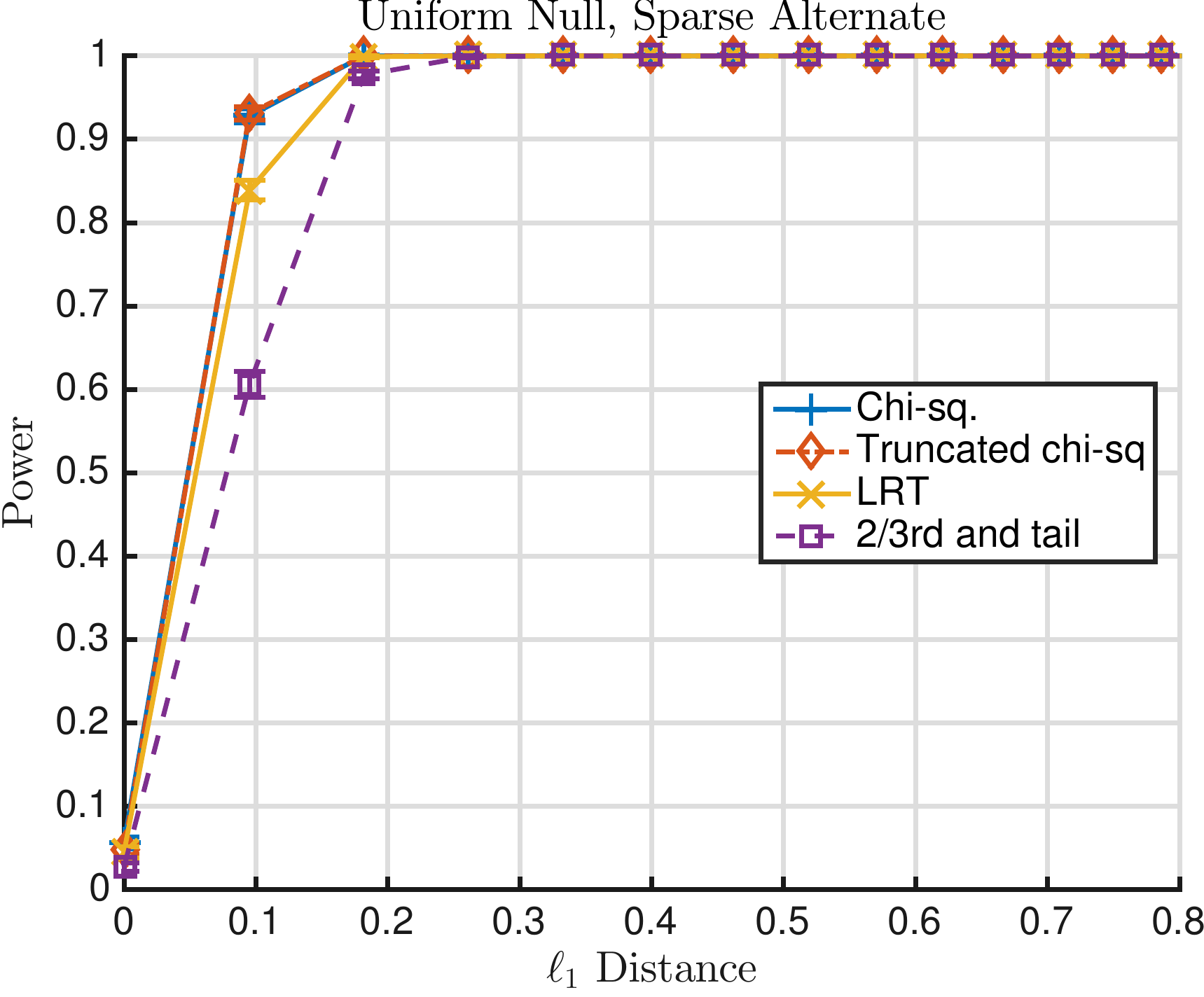}
\end{tabular}
\caption{A comparison between the truncated $\chi^2$ test, the 2/3rd + tail test \cite{valiant14}, 
the $\chi^2$-test and the likelihood ratio test. The null is chosen to be uniform, and the alternate is either a dense or sparse perturbation of the null. The power of the tests are plotted against the $\ell_1$ distance between the null and alternate. Each point in the graph is an average over 1000 trials. Despite the high-dimensionality (i.e. $n = 200, d = 2000$) the tests have high-power, and perform comparably.}
\label{fig:unif}
\end{figure}
\end{center}
\begin{enumerate}
\item Figure~\ref{fig:unif} considers a high-dimensional setting where 
$n = 200, d = 2000$, the null distribution is uniform, and the alternate is either dense (perturbing each coordinate by a scaled Rademacher) or sparse (perturbing only two coordinates).

In each case we observe that all the tests perform comparably
indicating that a variety of tests are optimal around the uniform
distribution, a fact that we exploit in designing the max test. The test from \cite{valiant14} performs slightly
worse than others due to the Bonferroni correction from applying a
two-stage test.

\item Figure~\ref{fig:power} considers a power-law null where $p_0(i) \propto 1/i$. 
Again we take $n = 200, d = 2000$, and compare against
a dense and sparse alternative. In this setting, we choose the
sparse alternative to only perturb the first two coordinates of the
distribution.

We observe two notable effects. First, we see that when the
alternate is dense, the truncated $\chi^2$ test, although consistent
in the high-dimensional regime, is outperformed by the other tests
highlighting the need to study the local-minimax properties of
tests. Perhaps more surprising is that in the setting where the
alternate is sparse, the classical $\chi^2$ and likelihood ratio
tests can fail dramatically.
\end{enumerate}

\noindent 
The locally minimax test is remarkably robust
across simulation settings. However, it requires that we specify $\critradn$,
a drawback shared by the max test.  In
Appendix~\ref{app:adapt} we provide adaptive alternatives that
adapt to unknown $\critradn$.

\begin{center}
\begin{figure}
\begin{tabular}{cc}
\includegraphics[scale=0.4]{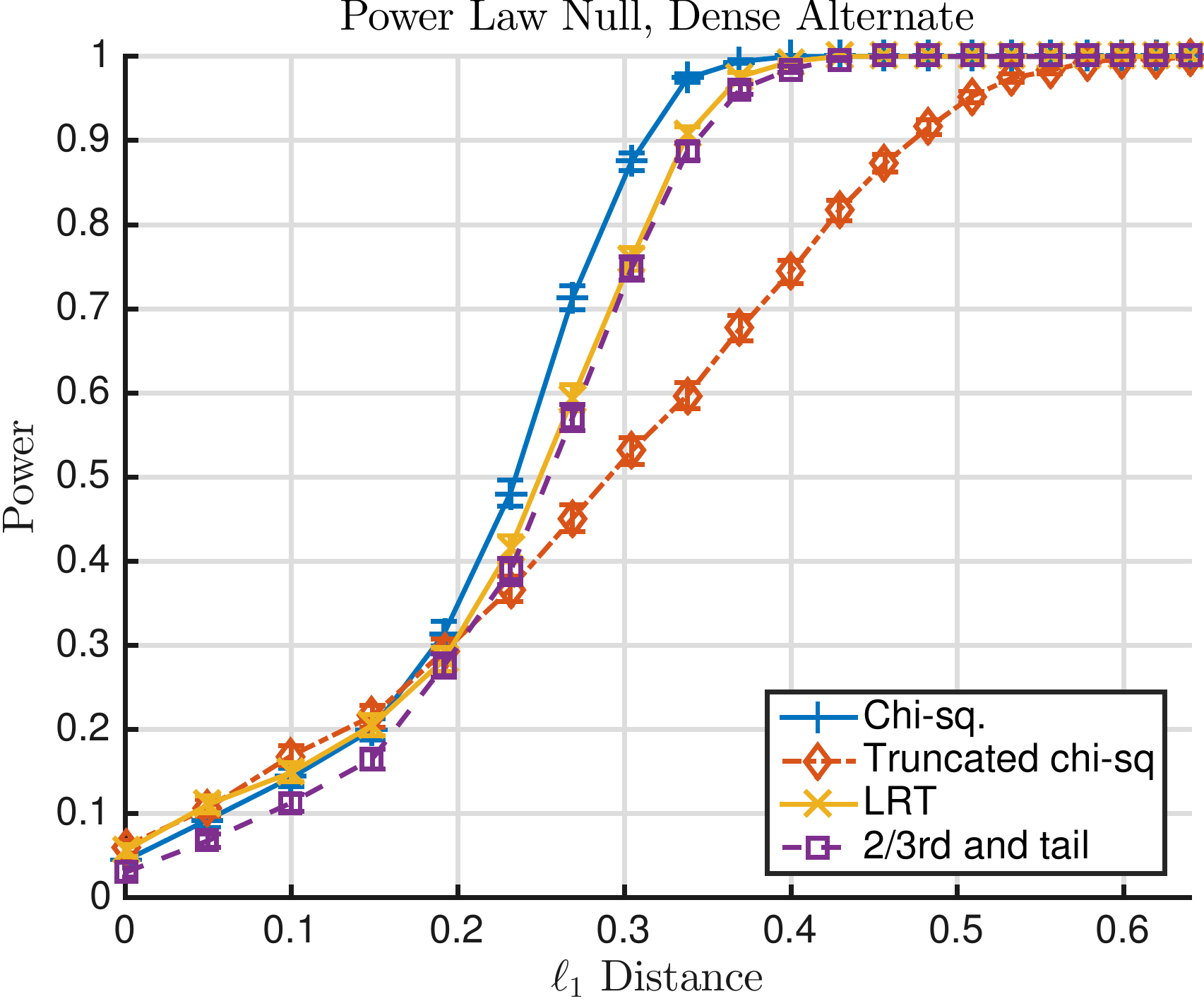} &  ~~~~~\includegraphics[scale=0.4]{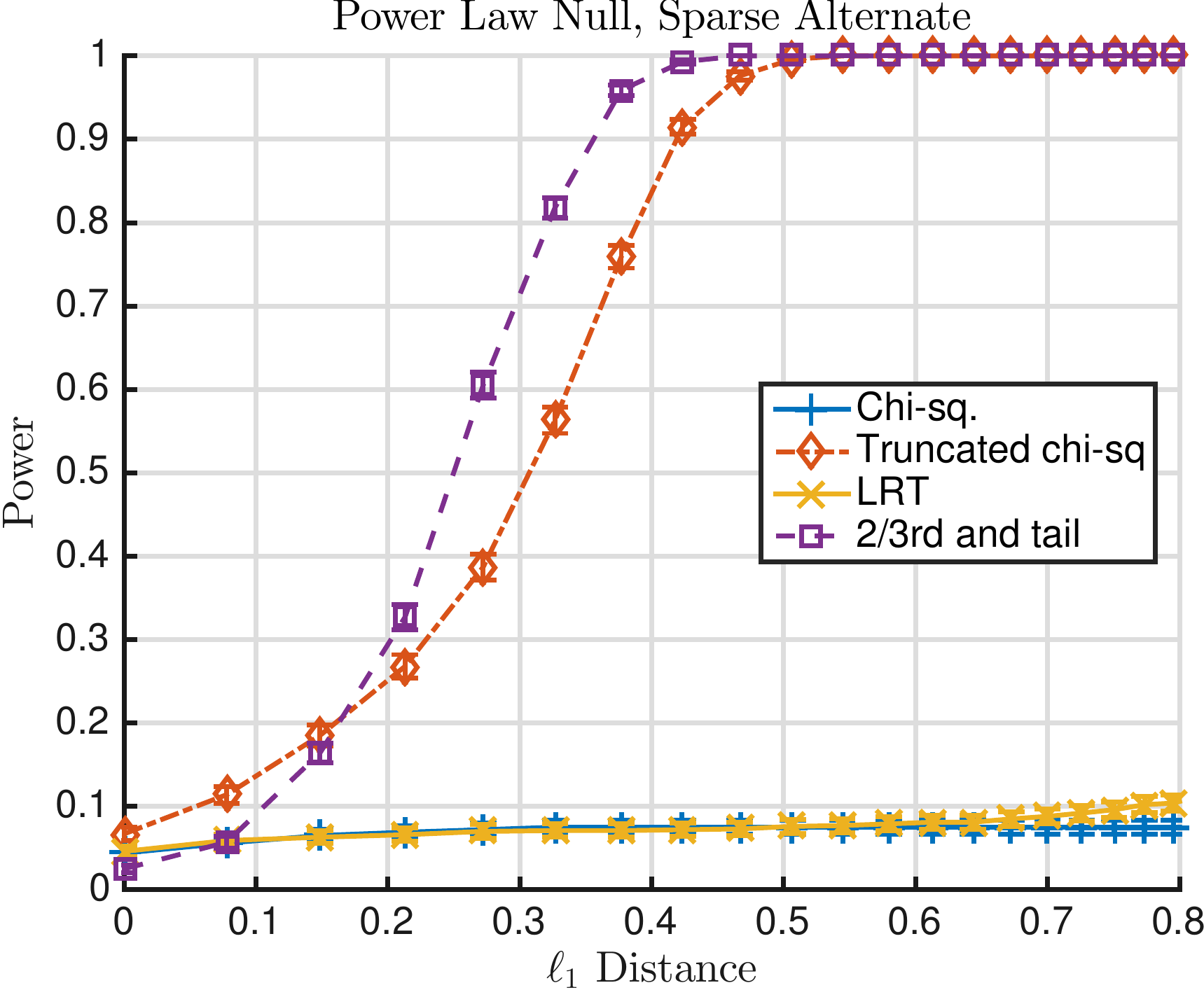}
\end{tabular}
\caption{A comparison between the truncated $\chi^2$ test, the 2/3rd + tail test \cite{valiant14}, 
the $\chi^2$-test and the likelihood ratio test. The null is chosen to be a power law, and the alternate is either a dense or sparse perturbation of the null. The power of the tests are plotted against the $\ell_1$ distance between the null and alternate. Each point in the graph is an average over 1000 trials. The truncated $\chi^2$ test despite being globally minimax optimal, can perform poorly for any particular fixed null. The $\chi^2$ and likelihood ratio tests can fail to be consistent even when $\critradn$ is quite large, and $n \gg \sqrt{d}$. }
\label{fig:power}
\end{figure}
\end{center}

\section{Testing Lipschitz Densities}
\label{sec:lipschitz}
In this section, we focus our attention on the Lipschitz testing problem~\eqref{test:lip}. 
As is standard in non-parametric problems, throughout this section, we treat the dimension $d$ as a fixed (universal) constant.  Our emphasis is on understanding the local critical radius while making minimal assumptions.
In contrast to past work, we do not assume that the null is uniform or even that its support is compact. 
We would like to be able to detect more subtle deviations from the null as the sample size gets large and hence we do not assume that the Lipschitz parameter $\lipcons$ is fixed as $n$ grows.

 The classical method, due to \citet{ingster94,ingster1997adaptive} to constructing lower and upper bounds on the critical radius, is based on binning the domain of the density. In particular, upper bounds were obtained by considering $\chi^2$ tests applied to the multinomial that results from binning the null distribution. 
 Ingster focused on the case when the null distribution $P_0$ was taken to be uniform on $[0,1]$, noting that the testing
problem for a general null distribution could be ``reduced'' to testing uniformity 
by modifying the observations via the quantile transformation corresponding to the null distribution $P_0$ (see also \cite{gine15}). We emphasize that such a reduction \emph{alters the smoothness class} tailoring it to the null distribution $P_0$. The quantile transformation forces the deviations from the null distribution to be more smooth in regions where $P_0$ is small and less smooth where $P_0$ is large, i.e. we need to re-interpret smoothness of the alternative density $p$
as an assumption about the function $p(F_0^{-1}(t)),$ where $F_0^{-1}$ is the quantile function of the null distribution $P_0$. We find this assumption to be unnatural and instead aim to directly test the hypotheses in~\eqref{test:lip}.

We begin with some high-level intuition for our upper and lower bounds.
\begin{itemize}
\item {\bf Upper bounding the critical radius: } 
The strategy of binning domain of $p_0$, and then testing the resulting multinomial against an appropriate $\ell_1$ neighborhood using a locally minimax test is natural even when $p_0$ is not uniform. However, there is considerable flexibility in how precisely to bin the domain of $p_0$. Essentially, the only constraint in the choice of bin-widths is that the approximation error (of approximating the density by its piecewise constant, histogram approximation) is sufficiently well-controlled.
When the null is not uniform the choice of fixed bin-widths is arbitrary and as we will see, sub-optimal. A bulk of the technical effort in constructing our optimal tests is then in determining the optimal inhomogenous, spatially adaptive partition of the domain in order to apply a multinomial test.
\item {\bf Lower bounding the critical radius: } At a high-level the construction of Ingster is similar to standard lower bounds in non-parametric problems. Roughly, we create a collection of possible alternate densities, 
by evenly partitioning the domain of $p_0$, and then perturbing each cell of the partition by adding or subtracting a small (sufficiently smooth) bump.  We then analyze the optimal likelihood ratio test for the (simple versus simple) testing problem of distinguishing $p_0$ from a uniform mixture of the set of possible alternate densities. 
We observe that when $p_0$ is not uniform once again creating a fixed bin-width partition is not optimal. The optimal choice of bin-widths is to choose larger bin-widths when $p_0$ is large and smaller bin-widths when $p_0$ is small. Intuitively, this choice allows us to perturb the null distribution $p_0$ more when the density is large, without violating the constraint that the alternate distribution remain sufficiently smooth. Once again, we create an inhomogenous, spatially adaptive partition of the domain, and then use this partition to construct the optimal perturbation of the null.
\end{itemize}
Define,
\begin{align}
\gamma := \frac{2}{3+d},
\end{align}
and for any $0 \leq \sigma \leq 1$ denote the collection of sets of probability mass at least $1 - \sigma$ as $\mathcal{B}_{\sigma}$, i.e. ${\cal B}_{\sigma} := \{ B:\ P_0(B) \geq 1-\sigma \}.$
Define the functional, 
\begin{align}
\label{eqn:tfuncdef}
T_{\sigma}(p_0) := \inf_{B \in {\cal B}_\sigma}  \left(
\int_{B} p_0^\gamma(x) dx\right)^{1/\gamma}. 
\end{align}
We refer to this as the truncated $T$-functional\footnote{Although the set $B$ that achieves the minimum in the definition 
of $T_{\sigma}(p_0)$ need not be unique, the functional itself is well-defined.}.
The functional $T_{\sigma}(p_0)$ is the analog of the functional $V_{\sigma}(p_0)$ in~\eqref{eqn:vfunc}, and roughly characterizes the local critical radius. We return to study this functional in light of several examples, in Section~\ref{sec:ex} (and Appendix~\ref{app:tfunc}).

In constructing lower bounds 
we will assume that the null density lies in the interior of 
the Lipschitz ball, i.e. we assume that for some constant $0 \leq \intconst < 1$, we 
have that, $p_0 \in \lipclass(\intconst \lipcons).$
This assumption avoids certain technical issues that arise in creating 
perturbations of the null density
when it lies on the boundary of the Lipschitz ball.

Finally, we define for two universal constants $C \geq c > 0$ (that are explicit in our proofs) the upper and lower critical radii:
\begin{align}
\label{eqn:liprad}
v_n(p_0) = \Big(  \frac{ \lipcons^{d/2} \tfunc{C v_n(p_0)}{p_0}}{n} \Big)^{\frac{2}{4+d}},\ \
w_n(p_0) = \Big(  \frac{ \lipcons^{d/2}\tfunc{c w_n(p_0)}{p_0}}{n} \Big)^{\frac{2}{4+d}}.
\end{align}
With these preliminaries in place we now state our main result on testing Lipschitz densities. We let
$c, C > 0$ denote two positive universal constants (different from the ones above). 
\begin{theorem}
\label{thm:main}
The local critical radius $\critradn(\nulldens,\lipclass(\lipcons))$ 
for testing Lipschitz densities is upper bounded as:
\begin{align}
\label{eqn:minimaxcritradlip}
\critradn(\nulldens,\lipclass(\lipcons)) \leq C w_n(p_0).
\end{align}
Furthermore, if for some constant $0 \leq \intconst < 1$ we 
have that, $p_0 \in \lipclass(\intconst \lipcons),$ then the critical radius is lower bounded as
\begin{align}
\label{eqn:minimaxcritradlipup}
c v_n(p_0) \leq \critradn(\nulldens,\lipclass(\lipcons)).
\end{align}
\end{theorem}
\noindent {\bf Remarks: }
\begin{itemize}
\item A natural question of interest is to understand the worst-case rate for the critical radius, or
equivalently to understand the largest that the $T$-functional can be. Since the $T$-functional can be infinite if the support is unrestricted, we restrict our attention to Lipschitz densities with a bounded support $S$. In this case, letting $\mu(S)$ denote the Lebesgue measure of $S$ and using H\"{o}lder's inequality (see Appendix~\ref{app:tfunc}) we have that for any $\sigma > 0$,
\begin{align}
\label{eqn:claim}
T_{\sigma}(p_0) \leq (1 - \sigma) \mu(S)^{\frac{1-\gamma}{\gamma}}.
\end{align}
Up to constants involving $\gamma, \sigma$ this is attained when $p_0$ is uniform on the set $S$. In other words, the critical radius is maximal for testing the uniform density against a Lipschitz, $\ell_1$ neighborhood. In this case, we simply recover a generalization of the result of \cite{ingster94} for testing when $p_0$ is uniform on $[0,1]$.
\item The main discrepancy between the upper and lower bounds is in the truncation level, i.e. the upper and lower bounds depend on the functional $T_{\sigma}(p_0)$ for different values of the parameter $\sigma$. This is identical to the situation in Theorem~\ref{thm:val-val} for testing multinomials.
In most non-pathological examples this functional is stable with respect to constant factor discrepancies in the truncation level and consequently our upper and lower bounds are typically tight (see the examples in Section~\ref{sec:ex}). In the Supplementary Material (see Appendix~\ref{app:tfunc}) we formally study the stability of the $T$-functional. We provide general bounds and relate the stability of the $T$-functional to the stability of the level-sets of $p_0$.
\end{itemize}

The remainder of this section is organized as follows: we first consider various examples, calculate the $T$-functional and develop the consequences of Theorem~\ref{thm:main} for these examples. We then turn our attention to our adaptive binning, describing both a recursive partitioning algorithm for constructing it as well as developing some of its useful properties. Finally, we provide the body 
of our proof of Theorem~\ref{thm:main} and defer more technical aspects to the Supplementary Material. We conclude with a few illustrative simulations.

\subsection{Examples}
\label{sec:ex}
The result in Theorem~\ref{thm:main} provides a general characterization of the critical radius for testing any density $p_0$, against a Lipschitz, $\ell_1$ neighborhood. In this section we consider several concrete examples.
Although our theorem is more generally applicable, for ease of exposition we focus on the setting where $d = 1$, highlighting the variability of the $T$-functional and consequently of the critical radius as the null density is changed. Our examples have straightforward $d$-dimensional extensions.

When $d = 1$, we have that $\gamma = 1/2$ so the $T$-functional is simply:
\begin{align*}
T_{\sigma}(p_0) = \inf_{B \in {\cal B}_\sigma}  \left(
\int_{B} \sqrt{p_0(x)} dx\right)^{2},
\end{align*}
where $\mathcal{B}_{\sigma}$ is as before. Our interest in general is in the setting where $\sigma \rightarrow 0$ (which happens as $n \rightarrow \infty$), so in some examples we will simply calculate $T_0(p_0)$. In other examples however, the truncation at level $\sigma$ will play a crucial role and in those cases we will compute $T_\sigma(p_0)$.

\begin{example}[Uniform null]
Suppose that the null distribution $p_0$ is Uniform$[a,b]$ then,
\begin{align*}
T_0(p_0) = |b - a|.
\end{align*}
\end{example}

\begin{example}[Gaussian null]
Suppose that the null distribution $p_0$ is a Gaussian, i.e. for some $\nu > 0, \mu \in \mathbb{R}$, 
\begin{align*}
p_0(x) = \frac{1}{\sqrt{2\pi} \nu } \exp (-(x- \mu)^2/(2 \nu^2)).
\end{align*}
In this case, a simple calculation (see Appendix~\ref{app:ex}) shows that,
\begin{align*}
T_0(p_0) = (8\pi)^{1/2} \nu.
\end{align*}
\end{example}

\begin{example}[Beta null]
Suppose that the null density is a Beta distribution:
\begin{align*}
p_0(x) = \frac{\Gamma(\alpha + \beta)}{\Gamma(\alpha) \Gamma(\beta)} x^{\alpha - 1} x^{\beta - 1} =
\frac{1}{B(\alpha,\beta)} x^{\alpha - 1} x^{\beta - 1}, 
\end{align*}
where $\Gamma$ and $B$ denote the gamma and beta functions respectively.
It is easy to verify that,
\begin{align*}
T_0(p_0) = \left(\int_0^1 \sqrt{p_0(x)} dx\right)^2 = \frac{B^2((\alpha + 1)/2, (\beta + 1)/2)}{ B(\alpha,\beta) }.
\end{align*}
To get some sense of the behaviour of this functional, we consider the case when $\alpha = \beta = t \rightarrow \infty.$ In this case, we show (see Appendix~\ref{app:ex}) that for $t \geq 1$,
\begin{align*}
 \frac{\pi^2}{4e^4} t^{-1/2}  \leq T_0(p_0) \leq \frac{e^4}{4} t^{-1/2}.
\end{align*}
In particular, we have that $T_0(p_0) \asymp t^{-1/2}$. 



\end{example}

\noindent {\bf Remark: } 
\begin{itemize}
\item These examples illustrate that in the simplest settings when the density $p_0$ is close to uniform, the $T$-functional is roughly the effective support of $p_0$. In each of these cases, it is straightforward to verify that the truncation of the $T$-functional simply affects constants so that the critical radius scales as:
\begin{align*}
\critradn \asymp \left(\frac{\sqrt{\lipcons} T_0(p_0) }{n} \right)^{2/5},
\end{align*}
where $T_0(p_0)$ in each case scales as roughly the size of the $(1 - \critradn)$-support of the density $p_0$, i.e. as the Lebesgue measure of the smallest set that contains $(1 - \critradn)$ probability mass.
This motivates understanding the Lipschitz density with smallest effective support, and we consider this next.
\end{itemize}

\begin{example}[Spiky null]
Suppose that the null hypothesis is:
\begin{align*}
p_0(x) = \begin{cases}
\lipcons x & 0 \leq x \leq \frac{1}{\sqrt{\lipcons}} \\
\frac{2}{\sqrt{\lipcons}} - \lipcons x & \frac{1}{\sqrt{\lipcons}} \leq x \leq \frac{2}{\sqrt{\lipcons}} \\
0 & \text{otherwise},
\end{cases}
\end{align*}
then we have that $T_0(p_0) \asymp \frac{1}{\sqrt{\lipcons}}.$
\end{example}

\noindent {\bf Remark: } 
\begin{itemize}
\item For the spiky null distribution we obtain an extremely fast rate, i.e. we have that the critical radius $\critradn \asymp n^{- 2/5},$
and is independent of the Lipschitz parameter $\lipcons$ (although, we note that the null $p_0$ is more spiky as $\lipcons$ increases). 
This is the fastest rate we obtain for Lipschitz testing. 
In settings where the tail decay is slow, the truncation of the $T$-functional can be crucial 
and the rates can be much slower. We consider these examples next. 
\end{itemize}

\begin{example}[Cauchy distribution]
The mean zero, Cauchy
distribution with parameter $\alpha$ has pdf:
\begin{align*}
p_0(x) = \frac{1}{\pi \alpha} \frac{\alpha^2}{x^2 + \alpha^2}.
\end{align*}
As we show (see Appendix~\ref{app:ex}), the $T$-functional without truncation is infinite, i.e. 
$T_0(p_0) = \infty$. However, the truncated $T$-functional is finite. In the Supplementary Material
we show that for any $0 \leq \sigma \leq 0.5$ (recall that our interest is in cases where $\sigma \rightarrow 0$),
\begin{align*}
 \frac{4\alpha}{\pi}  \left[  \ln^2 \left( \frac{1}{\sigma}  \right)  \right] \leq T_{\sigma}(p_0) \leq  \frac{4\alpha}{\pi}  \left[  \ln^2 \left( \frac{2 e}{\pi \sigma}  \right)  \right],
\end{align*}
i.e. we have that $T_\sigma(p_0) \asymp \ln^2(1/\sigma)$.
\end{example}
\noindent {\bf Remark: } 
\begin{itemize}
\item When the null distribution is Cauchy as above, we note that the rate for the critical radius
is no longer the typical $\critradn \asymp n^{-2/5},$ even when the other problem specific parameters ($\lipcons$ and the Cauchy parameter $\alpha$) are held fixed. We instead obtain
a slower $\critradn \asymp (n/\log^2 n)^{-2/5}$ rate. Our final example, shows that we can obtain an entire spectrum of slower rates.
\end{itemize}

\begin{example}[Pareto null] For a fixed $x_0 > 0$ and for $0 < \alpha < 1$, suppose that the null distribution is
\begin{align*}
p_0(x) = \begin{cases}
\frac{\alpha x_0^\alpha}{x^{\alpha + 1}}~~\text{for}~~ x \geq x_0, \\
0~~\text{for}~~x < x_0. \\
\end{cases}
\end{align*}
This distribution for $0 < \alpha < 1$ has thicker tails than the Cauchy distribution. The $T$-functional without truncation is infinite, i.e. $T_0(p_0) = \infty$, and we can further show that (see Appendix~\ref{app:ex}):
\begin{align*}
 \frac{4 \alpha x_0}{(1 - \alpha)^2} \left( \sigma^{- \frac{1 - \alpha}{2\alpha}} - 1\right)^2 = T_\sigma(p_0) \leq  \frac{4 \alpha x_0}{(1 - \alpha)^2} \sigma^{- \frac{1 - \alpha}{\alpha}}.
\end{align*}
In the regime of interest when $\sigma \rightarrow 0$, we have that $T_\sigma(p_0) \sim \sigma^{- \frac{1 - \alpha}{\alpha}}.$
\end{example}
\noindent {\bf Remark: } 
\begin{itemize}
\item We observe that once again, the critical radius no longer follows the typical rate: $\critradn \asymp n^{-2/5}.$ Instead we obtain the rate, $\critradn \asymp n^{-2 \alpha/(2 + 3 \alpha)},$ and indeed have much slower rates as $\alpha \rightarrow 0$, indicating the difficulty of testing heavy-tailed distributions against a Lipschitz, $\ell_1$ neighborhood.
\end{itemize}
We conclude this section by emphasizing the value of the local minimax perspective and of studying the goodness-of-fit problem beyond the uniform null. We are able to provide a sharp characterization of the critical radius for a broad class of interesting examples, and we obtain faster (than at uniform) rates when the null is spiky and non-standard rates in cases when the null is heavy-tailed.

\subsection{A recursive partitioning scheme}
\label{sec:part}
At the heart of our upper and lower bounds are spatially adaptive partitions of the domain of $p_0$. The partitions used in our upper and lower bounds are similar but not identical.  
In this section, we describe an algorithm for producing the desired partitions and then briefly describe some of the main properties of the partition that we leverage in our upper and lower bounds.

We begin by describing the desiderata for the partition from the perspective of the upper bound.
Our goal is to construct a test for the hypotheses in~\eqref{test:lip}, and we do so by constructing
a partition (consisting of $N+1$ cells) 
$\{A_1,\ldots,A_{\smallN},A_{\infty}\}$ of $\mathbb{R}^d$. Each cell $A_i$ for $i \in \{1,\ldots,\smallN\}$ will be a cube, while the cell $A_{\infty}$ will be arbitrary but will have small total probability content. We let,
\begin{align}
\label{eqn:k}
K := \bigcup_{i=1}^{\smallN} A_i.
\end{align}
We form the multinomial corresponding to the partition
$\{P_0(A_1),\ldots,P_0(A_{\smallN}),P_0(A_{\infty})\},$ where $P_0(A_i) = \int_{A_i} p_0(x) dx$.
We then test this multinomial using the counts of the number of samples falling in each cell of the partition.

\noindent {\bf Requirement 1: } A basic requirement of the partition is that it must ensure that a density $p$ that is 
at least $\critradn$ far 
away in $\ell_1$ distance from $p_0$ should remain roughly $\critradn$ away from $p_0$ when converted to a multinomial. Formally, for any $p$ such that $\|p - p_0\|_1 \geq \critradn, p \in \lipclass(\lipcons)$ we require that for some small constant $c > 0$,
\begin{align}
\label{eqn:prop1}
\sum_{i=1}^N | P_0(A_i) - P(A_i) | + | P_0(A_{\infty}) - P(A_{\infty})| \geq c \critradn.
\end{align}
Of course, there are several ways to ensure this condition is met. In particular, supposing that we restrict attention to densities supported on $[0,1]$ then it suffices for instance to choose roughly $\lipcons/\critradn$ even-width bins. This is precisely the partition considered in prior work \cite{ingster94,ingster1997adaptive,ariascastro16}. When we do not restrict attention to compactly supported, uniform densities an even-width partition is no longer optimal and a careful optimization of the upper and lower bounds with respect to the partition yields the optimal choice. 
The optimal partition has bin-widths that are roughly taken
proportional to $p_0^{\gamma}(x)$, where the constant of proportionality is chosen to ensure
that the condition in~\eqref{eqn:prop1} is satisfied. Precisely determining the constant of proportionality turns out to be quite subtle so we defer a discussion of this to the end of this section. 



\noindent {\bf Requirement 2: } A second requirement that arises in both our upper and lower bounds 
is that the cells of our partition (except $A_{\infty}$) are not chosen too wide. In particular, we must choose the cells small enough to ensure that the density is roughly constant on each cell, i.e. on each cell we need that for any $i \in \{1,\ldots,\smallN\}$,
\begin{align}
\frac{\sup_{x \in A_i} p_0(x)}{\inf_{x \in A_i} p_0(x)} \leq 3.
\end{align}
Using the Lipschitz property of $p_0$, this condition is satisfied if any point $x$ is in a cell of diameter at most $p_0(x)/(2\lipcons)$.


Taken together the first two requirements suggest that we need to create a partition such that: for every point $x \in K$ the diameter of the cell $A$ containing the point $x$, should be roughly,
\begin{align*}
\text{diam}(A) \approx \min \left\{\theta_1 p_0(x), \theta_2 p_0^{\gamma}(x) \right\},
\end{align*}
where $\theta_1$ is to be chosen to be smaller than $1/(2\lipcons)$, and $\theta_2$ is chosen to ensure that Requirement 1 is satisfied.

Algorithm~\ref{alg:one} constructively establishes the existence of a partition satisfying these requirements. The upper and lower bounds use this algorithm with slightly different parameters.
The key idea is to recursively partition cells that are too large by halving each side. This is illustrated in Figure~\ref{fig:happy_part}. The proof of correctness of the algorithm uses the smoothness of $p_0$ in an essential fashion. Indeed, were the density $p_0$ not sufficiently smooth then such a partition would likely not exist.

In order to ensure that the algorithm has a finite termination, we choose two parameters $a,b \ll \critradn$ (these are chosen sufficiently small to not affect subsequent results): 
\begin{itemize}
\item We restrict our attention to the $a$-effective support of $p_0$, i.e. we define $S_a$ to be the
smallest cube
centered at the mean of $p_0$ such that,
$P_0(S_a) \geq 1 - a.$
We begin with $A_{\infty} = S_{a}^c$.
\item If the density in any cell is sufficiently small we do not split the cell further, i.e. for a parameter $b$, if $\sup_{x \in A} p_0(x) \leq b/\text{vol}(S_a)$ then we do not split it, rather we add it to $A_{\infty}$. By construction, such cells have total probability content at most $b$. 
\end{itemize}

For each cube $A_i$ for $i \in \{1,\ldots,\bigN\}$ we let $x_i$ denote its centroid, and we let $\bigN$ denote the number of cubes created by Algorithm~\ref{alg:one}. \\
\begin{algorithm}[H]
\label{alg:one}
\caption{Adaptive Partition}
\begin{enumerate}[leftmargin=0cm]
\item {\bf Input: }  Parameters $\theta_1,\theta_2, a,b.$

\item Set $A_\infty = \emptyset$ and $A_1 = S_{a}$.

\item For each cube $A_i$ do:
\begin{itemize} 
\item If 
\begin{equation}\label{eq:stop1}
\sup_{x \in A_i} p_0(x_i) \leq \frac{b}{ \text{vol}(S_{a})},
\end{equation}
then remove $A_i$ from the partition and let $A_\infty = A_\infty \cup A_i$. 
\item If 
\begin{equation}\label{eq:stop2}
\text{diam}(A_i) \leq \min 
\left\{ \theta_1 \nul(x_i), \theta_2 \nul^\gamma(x_i) \right\},
\end{equation}
then do nothing to $A_i$.
\item If $A_i$ fails to satisfy
(\ref{eq:stop1}) or (\ref{eq:stop2}) then
replace $A_i$ by a
set of $2^d$ cubes that are obtained by halving the
original $A_i$ along each of its axes. 
\end{itemize}

\item If no cubes are split or removed, STOP. Else go to step 3.

\item {\bf Output: } Partition ${\cal P}=\{A_1,\ldots,A_{\bigN},A_{\infty}\}$.
\end{enumerate}
\end{algorithm}

\begin{center}
\begin{figure}
\begin{subfigure}{0.5\textwidth} 
\includegraphics[scale=0.44]{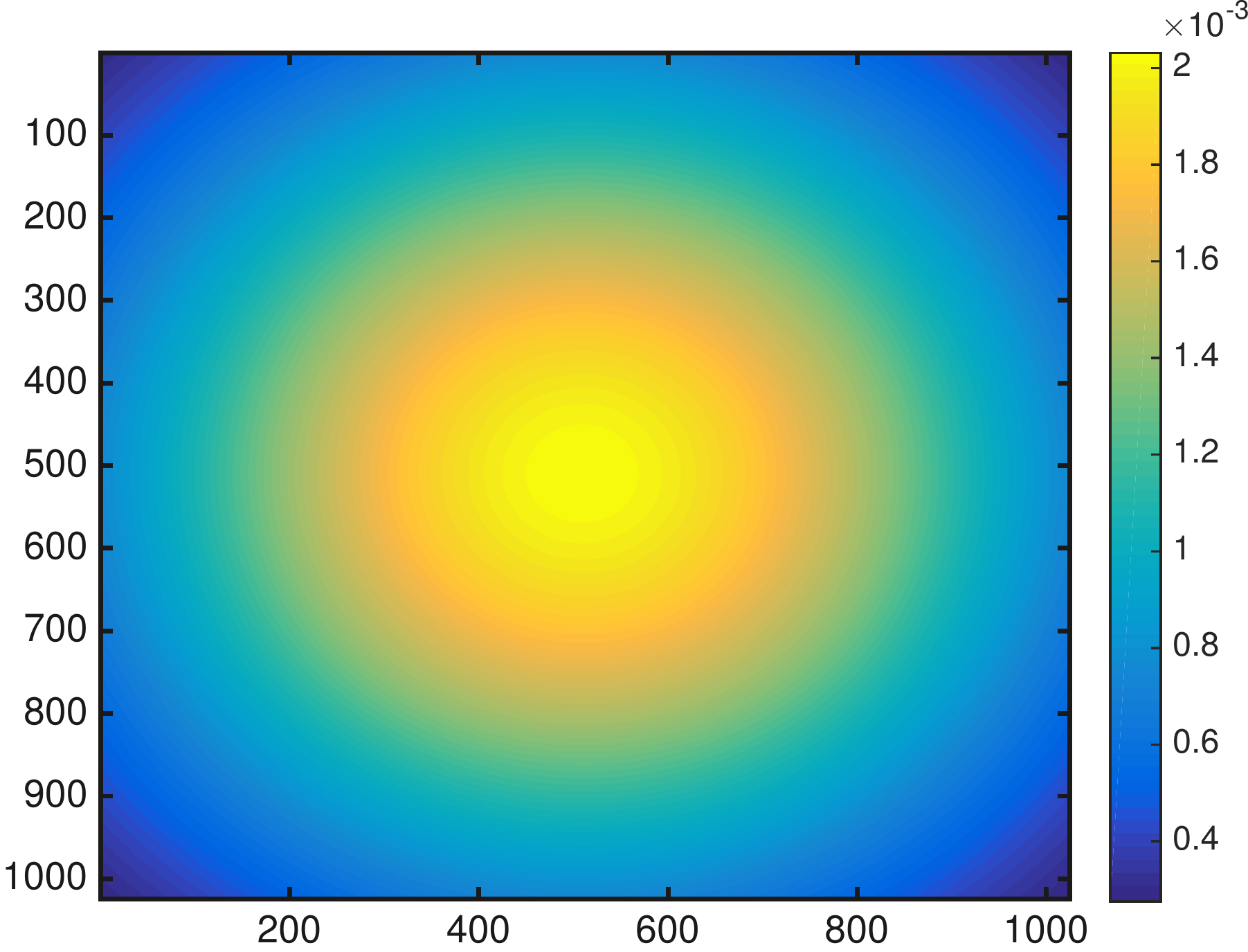}
\caption{\phantom{}}
\end{subfigure}
\hspace{0.1cm}
\begin{subfigure}{0.5\textwidth} 
\includegraphics[scale=0.44]{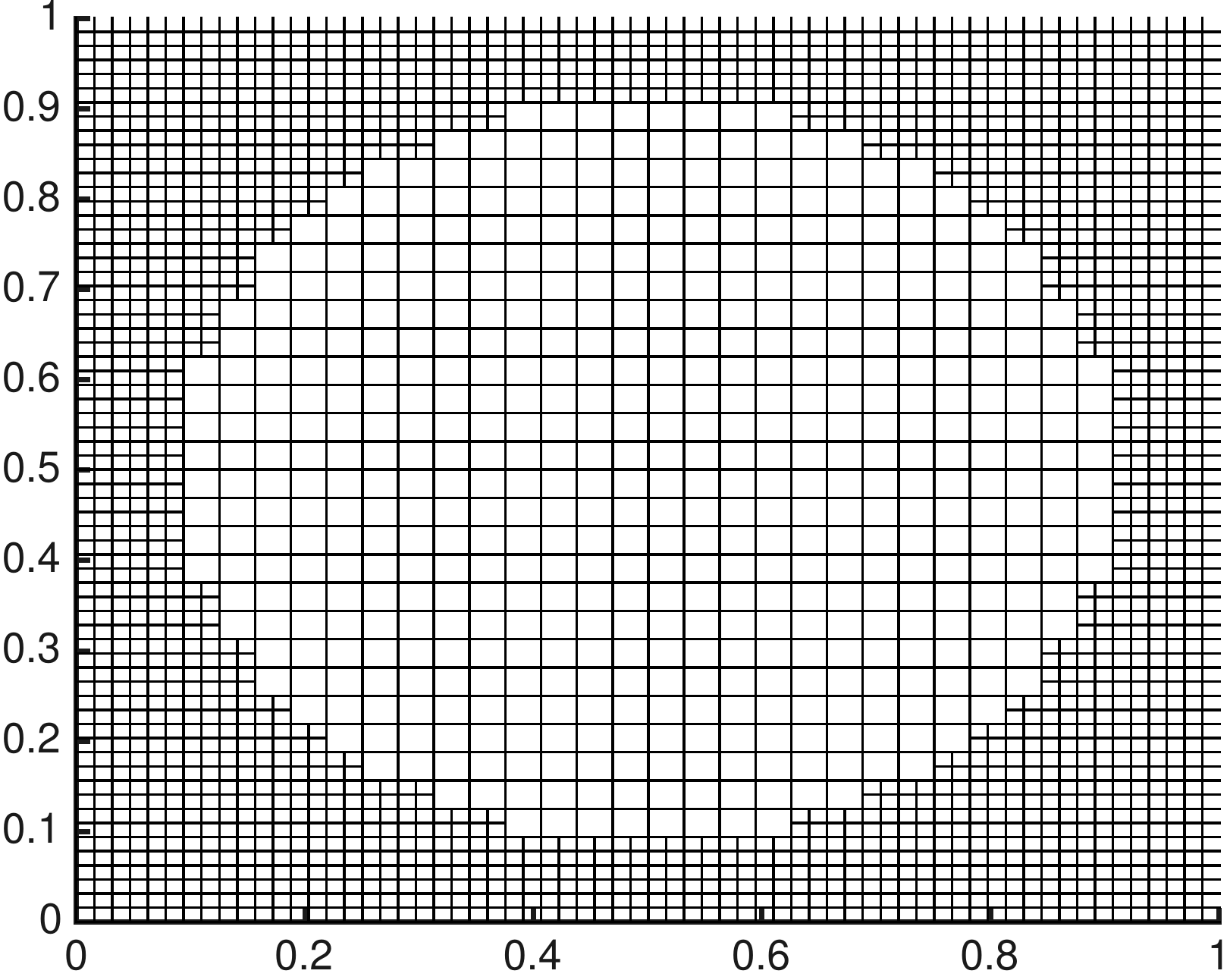}
\caption{\phantom{}}
\end{subfigure}
\caption{(a) A density $p_0$ on $[0,1]^2$ evaluated on a $1000 \times 1000$ grid. (b) The corresponding spatially adaptive partition $\mathcal{P}$ produced by Algorithm~\ref{alg:one}. Cells of the partition are larger in regions where the density $p_0$ is higher. }
\label{fig:happy_part}
\end{figure}
\end{center}

\noindent {\bf Requirement 3: } The final major requirement
is two-fold: (1) we require that
the $\gamma$-norm of the density 
over the support of the partition should be upper bounded by the truncated $T$-functional,
and (2) that the density over the cells of the partition be sufficiently large.
This necessitates a further pruning of the partition, 
where we order cells by their probability content and 
successively 
eliminate (adding them to $A_{\infty}$)
cells of low probability until we have eliminated mass that is close to the desired truncation level.
This is accomplished by Algorithm~\ref{alg:two}.

\begin{algorithm}[H]
\label{alg:two}
\caption{Prune Partition}
\begin{enumerate}[leftmargin=0cm]
\item {\bf Input: }  
Unpruned partition ${\cal P}=\{A_1,\ldots,A_{\bigN}, A_{\infty}\}$
and a target pruning level $c.$ Without loss of generality 
we assume $P_0(A_1) \geq P_0(A_2) \geq \ldots \geq P_0(A_{\bigN}).$  \\

\item For any $j \in \{1,\ldots,\bigN\}$ let $\mathcal{Q}(j) = \sum_{i = j}^{\bigN} P_0(A_i)$.
Let $j^*$ denote the smallest positive integer such that, $\mathcal{Q}(j^*) \leq c.$ 

\item If $\mathcal{Q}(j^*) \geq c/5$:
\begin{itemize}
\item Set $N = j^* - 1,$ and $A_{\infty} = A_{\infty} \cup A_{j^*} \cup \ldots \cup A_{\bigN}$.
\end{itemize}

\item If $\mathcal{Q}(j^*) \leq c/5$:
\begin{itemize}
\item Set $N = j^* - 1, \alpha = \min\{c/(5P_0(A_N)), 1/5\},$ and $A_{\infty} = A_{\infty} \cup A_{j^*} \cup \ldots \cup A_{\bigN}$.
\item $A_N$ is a cube, i.e. for some $\delta > 0$, $A_N = [a_1, a_1+\delta] \times \cdots \times [a_d,a_d+\delta]$. Let $D_1 = [a_1, (1-\alpha)(a_1+\delta)] \times \cdots \times [a_d,(1-\alpha)(a_d+\delta)]$ and $D_2 = A_N - D_1$. Set: $A_N = D_1$ and $A_\infty = A_\infty \cup D_2$.
\end{itemize}

\item {\bf Output: }  ${\cal P}^\dagger = \{A_1,\ldots, A_N, A_\infty\}.$
\end{enumerate}

%
%
%
%
%
\end{algorithm}

It remains to specify a precise choice for the parameter $\theta_2$. We do so indirectly by defining
a function $\mu: \mathbb{R} \mapsto \mathbb{R}$ that is closely related to the truncated $T$-functional. 
For $x \in \mathbb{R}$ we define $\mu(x)$ as the smallest positive number that satisfies the equation:
\begin{equation}
\label{eqn:mudef}
\critradn = \int_{\mathbb{R}^d} \min \left\{ \frac{p_0(y)}{x}, \frac{\critradn p_0(y)^{\gamma}}{ \mu(x)} \right\} dy. 
\end{equation}
If $x < 1/\critradn$ then we obtain a finite value for $\mu(x)$, otherwise we take $\mu(x) = \infty$.
The following result, relates $\mu$ to the truncated $T$-functional.
\begin{lemma}
\label{lem:mubound}
For any $0 \leq x < 1/\critradn,$
\begin{equation}
T_{x \critradn}^\gamma(p_0) \leq \mu(x) \leq 2  T^\gamma_{x \critradn/2}(p_0). 
\end{equation}
\end{lemma}
\noindent With the definition of $\mu$ in place, we now state our main result regarding the partitions produced
by Algorithms~\ref{alg:one} and~\ref{alg:two}.
We let $\mathcal{P}$ denote the unpruned partition obtained from Algorithm~\ref{alg:one} and $\mathcal{P}^{\dagger}$ denote the pruned partition obtained from Algorithm~\ref{alg:two}.
For each cell $A_i$ we denote its centroid by $x_i$.
We have the following result summarizing some of the important properties of $\mathcal{P}$ and $\mathcal{P}^{\dagger}$.
\begin{lemma}
\label{lem:uppermain}
Suppose we choose, $\theta_1 =  1/(2\lipcons), \theta_2 = \critradn/(8 \lipcons \mu(1/4)), a = b = \critradn/1024, c = \critradn/512$, 
then the partition $\mathcal{P}^{\dagger}$ satisfies the following properties:
\begin{enumerate}
\item~[Diameter control] The partition has the property that,
\begin{align}
\label{eqn:prop1f}
\frac{1}{5} \min \left\{ \theta_1p_0(x_i), \theta_2 p_0^{\gamma}(x_i) \right\} \leq 
{\rm diam}(A_i) \leq 
\min \left\{ \theta_1p_0(x_i), \theta_2 p_0^{\gamma}(x_i) \right\}. 
\end{align}
\item~[Multiplicative control] The density is multiplicatively controlled on each cell, i.e.
for $i \in \{1,\ldots,\smallN\}$ we have that,
\begin{align}
\label{eqn:prop2f}
\frac{\sup_{x \in A_i} p_0(x)}{\inf_{x \in A_i} p_0(x)} \leq 2.
\end{align}
\item~[Properties of $A_{\infty}$] The cell $A_{\infty}$ has probability content roughly $\critradn$, i.e.
\begin{align}
\label{eqn:prop6f}
\frac{\critradn}{2560} \leq P(A_{\infty}) \leq \frac{\critradn}{256}.
\end{align}
\item~[$\ell_1$ distance]  The partition preserves the $\ell_1$ distance, i.e. 
for any $p$ such that $\|p - p_0\|_1 \geq \critradn, p \in \lipclass(\lipcons)$,
\begin{align}
\label{eqn:prop3f}
\sum_{i=1}^N | P_0(A_i) - P(A_i) | + | P_0(A_{\infty}) - P(A_{\infty})| \geq \frac{\critradn}{8}.
\end{align}
\item~[Truncated $T$-functional] Recalling the definition of $K$ in~\eqref{eqn:k}, we have that,
\begin{align}
\label{eqn:prop4f}
\int_K p_0^{\gamma}(x) dx \leq T_{\critradn/5120}^{\gamma}(p_0).
\end{align}
\item~[Density Lower Bound] The density over $K$ is lower bounded as:
\begin{align}
\label{eqn:prop5f}
\inf_{x \in K} p_0(x) \geq \left(\frac{\critradn}{5120 \mu(1/5120)}\right)^{1/(1-\gamma)}. 
\end{align}
\end{enumerate}
Furthermore, for any choice of the parameter $\theta_2$ the unpruned partition $\mathcal{P}$ 
of Algorithm~\ref{alg:one} satisfies~\eqref{eqn:prop1f} with the constant $5$ sharpened to $4$,~\eqref{eqn:prop2f} and the upper bound in~\eqref{eqn:prop6f}.
\end{lemma}
\noindent The proof of this result is technical and we defer it to Appendix~\ref{app:lip}.

While we focused our discussions on the properties of the partition from the perspective of establishing the upper bound in Theorem~\ref{thm:main} it turns out that several of these properties are crucial in proving the lower bound as well. 
The optimal adaptive partition creates larger cells in regions where the density $p_0$ is higher, and smaller cells where $p_0$ is lower. This might seem counter-intuitive from the perspective of the upper bound since we create many low-probability cells which are likely to be empty in a small finite-sample, and indeed this construction is in some sense opposite to the quantile transformation suggested by previous work \cite{ingster94,gine15}. However, from the perspective of the lower bound this is completely natural. It is intuitive that 
our perturbation be large in regions where the density is large since the likelihood ratio is relatively stable in these regions, and hence these changes are more difficult to detect. The requirement of smoothness constrains the amount by which we can we can perturb the density on any given cell, i.e. for a large perturbation the corresponding cell should have a large diameter leading to the conclusion that we must use larger cells in regions where $p_0$ is higher.

In this section, we have focused on providing intuition for our adaptive partitioning scheme. In the next section we provide the body of the proof of Theorem~\ref{thm:main}, and defer the remaining technical aspects to the Supplementary Material.

\subsection{Proof of Theorem~\ref{thm:main}}
\label{sec:lipproof}
We consider the lower and upper bounds in turn.
\subsubsection{Proof of Lower Bound}
We note that the lower bound in~\eqref{eqn:minimaxcritradlipup} is trivial when $\critradn \geq 1/C$ so throughout the proof we focus on the case when $\critradn$ is smaller than a universal constant, i.e. when $\critradn \leq \frac{1}{C}$.

\noindent {\bf Preliminaries: }  We begin by briefly introducing the lower bound technique due to Ingster (see for instance \cite{ingster03}). Let ${\cal P}$ be a set of distributions and let $\Phi_n$ be the set
level $\alpha$ tests based on $n$ observations
where $0 < \alpha < 1$ is fixed.
We want to bound the minimax type II error
\begin{align*}
\zeta_n({\cal P}) = \inf_{\phi\in\Phi_n}\sup_{P\in {\cal P}}P^n(\phi=0).
\end{align*}
Define $Q$ as $Q(A) = \int P^n (A) d\prior (P)$, where $\prior$ is a prior distribution whose support is contained in ${\cal P}$.
In particular, if
$\prior$ is uniform on a finite set $P_1,\ldots, P_N$ then
\begin{equation*}
Q(A) = \frac{1}{N}\sum_j P_j^n(A).
\end{equation*}
Given $n$ observations we define the likelihood ratio
\begin{equation*}
\likrat(Z_1,\ldots,Z_n) = \frac{d Q}{dP_0^n} = \int \frac{p(Z_1,\ldots,Z_n)}{p_0(Z_1,\ldots,Z_n)} d\prior(p) = 
\int \prod_j \frac{p(Z_j)}{p_0(Z_j)} d\prior(p).
\end{equation*}

\begin{lemma}
\label{lemma:ingster}
Let $0<\zeta < 1-\alpha$.
If
\begin{equation}\label{eq:L}
\mathbb{E}_0[\likrat^2(Z_1,\ldots,Z_n)] \leq 1 + 4 (1-\alpha-\zeta)^2
\end{equation}
then
$\zeta_n({\cal P}) \geq \zeta$.
\end{lemma}
\noindent Roughly, this result asserts that in order to produce a minimax lower bound on the Type II error, it suffices to appropriately upper bound the second moment under the null of the likelihood ratio.
The proof is standard but presented in Appendix~\ref{app:lip} for completeness. A natural way to construct the prior $\prior$ on the set of alternatives, is to partition the domain of $p_0$ and then to locally perturb $p_0$ by adding or subtracting sufficiently smooth ``bumps''. In the setting where the partition has fixed-width cells this construction is standard \cite{ingster94,ariascastro16} and we provide a  generalization to allow for variable width partitions and to allow for non-uniform $p_0$.
Formally, let $\psi$ be a smooth bounded function on the hypercube $\mathcal{I} = [-1/2,1/2]^d$ such that
\begin{align*}
\int_{\mathcal{I}} \psi(x) dx = 0~~~\text{and}~~~\int_{\mathcal{I}} \psi^2(x) dx = 1.
\end{align*}
Let
$\mathcal{P} = \{A_1,\ldots,A_{\smallN},A_{\infty}\}$ be any partition that satisfies
the condition in~\eqref{eqn:prop2f}, and further let $\{x_1,\ldots,x_{\smallN}\}$ denote the centroids
of the cells $\{A_1,\ldots,A_{\smallN}\}$. Suppose further, that each cell 
$A_j$ for $j \in \{1,\ldots,\smallN\}$ is a cube with side-length $c_j h_j$ for some constants $c_j \leq 1$, and 
\begin{align*}
h_j = \frac{1}{\sqrt{d}}\min\{\theta_1 p_0(x_j), \theta_2 p_0^{\gamma}(x_j)\}.
\end{align*}
Let $\eta = (\eta_1,\eta_2,\ldots,\eta_{\smallN})$ be a Rademacher sequence
and define
\begin{equation}\label{eq:feta}
p_\eta = p_0 +  \sum_{j=1}^{\smallN} \rho_j \eta_j \psi_j
\end{equation}
where each
$\rho_j  \geq 0$ and
$$
\psi_j(t) = \frac{1}{c_j^{d/2}h_j^{d/2}}\, \psi\left( \frac{t-x_j}{c_jh_j}\right)
$$
for $t\in A_j$.
Hence, $\int_{A_j}\psi_j(t) = 0$ and
$\int_{A_j}\psi_j^2(t) = 1$.
Finally, let us denote:
\begin{align*}
\finalcon := \max \left\{ \|\psi\|_{\infty}, \frac{8 \|\psi^{\prime}\|_{\infty}}{(1 - \intconst)} \right\},~~~\text{and}~~~\finalcontwo := \|\psi\|_1.
\end{align*}
With these definitions in place we state a result that gives a lower bound for a sequence of
perturbations $\rho_j$ that satisfy certain conditions.
\begin{lemma}
\label{lem:rho}
Let $\alpha,\zeta$ and $\critradn$ be non-negative numbers with
$1-\alpha-\zeta>0$.
Let 
$C_0 = 1 + 4(1-\alpha-\zeta)^2$.
Assume that for each $j \in \{1,\ldots,\smallN\},$ $\rho_j$ and $h_j$ satisfy:
\begin{align}
(a)\ & \rho_j  \leq \frac{c_j^{d/2}}{\finalcon} \lipcons h_j^{1+ \frac{d}{2}}~\label{eqn:condone}\\
(b)\ & \sum_{j=1}^{\smallN} \rho_j c_j^{d/2} h_j^{d/2} \geq \frac{\critradn}{\finalcontwo}~\label{eqn:condtwo}\\
(c)\ & \sum_{j=1}^{\smallN} \frac{\rho_j^4}{p_0^2(x_j)} \leq \frac{ \log C_0}{4n^2}~\label{eqn:condthree},
\end{align}
then the Type II error of any test is at least $\zeta$.
\end{lemma}
Effectively, this lemma generalizes the result of \citet{ingster94} to allow for non-uniform $p_0$
and further allows for variable width bins. The proof proceeds by verifying that under the conditions of the lemma, $p_{\eta}$ is sufficiently smooth, and separated from $p_0$ by at least $\critradn$ in the $\ell_1$ metric. 
We let the prior be uniform on the the set of possible distributions $p_{\eta}$ and
directly analyze the second moment of the likelihood ratio, and obtain the result by applying Lemma~\ref{lemma:ingster}. See Appendix~\ref{app:lip} for the proof of this lemma.
It is worth noting the condition in~\eqref{eqn:condone}, which ensures smoothness of $p_{\eta}$, allows for larger perturbations $\rho_j$ for bins where $h_j$ is large, which is one of the key benefits
of using variable bin-widths in the lower bound.

With this result in place, to produce the desired minimax lower bound it only remains to specify the partition, select a sequence of perturbations $\rho_j$ and verify that the conditions of Lemma~\ref{lem:rho} are satisfied.

\noindent {\bf Final Steps: } 
We begin by specifying the partition. 
We define,
\begin{align*}
\nu = \min\left\{ \frac{\finalcontwo}{\finalcon 4^{d+1} \sqrt{d}}, 1 \right\}.
\end{align*}
For the lower bound we do not need to prune the partition, rather we simply apply Algorithm~\ref{alg:one} with $\theta_1 = 1/(2\lipcons),$ and $\theta_2 = \critradn/(\lipcons \nu \mu(2/\nu))$.
We choose $a = b = \critradn/1024$, and denote the resulting partition~$\mathcal{P} = \{A_1,\ldots,A_{\smallN}, A_{\infty}\}$. Using Lemma~\ref{lem:uppermain} we have that the
partition satisfies~\eqref{eqn:prop1f} with the constant 5 replaced by 4, \eqref{eqn:prop2f} and the upper bound in~\eqref{eqn:prop6f}.
We now choose a sequence $\{\rho_1,\ldots,\rho_{\smallN}\}$, and proceed to verify that the conditions of Lemma~\ref{lem:rho} are satisfied.
We choose,
\begin{align*}
\rho_j =  \frac{c_j^{d/2}}{\finalcon} \lipcons h_j^{1+ \frac{d}{2}},
\end{align*}
thus ensuring the condition in~\eqref{eqn:condone} is satisfied.

\noindent {\bf Verifying the condition in~\eqref{eqn:condtwo}: } 
Recall the definition of $\mu$ in~\eqref{eqn:mudef},
\begin{align*}
\frac{\critradn}{\nu} = \int_{\mathbb{R}^d} \min \left\{ \frac{p_0(y)}{2}, \frac{\critradn p_0(y)^{\gamma}}{\nu \mu(2/\nu)} \right\} dy,
\end{align*}
provided that $\critradn \leq \nu/2$ which is true by our assumption on the critical radius.
Recalling the definition of $K$ in~\eqref{eqn:k}, we have that,
\begin{align*}
\int_K \min \left\{ \frac{p_0(y)}{2}, \frac{\critradn p_0(y)^{\gamma}}{\nu \mu(2/\nu)} \right\} dy \geq \frac{\critradn}{\nu} - \frac{P(A_{\infty})}{2}.
\end{align*}
We define the function 
\begin{align*}
h(y) := \frac{1}{\sqrt{d}}\min \left\{ \frac{p_0(y)}{2\lipcons}, \frac{\critradn p_0(y)^\gamma}{\lipcons \nu \mu(2/\nu)} \right\},
\end{align*}
and as a consequence of the property~\eqref{eqn:prop2f} we obtain that for any $y \in A_j$ for $j \in \{1,\ldots,\smallN\}$,
\begin{align*}
h_j \geq \frac{h(y)}{2}.
\end{align*}
This in turn yields that,
\begin{align*}
L \sum_{j=1}^{\smallN} h_j^{d+1} \geq \frac{1}{2\sqrt{d}} \int_K \min \left\{ \frac{p_0(y)}{2}, \frac{\critradn p_0(y)^{\gamma}}{\nu \mu(2/\nu)} \right\} dy \geq \frac{1}{2 \sqrt{d}} \left( \frac{\critradn}{\nu} - \frac{P(A_{\infty})}{2} \right) \geq \frac{\critradn}{4 \sqrt{d} \nu},
\end{align*}
where the final step uses the upper bound in~\eqref{eqn:prop6f}. We then have that,
\begin{align*}
\sum_{j=1}^{\smallN} \rho_j c_j^{d/2} h_j^{d/2} = \sum_{j=1}^N \frac{\lipcons c_j^d h_j^{d+1}}{\finalcon} \geq \sum_{j=1}^N \frac{\lipcons h_j^{d+1}}{4^d \finalcon} \geq \frac{\critradn}{\finalcontwo},
\end{align*}
which establishes the condition in~\eqref{eqn:condtwo}.

\noindent {\bf Verifying the condition in~\eqref{eqn:condthree}: } 
We note the inequality (which can be verified by simple case analysis) 
that for $a,b, u, v \geq 0$,
\begin{align*}
\min\{a,b\} \leq \min \{ a^{\frac{u}{u+v}} b^{\frac{v}{u+v}}, b\},
\end{align*}
in particular for $u=1, v=3+d$ we obtain,
\begin{align}
\label{eqn:magic}
\min \{a,b\} \leq \min\{ (a b^{3+d})^{\frac{1}{4+d}}, b\}.
\end{align}
Returning to the condition in~\eqref{eqn:condthree}
we have that,
\begin{align*}
\sum_{j=1}^\smallN \frac{\rho_j^4}{p_0(x_j)^2} \leq \frac{\lipcons^4}{\finalcon^4} \sum_{j=1}^\smallN \frac{c_j^{2d}  h_j^{4+2d}}{p_0(x_j)^2} \leq \frac{\lipcons^4}{\finalcon^4} \sum_{j=1}^\smallN \frac{h_j^d h_j^{4+d}}{p_0(x_j)^2},
\end{align*}
using the fact that $c_j \leq 1$.
Using the chosen values for $h_j$ we obtain that,
\begin{align*}
\sum_{j=1}^\smallN \frac{\rho_j^4}{p_0(x_j)^2} &\leq \frac{\lipcons^4}{\finalcon^4 \sqrt{d}^{4+d}} \sum_{j=1}^{\smallN} \frac{h_j^d}{p_0(x_j)^2} \min \left\{ \left[\frac{p_0(x_j)}{2 \lipcons }\right]^{4+d}, 
\left[ \frac{\critradn p_0^\gamma (x_j) }{\lipcons \nu \mu(2/\nu) }\right]^{4+d} \right\} \\
& \stackrel{\text{(i)}}{\leq}  \frac{\lipcons^4}{\finalcon^4 \sqrt{d}^{4+d}} \sum_{j=1}^{\smallN}  \frac{h_j^d}{p_0(x_j)^2}
\min \left\{ \frac{p_0(x_j)^3 \critradn^{3+d} }{ 2\lipcons  (\lipcons \nu \mu(2/\nu))^{3+d}}, \left[ \frac{\critradn p_0^\gamma (x_j) }{\lipcons \nu \mu(2/\nu) }\right]^{4+d} \right\} \\
&= \frac{\critradn^{3+d}}{\lipcons^d \mu(2/\nu)^{3+d} \nu^{4+d}  \finalcon^4 \sqrt{d}^{4+d}} 
\sum_{j=1}^{\smallN} h_j^d \min \left\{ \frac{p_0(x_j)}{2/\nu}, \frac{\critradn p_0(x_j)^\gamma }{\mu(2/\nu)}\right\} \\
&\leq \frac{2 \critradn^{3+d}}{\lipcons^d \mu(2/\nu)^{3+d} \nu^{4+d}  \finalcon^4 \sqrt{d}^{4+d}} 
\int_{K} \min \left\{ \frac{p_0(x)}{2/\nu}, \frac{\critradn p_0(x)^\gamma }{\mu(2/\nu)}\right\} dx \\
&\leq \frac{2 \critradn^{3+d}}{\lipcons^d \mu(2/\nu)^{3+d} \nu^{4+d}  \finalcon^4 \sqrt{d}^{4+d}} 
\int_{\mathbb{R}^d} \min \left\{ \frac{p_0(x)}{2/\nu}, \frac{\critradn p_0(x)^\gamma }{\mu(2/\nu)}\right\} dx \\
&\stackrel{\text{(ii)}}{\leq} \frac{2 \critradn^{4+d}}{\lipcons^d \mu(2/\nu)^{3+d} \nu^{4+d}  \finalcon^4 \sqrt{d}^{4+d}}.
\end{align*}
where (i) follows from the inequality in~\eqref{eqn:magic}, and (ii) uses~\eqref{eqn:mudef}.
Using Lemma~\ref{lem:mubound} we obtain,
\begin{align*}
\mu(2/\nu) \geq T^{\gamma}_{2\critradn/\nu} (p_0), 
\end{align*}
provided that $\critradn \leq \nu/2$. This yields that,
\begin{align*}
\sum_{j=1}^\smallN \frac{\rho_j^4}{p_0(x_j)^2} &\leq \frac{2 \critradn^{4+d}}{\lipcons^d 
T_{2 \critradn/\nu}^2 (p_0)
\nu^{4+d}  \finalcon^4 \sqrt{d}^{4+d}},
\end{align*}
and we require that this quantity is upper bounded by $\frac{\log C_0}{4n^2}$.
For constants $c_{1}, c_{2}$ that depend only on the dimension $d$
it suffices to choose $\critradn$ as the solution to the equation:
\begin{align*}
\critradn = \left( \frac{\lipcons^{d/2} T_{c_{2} \critradn}(p_0) \log C_0 }{ c_{1} n}  \right)^{2/(4+d)} 
\end{align*}
and an application of Lemma~\ref{lem:rho} yields the lower bound of Theorem~\ref{thm:main}.

\subsubsection{Proof of Upper Bound}
In order to establish the upper bound, we construct an adaptive partition using Algorithms~\ref{alg:one} and~\ref{alg:two}, and utilize the test analyzed in 
Theorem~\ref{thm:valup} from \cite{valiant14} to test the resulting multinomial. 
For the upper bound we use the partition $\mathcal{P}^{\dagger}$ studied in Lemma~\ref{lem:uppermain}, i.e. we take $\theta_1 = 1/(2\lipcons), \theta_2 = \critradn/(8 \lipcons \mu(1/4)), a = b = \critradn/1024$ and $c = \critradn/512$. Using the property in~\eqref{eqn:prop3f}, it suffices to upper bound the $V$-functional in~\eqref{eqn:vfunc}, 
for $\sigma = \critradn/128.$

The following technical lemma shows that the truncated $V$-functional is upper bounded by the $V$-functional over the partition excluding $A_{\infty}.$ For the partition $\mathcal{P}^{\dagger}$, we have the associated multinomial 
$q := \{P_0(A_1),\ldots,P_0(A_{\infty})\}$. With these definitions in place we have the following result.
\begin{lemma}
\label{lem:lipvfunc}
For the multinomial $q$ defined above, the truncated $V$-functional is upper bounded as:
\begin{align*}
V_{\critradn/128}^{2/3}(q) \leq \sum_{i=1}^{\smallN} P_0(A_i)^{2/3} := \kappa.
\end{align*}
\end{lemma}
\noindent We prove this result in Appendix~\ref{app:lip}.
Roughly, this lemma asserts that our pruning is less aggressive than the tail truncation
of the multinomial test from the perspective of controlling the 2/3-rd norm. 
With this claim in place it only remains to upper bound $\kappa$. Using the property
in~\eqref{eqn:prop2f} we have that,
\begin{align*}
\kappa \leq
\sum_{i=1}^{\smallN} \left(2 p_0(x_i) \text{vol}(A_i) \right)^{2/3} \leq 
2^{2/3} \sum_{i=1}^{\smallN} \frac{p_0(x_i)^{2/3}}{h_i^{d/3}} h_i^d.
\end{align*}
Using the condition in~\eqref{eqn:prop5f} verify that for all $x \in K$ we have that
\begin{align*}
\theta_1 p_0(x) \geq \frac{\critradn p_0^{\gamma}(x)}{ 10240  \lipcons \mu(1/5120) },
\end{align*}
and this yields that for a constant $c > 0$ for each $i \in \{1,\ldots,\smallN\}$,
\begin{align*}
h_i \geq  \frac{c \critradn p_0^{\gamma}(x_i)}{  \lipcons \mu(1/5120) }.
\end{align*}
Using the property in~\eqref{eqn:prop2f} we then obtain that for a constant $C > 0$,
\begin{align*}
\kappa \leq C \left(\frac{ \lipcons \mu(1/5120)}{\critradn} \right)^{d/3}\int_K p_0^{\gamma}(x) dx,
\end{align*}
and using the property~\eqref{eqn:prop4f} and~Lemma~\ref{lem:mubound}, we obtain that
for constants $c,C > 0$ that,
\begin{align*}
\kappa \leq C \left(\frac{ \lipcons}{\critradn} \right)^{d/3} T_{c \critradn}^{2/3}(p_0).
\end{align*}
%
With Lemma~\ref{lem:lipvfunc} we obtain that for the multinomial $q$,
\begin{align*}
V_{\critradn/128}(q) \leq C^{3/2} \left(\frac{ \lipcons}{\critradn} \right)^{d/2} T_{c \critradn}(p_0),
\end{align*}
which together with the upper bound of Theorem~\ref{thm:val-val} yields the desired
upper bound for Theorem~\ref{thm:main}. We note that a direct application of Theorem~\ref{thm:val-val} yields a bound on the critical radius that is the maximum of two terms, one scaling as $1/n$ and the other being the desired term in Theorem~\ref{thm:main}. In Lipschitz testing, the $1/n$ term is always dominated by the term involving the truncated functional. This follows from the lower bound on the truncated functional shown in~\eqref{eqn:lbclaim}.

\subsection{Simulations}
In this section, we report some simulation results on Lipschitz testing. We focus on the case
when $d = 1$.
In Figure~\ref{fig:lipsims} we compare the following tests:
\begin{enumerate}
\item 2/3-rd + Tail Test: This is the locally minimax test studied in Theorem~\ref{thm:main}, where we use our binning Algorithm followed by the locally minimax multinomial test from \cite{valiant14}.
\item Chi-sq. Test:  Here we use our binning Algorithm followed by the standard $\chi^2$ test.
\item Kolmogorov-Smirnov (KS) Test: Since we focus on the case when $d = 1$, we also compare to the standard KS test based on comparing the CDF of $p_0$ to the empirical CDF.
\item Naive Binning: Finally, we compare to the approach of using fixed-width bins, together with the $\chi^2$ test. Following the prescription of \citet{ingster94} (for the case when $p_0$ is uniform)
we choose the number of bins so that the $\ell_1$-distance between the null and alternate is approximately preserved, i.e. denoting the effective support to be $S$ we choose the bin-width
as $\critradn/(\lipcons \mu(S))$.
\end{enumerate}

We focus on two simulation scenarios: when the null distribution is a standard Gaussian, and when the null distribution has a heavier tail, i.e. is a Pareto distribution with parameter $\alpha = 0.5$.
We create the alternate density by smoothly perturbing the null after binning, and choose the perturbation weights as in our lower bound construction in order to construct a near worst-case alternative.

We set the $\alpha$-level threshold via simulation (by sampling from the null 1000 times) and we calculate the power under particular alternatives by averaging over a 1000 trials.
We observe several notable effects. First, we see that the locally minimax test can significantly out perform the KS test as well the test based on fixed bin-widths. The failure of the fixed bin-width test is more apparent in the setting where the null is Pareto as the distribution has a large effective support and the naive binning is far less parsimonious than the adaptive binning. On the other hand, we also observe that at least in these simulations the $\chi^2$ test and the locally minimax test from \cite{valiant14} perform comparably when based on our adaptive binning indicating the crucial role played by the binning procedure.

\begin{center}
\begin{figure}
\begin{tabular}{cc}
\includegraphics[scale=0.4]{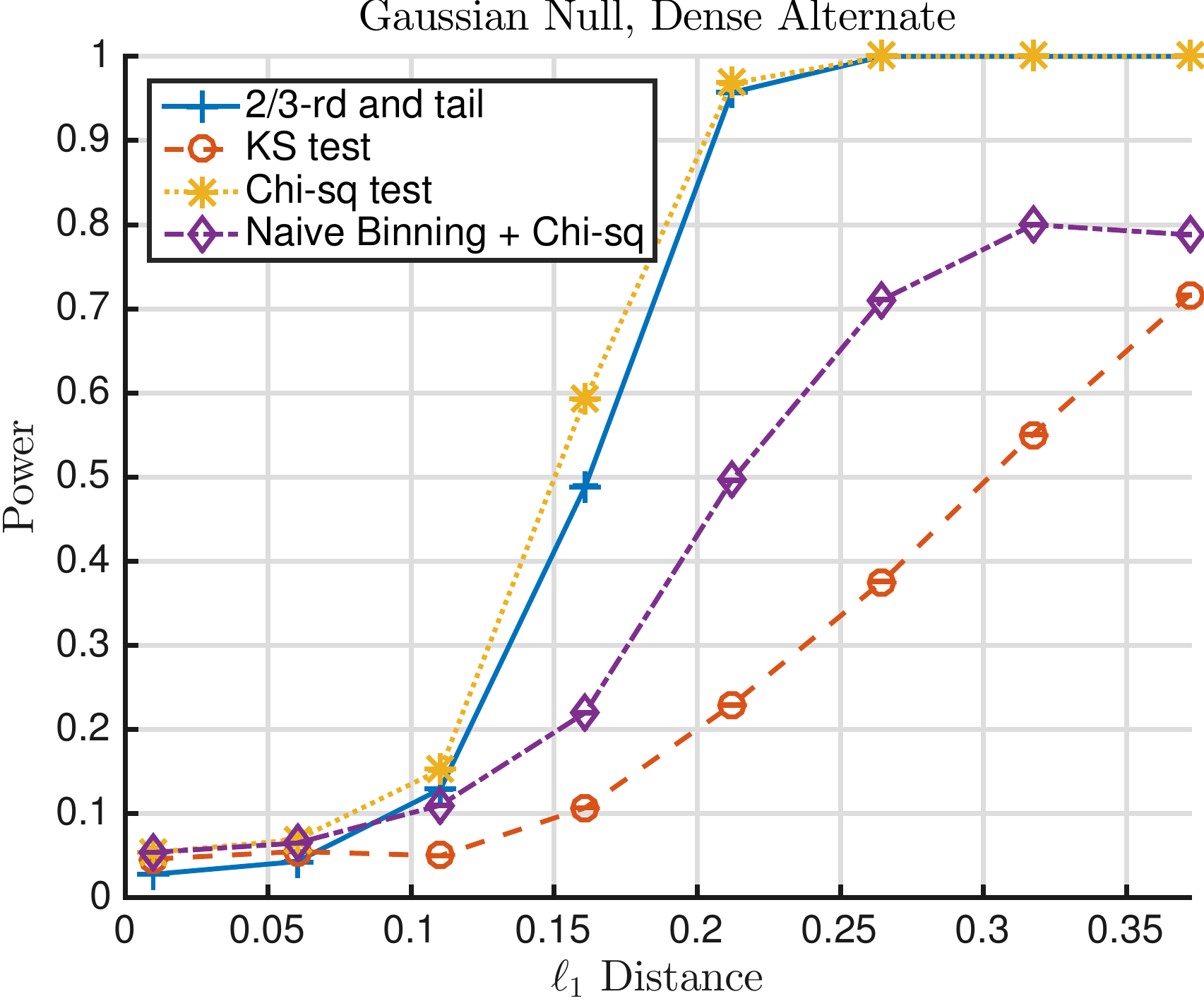} &~~~~~ \includegraphics[scale=0.4]{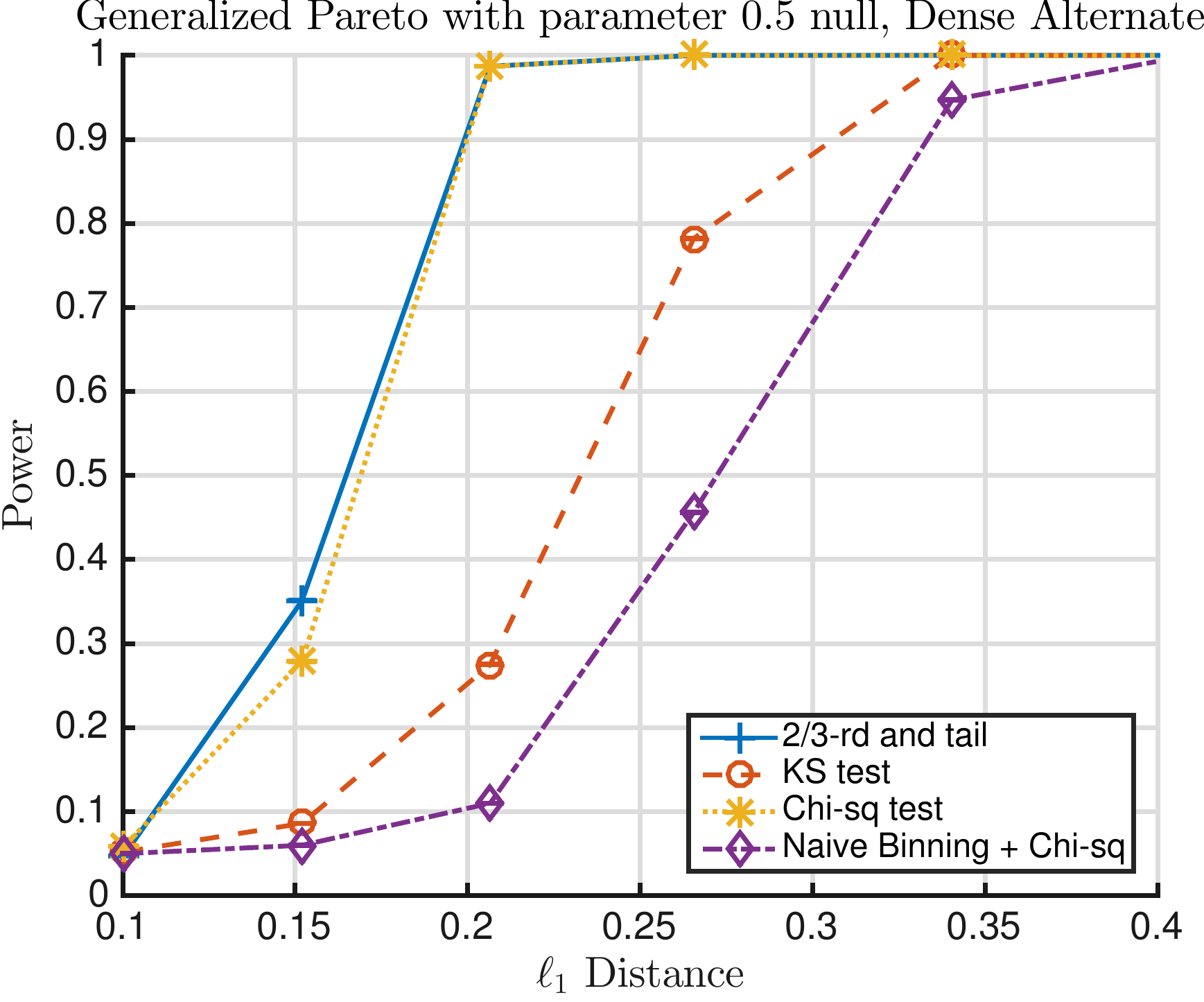}
\end{tabular}
\caption{A comparison between the KS test, multinomial tests on an adaptive binning and multinomial tests on a fixed bin-width binning. In the figure on the left we choose the null to be standard Gaussian and on the right we choose the null to be Pareto. The alternate is chosen to be a dense near worst-case, smooth perturbation of the null. The power of the tests are plotted against the $\ell_1$ distance between the null and alternate. Each point in the graph is an average over 1000 trials.   }
\label{fig:lipsims}
\end{figure}
\end{center}

\section{Discussion}
In this paper, we studied the goodness-of-fit testing problem in the context of testing multinomials and Lipschitz densities. For testing multinomials, we built on prior works \cite{valiant14,diakon16} to provide new globally and locally minimax tests. For testing Lipschitz densities we provide the first results that give a characterization of the critical radius under mild conditions. 

Our work highlights the heterogeneity of the critical radius in the goodness-of-fit testing problem
and the importance of understanding the local critical radius. In the multinomial testing problem it 
is particularly noteworthy that classical tests can perform quite poorly in the high-dimensional setting, and that simple modifications of these tests can lead to more robust inference. In the density testing problem, carefully constructed spatially adaptive partitions play a crucial role.

Our work motivates several open questions, and we conclude by highlighting a few of them. First, in the context of density testing we focused on the case when the density is Lipschitz. An important extension would be to consider higher-order smoothness. Surprisingly, \citet{ingster1997adaptive} shows that bin-based tests continue to be optimal for higher-order smoothness classes when the null is uniform on $[0,1]$. We conjecture that bin-based tests are no longer optimal when the null is not uniform, and further that the local critical radius is roughly determined by the solution to:
\begin{align*}
\critradn(p_0) \asymp \left[ \frac{\lipcons^{d/2s} S_{\critradn(p_0)} (p_0) }{n} \right]^{2s/(4s + d)},
\end{align*}
where the functional $S$ is defined as in~\eqref{eqn:tfuncdef} with $\gamma = 2s/(3s + d)$, and $\lipcons$ is the radius of the H\"{o}lder ball.
Second, it is possible to invert our locally minimax tests in order to construct confidence intervals. We believe that these intervals might also have some local adaptive properties that are worthy of further study. Finally, in the Appendix, we provide some basic results on the limiting distributions of the multinomial test statistics under the null when the null is uniform, and it would be interesting to consider the extension to settings where the null is arbitrary.

\section{Acknowledgements}
This work was partially supported by the NSF grant DMS-1713003. 
The authors would like to thank the participants of the Oberwolfach workshop on ``Statistical Recovery of Discrete, Geometric and Invariant Structures'', for their generous feedback. Suggestions by various participants including David Donoho, Richard Nickl, Markus Reiss, Vladimir Spokoiny, Alexandre Tsybakov, Martin Wainwright, Yuting Wei and Harry Zhou have been incorporated in various parts of this manuscript.

\bibliographystyle{plainnat}
\bibliography{bibtex}

\appendix

\section{Limiting behaviour of test statistics under the null}
\label{app:limit}
In this section, we consider the problem of  finding the asymptotic distribution of the 
multinomial test statistics under the null.
Broadly, there is a dichotomy between classical asymptotics where the null distribution is kept fixed
and a high-dimensional asymptotic where the number of cells is growing and the null distribution can vary with the number of cells. We present a few simple results on the limiting behaviour of our test statistics when the null is uniform and highlight some open problems. Although our techniques generalize in a straightforward way to non-uniform null distributions, they do not necessarily yield tight results.

We focus on the family of test statistics that we use in our paper, that are weighted $\chi^2$-type statistics:
\begin{align}
\label{eqn:tw}
T(w) = \sum_{i=1}^d \frac{(X_i - np_0(i))^2 - X_i}{w_i},
\end{align}
where each $w_i$ is a positive weight that is a fixed function of $p_0(i)$. This family includes the 2/3-rd statistic from~\cite{valiant14}, the truncated $\chi^2$ statistic that we propose, and the usual $\chi^2$ and $\ell_2$ statistics. When the null is uniform, this 
family of test statistics reduces to simple re-scalings of the $\ell_2$ statistic in~\eqref{eqn:elltwo}:
\begin{align*}
T_{\ell_2} =  \sum_{i=1}^d \left[(X_i - np_0(i))^2 - X_i\right].
\end{align*}
Our results are summarized in the following lemma.
\begin{lemma}
\begin{enumerate}
\item Classical Asymptotics: For any fixed $p_0$, the statistic $T(w)$ under the null converges in distribution to a weighted sum of $\chi^2$ distributions, i.e. for $Z_1,\ldots,Z_d \sim \chi^2_1,$ 
\begin{align}
\label{eqn:chisqasymp}
T(w) \overset{d}{\to} \sum_{i=1}^d \frac{p_i}{w_i} \left(Z_i - 1\right).
\end{align}
\item High-dimensional Asymptotics: Suppose $p_0$ is uniform
and $d \rightarrow \infty$, then we have that,
\begin{itemize}
\item If $n/\sqrt{d} \rightarrow \infty$, then 
\begin{align*}
\frac{T_{\ell_2}}{\sqrt{\textrm{Var}_0(T_{\ell_2})}} \overset{d}{\to} N(0, 1).
\end{align*}
\item If $n/\sqrt{d} \rightarrow 0$, then 
\begin{align*}
\frac{T_{\ell_2}}{\sqrt{\textrm{Var}_0(T_{\ell_2})}}  \overset{d}{\to} \delta_0.
\end{align*}
\end{itemize}
\end{enumerate}
\end{lemma}
\noindent {\bf Remarks:}
\begin{itemize}
\item The behaviour of the $\chi^2$-type statistics under classical asymptotics is well understood and we do not prove the claim in~\eqref{eqn:chisqasymp}.
\item Focusing on the high-dimensional setting, the asymptotic distribution of the test statistic is Gaussian in the regime where the risk of the optimal test tends to 0 as $n \rightarrow \infty$, and is degenerate in the regime where there are no consistent tests. In the most interesting regime when, $n/\sqrt{d} \rightarrow c$, the optimal test can have non-trivial risk, and the limiting distribution is neither Gaussian nor degenerate. 
\item More broadly, an important open question is to characterize the limiting distribution of the test statistic, under both the null and the alternate in the high-dimensional asymptotic.
\end{itemize}
\begin{proof}

The first part follows, by checking the Lyapunov conditions.
We denote 
\begin{align*}
\zeta_i = (X_i - np_0(i))^2 - X_i.
\end{align*}
and can calculate the sum of the variances as:
\begin{align*}
s_d^2 = \sum_{i=1}^d \text{var}(\zeta_i) =  \frac{2n^2}{d}.
\end{align*}
The Lyapunov condition then requires that,
\begin{align*}
\lim_{d \rightarrow \infty} \frac{1}{s_d^4} \sum_{i=1}^d \mathbb{E} \zeta_i^4 = 0.
\end{align*}
A straightforward computation gives that,
\begin{align*}
\mathbb{E}\zeta_i^4 = 8\frac{n^2}{d^2}  + 144 \frac{n^3}{d^3}  + 60 \frac{n^4}{d^4},
\end{align*}
so that the Lyapunov condition is satisfied provided that,
\begin{align*}
\lim_{d \rightarrow \infty} \frac{d^3}{n^6} \rightarrow 0,
\end{align*}
which is indeed the case.

In order to verify the degenerate limit it suffices to show that when $n/\sqrt{d} \rightarrow 0$, then the number of categories that have strictly larger than one occurrence converges to 0. When each observed category is observed only once we have that the test statistic is deterministic, i.e.,
\begin{align*}
T_{\ell_2}= \sum_{i=1}^d \zeta_i = (d-n) \frac{n^2}{d^2}  + n \left( \frac{n^2}{d^2} -  \frac{2n}{d} \right).
\end{align*}
When rescaled by the standard deviation we obtain that,
\begin{align*}
\frac{T_{\ell_2}}{\sqrt{\text{var}_0(T_{\ell_2})}} = \sqrt{\frac{d}{2n^2}} \left[(d-n) \frac{n^2}{d^2}  + n \left( \frac{n^2}{d^2} -  \frac{2n}{d} \right)\right] \rightarrow 0.
\end{align*}
Finally, we can bound the probability that any category is observed more than once as:
\begin{align*}
P( \exists~i, X_i \geq 2) &\leq \sum_{i=1}^d P(X_i \geq 2) \\
&\leq \sum_{i=1}^d \exp (- \lambda) \sum_{k=2}^\infty \left(\frac{n}{d}\right)^k \\
&\leq  \frac{Cn^2}{d} \rightarrow 0.
\end{align*}
Taken together these facts give the desired degenerate limit.
\end{proof}

\section{Analysis of Multinomial Tests}
\label{app:mult}

\subsection{Proof of Theorem~\ref{thm:truncchisq}}
In this section we analyze the truncated $\chi^2$ test. 
For convenience, throughout this proof we 
we work with a scaled version of the statistic in~\eqref{eqn:truncstat}, i.e. we let $T := \truncstat/d$ 
and abusing notation slightly we redefine $\theta_i$ appropriately, i.e.
we take $\theta_i = \max\{1,dp_0(i)\}$.

We begin by controlling the size of the truncated $\chi^2$ test. Fix any multinomial $p$ on $[d]$, and 
suppose we denote $\Delta_i = p_0(i) - p(i)$, then a straightforward computation shows that,
\begin{align}
\label{eqn:altmean2}
\mathbb{E}_p[T] &= n^2 \sum_{i =1}^d \frac{\Delta_i^2}{\theta_i}, 
\end{align}
\begin{align}
\label{eqn:altvar2}
\mathrm{Var}_p[T] &= \sum_{i =1}^d  \frac{1}{\theta_i^2} \left[ 2n^2 p_0(i)^2 + 2n^2 \Delta_i^2 - 4 n^2 \Delta_i p_0(i) + 4n^3 \Delta_i^2 p_0(i) - 4n^3 \Delta_i^3 \right].
\end{align}
This yields that the null variance of $T$ is given by:
\begin{align*}
\mathrm{Var}_0[T] =  \sum_{i =1}^d  \frac{ 2n^2 p_0(i)^2}{\theta_i^2},
\end{align*}
which together with Chebyshev's inequality yields the desired bound on the size. Turning our attention to the power of the test we fix a multinomial $p \in \mathcal{P}_1$. 
Denote the $\alpha$ level threshold of the test by 
\begin{align*}
t_{\alpha} = n \sqrt{ \frac{2}{\alpha} \sum_{i=1}^d \frac{ p_0(i)^2 }{ \theta_i^2}}.
\end{align*}
We observe that, if we can verify the following two conditions:
\begin{align}
\label{eqn:condonex}
t_{\alpha} &\leq \frac{\mathbb{E}_p[T]}{2} \\
\label{eqn:condtwox}
\mathbb{E}_p[T] &\geq2 \sqrt{\frac{\mathrm{Var}_p[T]}{\zeta}},
\end{align}
then we obtain that $P(\trunctest = 0) \leq \zeta.$ To see this, observe that
\begin{align*}
P(\trunctest = 0) &\leq P_1( T - \mathbb{E}_p[T] < t_{\alpha} - \mathbb{E}_p[T]) \\
&\leq P_1( (T - \mathbb{E}_p[T])^2 < (t_{\alpha} - \mathbb{E}_p[T])^2) \\
&\leq \frac{\mathrm{Var}_p[T]}{ (t_{\alpha} - \mathbb{E}_p[T])^2} \leq \frac{4 \mathrm{Var}_p[T]}{  \mathbb{E}_p[T]^2} \leq \zeta.
\end{align*}

\noindent {\bf Condition in Equation~\eqref{eqn:condonex}: } This condition reduces to verifying the following,
\begin{align*}
2 t_{\alpha} \leq n^2 \sum_{i =1}^d \frac{\Delta_i^2}{\theta_i}, 
\end{align*}
and as a result we focus on lower bounding the mean under the alternate. By Cauchy-Schwarz we obtain that,
\begin{align}
\label{eqn:lbmean}
\sum_{i =1}^d \frac{\Delta_i^2}{\theta_i}  \geq \frac{ \|\Delta\|_1^2 }{ \sum_{i=1}^d \theta_i} \geq \frac{ \critradn^2}{ \sum_{i=1}^d \{ 1 + dp_0(i)\} } \geq \frac{ \critradn^2}{ 2d}.
\end{align}
We can further upper bound $t_\alpha$ as
\begin{align*}
t_\alpha = n \sqrt{ \frac{2}{\alpha} \sum_{i=1}^d \frac{ p_0(i)^2 }{ \theta_i^2}}  \leq  n \sqrt{ \frac{2}{d\alpha}},
\end{align*}
using the fact that $p_0(i)/\theta_i \leq \frac{1}{d}.$ 
This yields that Equation~\eqref{eqn:condonex} is satisfied if:
\begin{align*}
\frac{ \critradn^2}{2d} \geq \frac{ 2 \sqrt{2}}{ \sqrt{d \alpha} n},
\end{align*}
which is indeed the case.

\noindent {\bf Condition in Equation~\eqref{eqn:condtwox}: } We can upper bound the variance under the alternate as:
\begin{align*}
\mathrm{Var}_p[T] &\leq \sum_{t =1}^d  \frac{1}{\theta_t^2} \left[ 4n^2 p_0(t)^2 + 4n^2 \Delta_t^2 + 4n^3 \Delta_t^2 p_0(t) - 4n^3 \Delta_t^3 \right] \\
&= \underbrace{\sum_{t =1}^d \frac{4n^2 p_0(t)^2}{\theta_t^2}}_{U_1}  + 
 \underbrace{\sum_{t =1}^d \frac{4n^2 \Delta_t^2}{\theta_t^2}}_{U_2} +  \underbrace{\sum_{t =1}^d \frac{4n^3 \Delta_t^2 p_0(t)}{\theta_t^2} }_{U_3} + \underbrace{\sum_{t =1}^d \frac{- 4n^3 \Delta_t^3}{\theta_t^2}}_{U_4}.
\end{align*}
Consequently, it suffices to verify that,
\begin{align*}
\sum_{i=1}^4 \frac{2 \sqrt{U_i/\zeta}}{\mathbb{E}_p[T]} \leq 1. 
\end{align*}
for $i = \{1,2,3,4\}$ and we do this by bounding each of these terms in turn.
For the first term we follow a similar argument to the one dealing with the first condition,
\begin{align*}
\frac{2 \sqrt{U_1/\zeta}}{\mathbb{E}_p[T]} \leq \frac{8 d \sqrt{ \sum_{t =1}^d \frac{ p_0(t)^2}{\theta_t^2}}}{\sqrt{\zeta} n \critradn^2} \leq \frac{8 \sqrt{d} }{\sqrt{\zeta} n \critradn^2} \leq \frac{1}{4}.
\end{align*}
For the second term,
\begin{align*}
\frac{2 \sqrt{U_2/\zeta}}{\mathbb{E}_p[T]} \leq \frac{4 \sqrt{ \frac{1}{\zeta} \sum_{t=1}^d \frac{n^2 \Delta_t^2}{\theta_t^2} }}{\mathbb{E}_p[T]} \leq  \frac{4 \sqrt{ \frac{1}{\zeta} \sum_{t=1}^d \frac{n^2 \Delta_t^2}{\theta_t} }}{\mathbb{E}_p[T]} =  \frac{4}{ \sqrt{\zeta \mathbb{E}_p[T]} }.
\end{align*}
Using Equation~\eqref{eqn:lbmean} we obtain that,
\begin{align*}
\mathbb{E}_p[T] \geq  \frac{n^2\critradn^2}{2d},
\end{align*}
which in turn yields that,
\begin{align*}
\frac{2 \sqrt{U_2/\zeta}}{\mathbb{E}_p[T]} \leq \frac{8 \sqrt{d}}{n \critradn \sqrt{\zeta}} \leq \frac{1}{4}.
\end{align*}
Turning our attention to the third term we obtain that,
\begin{align*}
\frac{2 \sqrt{U_3/\zeta}}{\mathbb{E}_p[T]} = \frac{4\sqrt{  \frac{1}{\zeta}\sum_{t =1}^d \frac{n^3 \Delta_t^2 p_0(t)}{\theta_t^2} }}{\mathbb{E}_p[T]} \leq \frac{4\sqrt{  \frac{n}{d \zeta}\sum_{t =1}^d \frac{n^2 \Delta_t^2}{\theta_t} }}{\mathbb{E}_p[T]} = \frac{4 \sqrt{ \frac{n}{d \zeta}} }{\sqrt{\mathbb{E}_p[T]}}.
\end{align*}
Using the lower bound on the mean we obtain that,
\begin{align*}
\frac{2 \sqrt{U_3/\zeta}}{\mathbb{E}_p[T]} \leq \frac{8  }{n \critradn \sqrt{\zeta}} \leq \frac{1}{4}.
\end{align*}
For the final term, 
\begin{align*}
\frac{2 \sqrt{U_4/\zeta}}{\mathbb{E}_p[T]} \leq \frac{4\sqrt{  \frac{1}{\zeta}\sum_{t =1}^d \frac{n^3 |\Delta_t^3| }{\theta_t^2} }}{\mathbb{E}_p[T]} \leq \frac{4\sqrt{  \frac{n^3}{\zeta}\sum_{t =1}^d \frac{|\Delta_t^3| }{\theta_t^2} }}{\mathbb{E}_p[T]} \leq  \frac{4\sqrt{  \frac{1}{\zeta}\left(  \sum_{i=1}^d \frac{n^2\Delta_i^2}{\theta_i^{4/3}} \right)^{3/2}}}{\mathbb{E}_p[T]},
\end{align*}
where the last step uses the monotonicity of the $\ell_p$ norms. Observing that $\theta_i \geq 1$, we have
\begin{align*}
\frac{2 \sqrt{U_4/\zeta}}{\mathbb{E}_p[T]} \leq  \frac{4\sqrt{  \frac{1}{\zeta}\left(  \sum_{i=1}^d \frac{n^2\Delta_i^2}{\theta_i} \right)^{3/2}}}{\mathbb{E}_p[T]} =  \frac{4\sqrt{  \frac{1}{\zeta}}}{\mathbb{E}_p[T]^{1/4}} \leq \frac{ 8 d^{1/4}} { \sqrt{\zeta} \critradn^{1/2} \sqrt{n}} \leq \frac{1}{4}. 
\end{align*}
This completes the proof.

\subsection{Proof of Theorem~\ref{thm:valup}}
Recall the definition of $\mathcal{B}_{\sigma}$ in~\eqref{eqn:bulkdef}.
We define:
\begin{align}
\label{eqn:bulkdev}
\Delta_{\mathcal{B}_{\sigma}} = \sum_{i \in \mathcal{B}_{\sigma}} |p_0(i) - p(i)|,
\end{align}
and
\begin{align}
\label{eqn:pmin}
p_{\min,\sigma} = \min_{i \in {\cal B}_\sigma} p_0(i).
\end{align}
Our main results concern the combined test $\phi_V$ in~\eqref{eqn:valtest}. It is easy to verify that the size of this test is at most $\alpha$ so it only remains to control its power.
We first provide a general result that allows for a range of possible values for the parameter $\vparam$.
\begin{lemma}
\label{lem:valinter}
For any $\sigma \leq \frac{\critradn}{8}$, if 
\begin{align*}
n \geq 2 \max \left\{ \frac{2}{\alpha}, \frac{1}{\zeta} \right\} \max\left\{ \frac{1}{\sigma}, \frac{4096 V_{\sigma/2}(p_0)}{ \critradn^2} \right\},
\end{align*}
then the Type II error $P(\phi_V = 0) \leq \zeta.$
\end{lemma}
Taking this lemma as given, it is straightforward to verify the result of Theorem~\ref{thm:valup}. In particular, if we take $\sigma = \critradn/8,$ then we recover the result of the theorem.

\noindent {\bf Proof of Lemma~\ref{lem:valinter}: } As a preliminary, we state two technical results from \cite{valiant14}. The following result is Lemma~6 in \cite{valiant14}.
\begin{lemma}
\label{lem:valiant_main}
For any $c \geq 1$, suppose that $n \geq c \max \left\{ \frac{V_{\sigma}(p_0)^{1/3}}{p_{\min,\sigma}^{1/3} \Delta_{\mathcal{B}_{\sigma}}}, \frac{ V_{\sigma}(p_0) }{ \Delta_{\mathcal{B}_{\sigma}}^2 }\right\},$
then we have that
\begin{align*}
\text{Var}_{p}(T_2(\sigma)) \leq \frac{16}{c} \left[\mathbb{E}_p(T_2(\sigma)) \right]^2.
\end{align*}
\end{lemma}
\noindent The following result appears in the proof of Proposition~1 of \cite{valiant14}.
\begin{lemma}
\label{lem:nrelate}
For any $c \geq 1$, suppose that,
\begin{align*}
n \geq 2 c \frac{ V_{\sigma/2}(p_0) }{ \Delta_{\mathcal{B}_{\sigma}}^2},
\end{align*}
then we have that,
\begin{align*}
n \geq c \max \left\{ \frac{V_{\sigma}(p_0)^{1/3}}{p_{\min,\sigma}^{1/3} \Delta_{\mathcal{B}_{\sigma}}}, \frac{ V_{\sigma}(p_0) }{ \Delta_{\mathcal{B}_{\sigma}}^2 }\right\}.
\end{align*}
\end{lemma}
With these two results in place, we can now complete the proof. 
We divide the space of alternatives into two sets:
\begin{align*}
\mathcal{S}_{1} &= 
\left\{p: \|p - p_0\|_1 \geq \critradn, 
\sum_{i \in \epstail{\sigma}} |p_0(i) - p(i)| \geq 3 \sigma \right\} \\
\mathcal{S}_{2} &= 
\left\{p: \|p - p_0\|_1 \geq \critradn, 
\sum_{i \in \epstail{\sigma}} |p_0(i) - p(i)| < 3 \sigma  \right\}.
\end{align*}
In order to show desired result it then suffices to show that
when $p \in \mathcal{S}_1$, $P(\phi_1  = 0) \leq \zeta$, and
that when $p \in \mathcal{S}_2$, $P(\phi_2 = 0) \leq \zeta$. We verify each of these claims in turn.

\noindent {\bf When $p \in \mathcal{S}_1$: } In this case, we have that 
$P(\epstail{\sigma}) \geq 2 \sigma.$
Under the alternate we have that 
$T_1(\sigma) \sim \text{Poi}(nP(\epstail{\sigma})) - nP_0(\epstail{\sigma})$.
This yields,
\begin{align}
\label{eqn:evetail}
P(\phi_{\text{tail}} = 0) \leq  P\left( \text{Poi}(nP(\epstail{\sigma})) < \rho n P(\epstail{\sigma})\right),
\end{align}
where 
\begin{align*}
\rho = \frac{P_0(\epstail{\sigma})}{P(\epstail{\sigma})} +  \frac{1}{P(\epstail{\sigma})} \sqrt{ \frac{P_0(\epstail{\sigma})}{n \alpha}}.
\end{align*}
Provided $\rho \leq 1$ we obtain via Chebyshev's inequality that,
\begin{align*}
P(\phi_{\text{tail}} = 0) \leq \frac{1}{n(1 - \rho)^2 P_1(\epstail{\sigma})}.
\end{align*}
We further have that,
\begin{align*}
\rho \leq \frac{1}{2} \left[ 1 +  \frac{1}{ \sqrt{n \alpha \sigma} }\right].
\end{align*}
Under the conditions that,
\begin{align*}
n \geq \frac{4}{ \alpha \sigma},
\end{align*}
we obtain that $\rho \leq 1/2$, which yields that,
\begin{align*}
P(\phi_{\text{tail}} = 0) \leq \frac{2}{n  \sigma} \leq \zeta,
\end{align*}
where the final inequality uses the condition on $n$.

\noindent {\bf When $p \in \mathcal{S}_2$: } In this case, we first observe that the 
bulk deviation must be sufficiently large.
Concretely, at most $\critradn/2$ deviation can occur in the largest
element and at most $3 \sigma$ occurs in the tail, i.e.
\begin{align*}
\Delta_{\mathcal{B}_{\sigma}} \geq \frac{\critradn}{2} - 3 \sigma \geq \frac{\critradn}{8}.
\end{align*}
Our next goal will be to upper bound the test threshold, $t_2(\alpha/2,\sigma)$. In particular,
we claim that,
\begin{align}
\label{eqn:threshold_bound}
t_2(\alpha/2,\sigma) \leq \sqrt{\frac{2 \text{Var}(T_2(\sigma))}{\alpha}} 
\end{align}
Taking this claim as given for now and supposing that our sample size can
be written as
$n = c \max \left\{ \frac{T_{\sigma}(p_0)^{1/3}}{p_{\min,\sigma}^{1/3} \Delta_{\mathcal{B}_{\sigma}}}, \frac{ T_{\sigma}(p_0) }{ \Delta_{\mathcal{B}_{\sigma}}^2 }\right\},$ for some $c \geq 1$,
we can use Lemma~\ref{lem:valiant_main} and Chebyshev's inequality to obtain that, 
\begin{align*}
P(\phi_{2/3} = 0) \leq \frac{1}{ (\sqrt{\frac{c}{16}} - \sqrt{\frac{2}{\alpha}})^2},
\end{align*}
provided that $\sqrt{\frac{c}{16}} \geq \sqrt{\frac{2}{\alpha}}.$
Thus, it suffices to ensure that,
\begin{align*}
n \geq 64 \max\left\{ \frac{2}{\alpha}, \frac{1}{\zeta} \right\} \max \left\{ \frac{T_{\sigma}(p_0)^{1/3}}{p_{\min,\sigma}^{1/3} \Delta_{\mathcal{B}_{\sigma}}}, \frac{ T_{\sigma}(p_0) }{ \Delta_{\mathcal{B}_{\sigma}}^2 }\right\},
\end{align*}
to obtain that $P(\phi_{2/3} = 0) \leq \zeta$ as desired. Using Lemma~\ref{lem:nrelate}, we have that this holds under the condition on $n$. It remains to verify the claim in~\eqref{eqn:threshold_bound}. In order to do so we just note that the variance of the statistic is minimized at the null, i.e.
\begin{align*}
\text{Var}(T_2(\sigma)) \geq \sum_{i \in \mathcal{B}_{\sigma}} 2n^2 p_0(i)^{2/3} = \frac{\alpha t_2^2(\alpha/2,\sigma)}{2}.
\end{align*}
as desired.

\subsection{Proof of Theorem~\ref{thm:max}}
Fix any multinomial $p$ on $[d]$, and 
suppose we denote $\Delta_i = p_0(i) - p(i)$, then a straightforward computation shows that,
\begin{align}
\label{eqn:altmean}
\mathbb{E}_p[T_j] &= n^2 \sum_{t \in S_j} \Delta_t^2, 
\end{align}
\begin{align}
\label{eqn:altvar}
\mathrm{Var}_p[T_j] &= \sum_{t \in S_j}  \left[ 2n^2 p_0(t)^2 + 2n^2 \Delta_t^2 - 4 n^2 \Delta_t p_0(t) + 4n^3 \Delta_t^2 p_0(t) - 4n^3 \Delta_t^3 \right].
\end{align}
This in turn yields that the null variance of $T_j$ is simply $\mathrm{Var}_0[T_j] = 2n^2 \sum_{t \in S_j}  p(t)^2.$ By Chebyshev's inequality we then obtain that:
\begin{align*}
P_0(T_j > t_j) \leq \alpha/k,
\end{align*}
which together with the union bound yields,
\begin{align*}
P_0(\maxtest = 1) \leq \alpha.
\end{align*}
As in the proof of Theorem~\ref{thm:valup} we consider two cases: when $p \in \mathcal{S}_1$ and when $p \in \mathcal{S}_2.$ Since the composite test includes the tail test, the analysis of the case when $p \in \mathcal{S}_1$ is identical to before.  
Now, we consider the case when $p \in \mathcal{S}_2$.

We have further partitioned the bulk of the distribution into at most $k$ sets, so that at least
one of the sets $S_j$ must witness a discrepancy of at least $\critradn/(8k)$, i.e. when
$p \in \mathcal{S}_2$ we have that,
\begin{align*}
\sup_{j} \sum_{i \in S_j} |p_0(i) - p(i)| \geq \frac{\critradn}{8k}.
\end{align*}
Let $j^*$ denote the set that witnesses this discrepancy. 
We focus the rest of the proof on this fixed set $S_{j^*}$ and show that under the alternate $T_{j^*} > 
t_{j^*}$ with sufficiently high probability. Suppose that for $j^*$ we can verify the following two conditions:
\begin{align}
\label{eqn:condonemult2}
t_{j^*} &\leq \frac{\mathbb{E}_p[T_{j^*}]}{2} \\
\label{eqn:condtwomult2}
\mathbb{E}_p[T_{j^*}] &\geq2 \sqrt{\frac{\mathrm{Var}_p[T_{j^*}]}{\zeta}},
\end{align}
then we obtain that $P(\maxtest = 0) \leq \zeta.$ To see this, observe that
\begin{align*}
P(\maxtest = 0) &\leq P( T_{j^*} - \mathbb{E}_p[T_{j^*}] < t_{j^*} - \mathbb{E}_p[T_{j^*}]) \\
&\leq P( (T_{j^*} - \mathbb{E}_p[T_{j^*}])^2 < (t_{j^*} - \mathbb{E}_p[T_{j^*}])^2) \\
&\leq \frac{\mathrm{Var}_p[T_{j^*}]}{ (t_{j^*} - \mathbb{E}_p[T_{j^*}])^2} \leq \frac{4 \mathrm{Var}_p[T_{j^*}]}{  \mathbb{E}_p[T_{j^*}]^2} \leq \zeta. 
\end{align*}
Consequently, we focus the rest of the proof on showing the above two conditions.
We let $d_{j^*}$ denote the size of $S_{j^*}$.

\noindent {\bf Condition in Equation~\eqref{eqn:condonemult2}: } Observe that,
\begin{align}
\label{eqn:lbxx}
\sum_{i \in S_{j^*}} \Delta_i^2 \geq \frac{\left(\sum_{i \in S_{j^*}} |\Delta_i|\right)^2}{d_{j^*}} \geq
\frac{\critradn^2}{64 k^2 d_{j^*}}.
\end{align}
Using Equations~\eqref{eqn:thresh} and~\eqref{eqn:altmean}, it suffices to check that,
\begin{align*}
n \sqrt{ \frac{2k \sum_{i \in S_{j^*}} p_0(i)^2}{\alpha} } \leq n^2 \sum_{i \in S_{j^*}} \Delta_j^2,
\end{align*}
and applying the lower bound in Equation~\eqref{eqn:lbxx} it suffices if,
\begin{align*}
\frac{\critradn^2}{64 k^2 d_{j^*}} \geq \frac{1}{n} \sqrt{ \frac{2k \sum_{i \in S_{j^*}} p_0(i)^2}{\alpha} }. 
\end{align*}
Denote the maximum and minimum entry of the multinomial on $S_{j^*}$ as $b_{j^*}$ and $a_{j^*}$ 
respectively. Noting that on each bin the multinomial is roughly uniform one can further observe that,
\begin{align*}
d_{j^*} \sqrt{  \sum_{i \in S_{j^*}} p_0(i)^2} \leq d_{j^*}^{3/2} b_{j^*} \leq 2 d_{j^*}^{3/2} a_{j^*} \leq
2 \truncnorm{\critradn/8}{p_0}.
\end{align*}
This yields that the first condition is satisfied if,
\begin{align*}
\critradn^2 \geq \frac{256 k^{5/2}}{n} \frac{\truncnorm{\critradn/8}{p_0} }{\sqrt{\alpha} },
\end{align*}
which is indeed the case.

\noindent {\bf Condition in Equation~\eqref{eqn:condtwomult2}: } We proceed by upper bounding the variance under the alternate. Using Equation~\eqref{eqn:altvar} we have,
\begin{align*}
\mathrm{Var}_p[T_{j^*}] &= \sum_{t \in S_{j^*}}  \left[ 2n^2 p_0(t)^2 + 2n^2 \Delta_t^2 - 4 n^2 \Delta_t p_0(t) + 4n^3 \Delta_t^2 p_0(t) - 4n^3 \Delta_t^3 \right] \\
&\leq  \sum_{t \in S_{j^*}} \left[ 4n^2 p_0(t)^2 + 4n^2 \Delta_t^2+ 4n^3 \Delta_t^2 p_0(t) - 4n^3 \Delta_t^3 \right] \\
&\leq \underbrace{4 n^2 b_{j^*}^2 d_{j^*}}_{U_1} +  \underbrace{4n^2 \sum_{t \in S_{j^*}}  \Delta_t^2 }_{U_2} +  \underbrace{4n^3 b_{j^*}  \sum_{t \in S_{j^*}}  \Delta_t^2 }_{U_3} \underbrace{ - 4n^3 \sum_{t \in S_{j^*}} \Delta_t^3}_{U_4} .
\end{align*}
In order to check the desired condition, it suffices to verify that
\begin{align*}
\mathbb{E}_p[T_{j^*}] \geq 8 \sqrt{\frac{U_i}{\zeta}}, 
\end{align*}
for each $i \in \{1,2,3,4\}$. We consider these tasks in sequence. For the first term we obtain
that it suffices if, 
\begin{align*}
 \sum_{i \in S_{j^*}} \Delta_i^2  \geq \left(\frac{16 b_{j^*} d_{j^*}^{1/2} } {n \sqrt{ \zeta}}\right),
\end{align*}
and applying the lower bound in Equation~\eqref{eqn:lbxx}, and from some straightforward algebra it is sufficient to ensure that,
\begin{align*}
\critradn^2 \geq \frac{2048 k^2 \truncnorm{\critradn/8}{p_0} }{ n \sqrt{\zeta}},
\end{align*}
which is indeed the case. For the second term, some simple algebra yields that it suffices to have that,
\begin{align}
\label{eqn:boundtemp}
 \sum_{i \in S_{j^*}} \Delta_i^2 \geq \left( \frac{ 144}{ n^2 \zeta } \right).
\end{align}
In order to establish this, we need to appropriately lower bound $n$. Let $p_{\min}$ denote
the smallest entry in $\bulk{\critradn/8}$. For a sufficiently large universal constant $C > 0$, let us denote:
\begin{align*}
\theta_{k,\alpha} := C k^2 \left[ \sqrt{ \frac{k}{\alpha}} + \frac{1}{\zeta} \right].
\end{align*}
Then using the lower bound on $\critradn$ we obtain,
\begin{align*}
n \geq \frac{\theta_{k,\alpha} \truncnorm{\critradn/16}{p_0}}{\critradn^2} = 
\frac{\theta_{k,\alpha} \truncnorm{\critradn/16}{p_0}^{1/3} \left[ \sum_{i \in \bulk{\critradn/16}} p_0(i)^{2/3}\right]}{\critradn^2}.
\end{align*}
Now denote $B = \bulk{\critradn/16} \backslash \bulk{\critradn/8}$, then we have that,
\begin{align*}
p_{\min} + \sum_{i \in B} p_i \geq \critradn/16,
\end{align*}
so that,
\begin{align*}
\sum_{i \in \bulk{\critradn/16}} p_0(i)^{2/3} \geq \sum_{i \in B} p_0(i)^{2/3} + p_{\min}^{2/3} = \frac{1}{p_{\min}^{1/3}} \left[ \sum_{i \in B} p_0(i)^{2/3} p_{\min}^{2/3} + p_{\min} \right] \geq \frac{\critradn}{16 p_{\min}^{1/3}}.
\end{align*}
This gives the lower bound, 
\begin{align*}
n \geq \frac{\theta_{k,\alpha} \truncnorm{\critradn/16}{p_0}^{1/3}}{16 \critradn p_{\min}^{1/3}} \geq
\frac{\theta_{k,\alpha} \left(\sum_{t \in S_{j^*}} p_0(t)^{2/3} \right)^{1/2} }{16 \critradn p_{\min}^{1/3}} \geq
\frac{\theta_{k,\alpha} \sqrt{d_{j^*}}  }{16 \critradn}.
\end{align*}
Returning to the bound in Equation~\eqref{eqn:boundtemp}, and using the lower bound in Equation~\eqref{eqn:lbxx} we obtain that it suffices to ensure that
\begin{align*}
\frac{\critradn}{8 k \sqrt{d_{j^*}}} \geq \left( \frac{192 \critradn}{\theta_{k,\alpha} \sqrt{d_{j^*}} \sqrt{ \zeta} } \right),
\end{align*}
which is indeed the case. Turning our attention to the term involving $U_3$ we have,
that by some simple algebra it suffices to verify that,
\begin{align*}
\sum_{i \in S_{j^*}} \Delta_i^2 \geq \left( \frac{ 144 b_{j^*}}{ n \zeta } \right).
\end{align*}
Using the lower bound in Equation~\eqref{eqn:lbxx} we obtain that it is sufficient to ensure,
\begin{align*}
\frac{\critradn^2}{64 k^2 d_{j^*} } \geq  \left( \frac{ 144 b_{j^*}}{ n \zeta } \right),
\end{align*}
and with the observation that $d_{j^*} b_{j^*} \leq 2 d_{j^*}^{3/2} a_{j^*} \leq 2 \truncnorm{\critradn/8}{p_0}$ we obtain,
\begin{align*}
\critradn^2 \geq  \left( \frac{ 18432 k^2 \truncnorm{\critradn/8}{p_0}}{ n \zeta } \right),
\end{align*}
which is indeed the case. Finally, we turn our attention to the term involving $U_4$. In this case 
we have that it suffices to show that,
\begin{align*}
n^{1/2} \sum_{i \in S_{j^*}} \Delta_i^2 \geq 16 \sqrt{ \frac{\sum_{i \in S_{j^*}} \Delta_i^3 }{\zeta} },
\end{align*}
by the monotonicity of the $\ell_p$ norm it suffices then to show that,
\begin{align*}
n^{1/2} \sum_{i \in S_{j^*}} \Delta_i^2 \geq 16 \sqrt{ \frac{ \left[ \sum_{i \in S_{j^*}} \Delta_i^2 \right]^{3/2} }{\zeta} },
\end{align*}
and after some simple algebra this yields that it suffices to have,
\begin{align*}
\sum_{i \in S_{j^*}} \Delta_i^2 \geq \frac{(16)^4}{\zeta^2 n^2},
\end{align*}
and this follows from an essentially identical argument to the one handling the term involving $U_2$. This completes the proof.

\section{Proofs for Examples of Lipschitz Testing}
\label{app:ex}
In this Section we provide proofs of the claims in Section~\ref{sec:ex}. For convenience, we restate
all the claims in the following lemma.
\begin{lemma}
\label{lem:ex}
\begin{itemize}
\item Suppose that $p_0$ is a standard one-dimensional Gaussian, with mean $\mu$, and variance $\nu^2$, then we have that:
\begin{align}
\label{eqn:exclaim1}
T_0(p_0) = (8 \pi)^{1/2} \nu.
\end{align}
\item Suppose that $p_0$ is a Beta distribution with parameters $\alpha, \beta$. Then we have,
\begin{align}
\label{eqn:exclaim2}
T_0(p_0) = \left(\int_0^1 \sqrt{p_0(x)} dx\right)^2 = \frac{B^2((\alpha + 1)/2, (\beta + 1)/2)}{ B(\alpha,\beta) },
\end{align}
where $B: \mathbb{R}^2 \mapsto \mathbb{R}$ is the Beta function.
Furthermore, if we take $\alpha = \beta = t \geq 1,$ then we have that,
\begin{align}
\label{eqn:exclaim3}
 \frac{\pi^2}{4e^4} t^{-1/2}  \leq T_0(p_0) \leq \frac{e^4}{4} t^{-1/2}.
\end{align}
\item Suppose that $p_0$ is Cauchy with parameter $\alpha$, then we have that,
\begin{align}
\label{eqn:exclaim4}
T_0(p_0) = \infty.
\end{align}
Furthermore, if $0 \leq \sigma \leq 0.5$ then,
\begin{align}
\label{eqn:exclaim5}
 \frac{4\alpha}{\pi}  \left[  \ln^2 \left( \frac{1}{\sigma}  \right)  \right] \leq T_{\sigma}(p_0) \leq  \frac{4\alpha}{\pi}  \left[  \ln^2 \left( \frac{2 e}{\pi \sigma}  \right)  \right].
\end{align}
\item Suppose that $p_0$ has a Pareto distribution with parameter $\alpha$ then we have that,
\begin{align}
\label{eqn:exclaim6}
T_0(p_0) = \infty,
\end{align}
while the truncated $T$-functional satisfies:
\begin{align}
\label{eqn:exclaim7}
 \frac{4 \alpha x_0}{(1 - \alpha)^2} \left( \sigma^{- \frac{1 - \alpha}{2\alpha}} - 1\right)^2 = T_\sigma(p_0) \leq  \frac{4 \alpha x_0}{(1 - \alpha)^2} \sigma^{- \frac{1 - \alpha}{\alpha}}.
\end{align}
\end{itemize}
\end{lemma}

\begin{proof}
Notice that Claims~\eqref{eqn:exclaim4} and~\eqref{eqn:exclaim6} follow by taking $\sigma \rightarrow 0$ in Claims~\eqref{eqn:exclaim5} and~\eqref{eqn:exclaim7} respectively.
We prove the remaining claims in turn.

\noindent {\bf Proof of Claim~\eqref{eqn:exclaim1}: } Observe that,
\begin{align*}
T_0(p_0) &= \frac{1}{\sqrt{2\pi} \nu} \left(\int_{-\infty}^{\infty} \exp( - (x - \mu)^2/(4\nu^2)) dx \right)^2 \\
&= \frac{1}{\sqrt{2\pi} \nu} 4\pi \nu^2 \\
&= \sqrt{8 \pi} \nu.
\end{align*}

\noindent {\bf Proof of Claim~\eqref{eqn:exclaim2}: } The Beta density can be written 
as:
\begin{align*}
p_0(x) = \frac{\Gamma(\alpha + \beta)}{\Gamma(\alpha) \Gamma(\beta)} x^{\alpha - 1} x^{\beta - 1} =
\frac{1}{B(\alpha,\beta)} x^{\alpha - 1} x^{\beta - 1},
\end{align*}
where $\Gamma: \mathbb{R} \mapsto \mathbb{R}$ denotes the Gamma function. Some simple algebra yields that the $T$-functional is simply:
\begin{align}
\label{eqn:betat}
T_0(p_0) = \int_0^1 \sqrt{p_0(x)} dx = \frac{B((\alpha + 1)/2, (\beta + 1)/2)}{ \sqrt{ B(\alpha,\beta) }}.
\end{align}

\noindent {\bf Proof of Claim~\eqref{eqn:exclaim3}: } We now take $\alpha = \beta = t \geq 1$ in the above expression. To prove the claim we use standard approximations to the Beta function derived using Stirling's formula. Recall, that by Stirling's formula we have that:
\begin{align*}
\sqrt{2 \pi n} \left( \frac{n}{e} \right)^n \leq n! \leq e \sqrt{n} \left( \frac{n}{e} \right)^n.
\end{align*}
We begin by upper bounding the Beta function for integers $\alpha, \beta \geq 0$:
\begin{align*}
B(\alpha, \beta) &= \frac{\Gamma(\alpha) \Gamma(\beta) }{\Gamma(\alpha + \beta)} = \frac{ (\alpha - 1)! (\beta - 1)! } { (\alpha + \beta - 1)! } = \frac{\alpha ! \beta ! }{ (\alpha + \beta)! } \frac{ \alpha + \beta}{\alpha \beta} \\
&\leq \frac{e^2}{\sqrt{2\pi}} \frac{ \alpha + \beta}{\alpha \beta}
\frac{ \sqrt{\alpha\beta } \alpha^{\alpha} \beta^{\beta} \exp (\alpha + \beta) }{\sqrt{\alpha + \beta} (\alpha + \beta)^{\alpha + \beta} \exp (\alpha + \beta) } \\
&= \frac{e^2}{\sqrt{2\pi}} \sqrt{ \frac{ \alpha + \beta}{\alpha \beta}} \frac{ \alpha^{\alpha} \beta^{\beta} }{(\alpha + \beta)^{\alpha + \beta}}.  
\end{align*}
Now, setting $\alpha = \beta = t \geq 1,$ we obtain:
\begin{align*}
B(t,t) \leq \frac{e^2}{\sqrt{\pi}} \frac{ 2^{-2t} }{ \sqrt{t}}.
\end{align*}
We can similarly lower bound the Beta function as:
\begin{align*}
B(t,t) \geq \frac{2\sqrt{2} \pi}{e} \frac{ 2^{-2t} }{ \sqrt{t}}.
\end{align*}
We also need to bound the Beta function at certain 
non-integer values. In particular,
we observe that,
\begin{align*}
B(t+1,t+1) \leq B(t + 1/2, t + 1/2) \leq B(t,t),
\end{align*}
so that we can similarly sandwich the Beta function at these non-integer values as:
\begin{align*}
\frac{2 \pi}{4e} \frac{ 2^{-2t} }{ \sqrt{t}}.\leq B(t + 1/2, t + 1/2) \leq \frac{e^2}{\sqrt{\pi}} \frac{ 2^{-2t} }{ \sqrt{t}}.
\end{align*}
With these bounds in place we can now upper and lower bound the $T$-functional in~\eqref{eqn:betat}. We can upper bound this expression by considering the cases when $t$ is odd and $t$ is even separately, and taking the worse of these two bounds to obtain:
\begin{align*}
T(p_0) \leq \frac{e^2}{2} t^{-1/4}.
\end{align*}
Similarly, using the above results we can lower bound the $T$-functional as:
\begin{align*}
T(p_0) \geq \frac{\pi}{2e^2} t^{-1/4},
\end{align*}
and this yields the claim.


\noindent {\bf Proof of Claim~\eqref{eqn:exclaim5}: } 
We are interested in the truncated $T$-functional. 
The set $B_\sigma$ of probability content $1 - \sigma$, takes the form $[-\alpha,\alpha]$, where
\begin{align*}
\alpha =  \gamma \tan \left( \frac{\pi}{2} (1 - \sigma) \right) = \gamma \cot  \left( \frac{\pi \sigma}{2} \right).
\end{align*}
Using the inequality that $\cot(x) \leq \frac{1}{x},$
we can upper bound $\alpha$ as:
\begin{align*}
\alpha \leq \frac{2 \gamma}{\pi \sigma}.
\end{align*}
Similarly, we can (numerically) lower bound $\alpha$ by noting that for $0 \leq \sigma \leq 0.5$ we have that,
\begin{align*}
\alpha \geq \frac{ \gamma }{4 \sigma}.
\end{align*}
With these bounds in place, we can now proceed to upper and lower bound the truncated $T$ functional. Concretely,
\begin{align*}
T_\sigma(p_0) &\leq   \frac{\gamma}{\sqrt{\pi \gamma}} \int_{-\frac{2 \gamma}{\pi \sigma}}^{ \frac{2 \gamma}{\pi \sigma}}  \frac{1}{\sqrt{x^2 + \gamma^2}} dx \leq  
 \frac{2 \gamma}{\sqrt{\pi \gamma}} \left[ \int_0^\gamma \frac{1}{\gamma} dx + \int_{\gamma}^{\frac{2 \gamma}{\pi \sigma}}  \frac{1}{x} dx \right] \\
&\leq \frac{2 \gamma}{\sqrt{\pi \gamma}} \left[ 1 + \ln \left( \frac{2}{\pi \sigma}  \right)  \right] \\
&= 2 \sqrt{\frac{\gamma}{\pi}} \left[  \ln \left( \frac{2 e}{\pi \sigma}  \right)  \right].
\end{align*}
In a similar fashion, we can lower bound the functional as:
\begin{align*}
T_\sigma(p_0) \geq 2 \sqrt{\frac{\gamma}{\pi}}  \left[  \ln \left( \frac{1}{\sigma}  \right)  \right].
\end{align*}
Taken together these bounds give the desired claim.

\noindent {\bf Proof of Claim~\eqref{eqn:exclaim7}: } We treat $x_0$ as a fixed constant. 
The CDF for 
the Pareto family of distributions takes the simple form:
\begin{align*}
F(x) = 1 - \left(\frac{x_0}{x} \right)^{\alpha},~~~\text{for}~~x \geq x_0,
\end{align*}
we obtain that the set $B_\sigma$ takes the form $[x_0, x_0 \sigma^{-1/\alpha}]$.
So that the truncated functional is simply:
\begin{align*}
T_{\sigma}(p_0) &= \int_{x_0}^{x_0  \sigma^{-1/\alpha}} \sqrt{p_0(x; x_0,\alpha)} dx \\
&= \frac{2 \sqrt{\alpha x_0}}{1 - \alpha} \left( \sigma^{- \frac{1 - \alpha}{2\alpha}} - 1\right),
\end{align*}
which yields the desired claim.
\end{proof}

\section{Properties of the $T$-functional}
\label{app:tfunc}
The rate for 
Lipschitz testing is largely dependent on the truncated
$T$-functional 
of the null hypothesis. 
In this section we establish several properties of the $T$-functional, 
and its stability with respect to perturbations. 
There are two notions of stability of the truncated $T$-functional that are of interest:
its stability with respect to perturbation of the truncation parameter, and its
stability with respect to perturbations of the density $p_0$. 
In particular, the truncation stability determines the discrepancy between the upper and lower bounds in Theorem~\ref{thm:main}.

Our interest is in the difference between $T_{\sigma_1}(p_0)$ and $T_{\sigma_2}(p_0)$ (where without loss of generality we take $\sigma_1 \leq \sigma_2$). We show that if the support of the density is stable with respect to the truncation parameter then so is the $T$-functional. Intuitively, the discrepancy can be large only if the density has a long $\sigma_1$-tail but a relatively small $\sigma_2$-tail. Returning to the definition of the $T$-functional in~\eqref{eqn:tfuncdef}, we let $B_{\sigma_1}$ and $B_{\sigma_2}$ denote the sets that achieve the infimum for $T_{\sigma_1}$ and 
$T_{\sigma_2}$ respectively. These are not typically well-defined 
for two reasons: the set may not be unique and 
the infimum might not be attained. The second problem 
can be easily dealt with by introducing a small amount of
slack. To deal with the non-uniqueness we simply choose the sets that have maximal overlap in Lebesgue measure, i.e.
we define $B_{\sigma_1}$ and $B_{\sigma_2}$ to be two sets that have maximal Lebesgue overlap such that,
\begin{align*}
\Big( \int_{B_{\sigma_1}} \sqrt{p_0(x)} dx \Big)^{2}  &\geq T_{\sigma_1}(p_0) - \xi,\\
\Big( \int_{B_{\sigma_2}} \sqrt{p_0(x)} dx \Big)^{2}  &\geq T_{\sigma_2}(p_0) - \xi,
\end{align*}
for an arbitrary small $\xi > 0$.
The quantity $\xi$ may be taken as small as we like and has no effect when chosen small enough so we ignore it in what follows. We define, $S = B_{\sigma_1} \backslash B_{\sigma_2}$ which measures
the stability of the support with respect to changes in the truncation parameter, i.e. if the Lebesgue
measure $\mu(S)$ is small then the support is stable.
With these definitions in place we have the following lemma:
\begin{lemma}
For any two truncation levels $\sigma_1 \leq \sigma_2$, we have that,
\begin{align*}
T_{\sigma_1}^\gamma(p_0) - T_{\sigma_2}^\gamma(p_0) \leq (\sigma_1 - \sigma_2)^{\gamma} \mu(S)^{1-\gamma}.
\end{align*}
\end{lemma}
\noindent {\bf Remarks: } 
\begin{itemize}
\item Since $\gamma < 1$, this result asserts that if the support of the density is stable with respect
to the truncation parameter then so is the truncated $T$-functional. This is the case in all the examples we considered in Section~\ref{sec:ex}.
\item If we restrict attention to compactly supported densities then we can upper bound $\mu(S)$ by 
the Lebesgue measure of the support indicating that in these cases the truncated $T$-functional is somewhat stable.
\item On the other hand this result also gives insight into when the truncated functional is not stable. In particular, it is straightforward to construct examples of densities $p_0$ which have a very long $\sigma_1$-tail but a light $\sigma_2$-tail, in which case this discrepancy can be arbitrarily large. 
Noting however that in our bounds the regime of interest is when the truncation 
parameter is not fixed, i.e. when $\sigma \rightarrow 0$, in which case this discrepancy can be large only for carefully constructed pathological densities. 
\end{itemize}
%
%
%
%

%
%
%
\begin{proof}
The result follows using H\"{o}lder's inequality:
\begin{align*}
T_{\sigma_1}^\gamma(p_0) - T_{\sigma_2}^\gamma(p_0) &= 
\int_{S} p_0^{\gamma}(x) dx \\
&= \mu(S) \int_{S} \frac{p_0^{\gamma}(x)}{\mu(S)} dx \\
&\leq \mu(S) \left(\int_{S} \frac{p_0(x)}{\mu(S)}dx \right)^{\gamma} \\
&= \mu(S)^{1 - \gamma} (\sigma_1 - \sigma_2)^{\gamma}.
\end{align*}
\end{proof}

In order to understand the stability of the $T$-functional with respect to perturbations of $p_0$ 
it is natural to consider a form of the modulus of continuity. We restrict our attention to densities
$p_0$ which have support contained in a fixed set $S$, and denote these densities by $\mathcal{L}(\lipcons,S)$, and only consider the case when $d = 1$ and hence $\gamma = 1/2.$

Focussing on the case when the truncation parameter is fixed (say to $0$) we define:
\begin{align*}
s(p_0, \tau,S) = \sup_{p, p_0 \in \mathcal{L}(\lipcons,S), \|p - p_0\|_1 \leq \tau} |T_0^{\gamma}(p) - T_0^{\gamma}(p_0)|.
\end{align*}
With these definitions in place, we have the following result:
\begin{lemma}
For any $p_0$, the modulus of continuity of the $T$-functional is upper bounded as:
\begin{align*}
s(p_0,\tau,S) \leq \sqrt{\tau \mu(S)}.
\end{align*}
\end{lemma}
\noindent {\bf Remark: } 
\begin{itemize}
\item This result guarantees that for densities that are close in $\ell_1$, their corresponding $T$-functionals are close, provided that we restrict attention to compactly supported densities. 

\item On the other hand, an inspection of the proof below reveals that if we eliminate the restriction
of compact support, then for any density $p_0$, we can construct a density $p$ that is close in $\ell_1$ but has an arbitrarily large discrepancy in the $T$-functional, i.e. the $T$-functional can be highly unstable to perturbations of $p_0$ if we allow densities with arbitrary support.

\end{itemize}

\begin{proof}
Notice that,
\begin{align*}
T_0^{\gamma}(p) - T_0^{\gamma}(p_0) &= \int_S (\sqrt{p(x)} - \sqrt{p_0(x)}) dx \\
&= \mu(S) \int_S (\sqrt{p(x)} - \sqrt{p_0(x)}) \frac{1}{\mu(S)} dx \\
&\leq \mu(S) \sqrt{ \int_S  (\sqrt{p(x)} - \sqrt{p_0(x)})^2 \frac{1}{\mu(S)} dx } \\
&\stackrel{\text{(i)}}{\leq} \sqrt{\mu(S) \|p - p_0\|_1} \leq \sqrt{ \tau \mu(S) },
\end{align*}
where (i) uses the fact that the Hellinger distance is upper bounded by the $\ell_1$ distance.
\end{proof}

\subsection{Proof of Claim~\eqref{eqn:claim}}
This claim is a straightforward consequence of H\"older's inequality. We have that,
\begin{align*}
T_{\sigma}(p_0) =  \inf_{B \in {\cal B}_\sigma}  \left(
\int_{B} p_0^\gamma(x) dx\right)^{1/\gamma}. 
\end{align*}
We restrict our attention to densities with support contained in a fixed set $S$. We let $B_{\sigma}$ denote an arbitrary set in $\mathcal{B}_{\sigma}$ that minimizes the above integral (dealing with non-uniqueness as before). Then,
\begin{align*}
T^{\gamma}_{\sigma}(p_0) &=  \mu(B_{\sigma}) \int_{B_{\sigma}} \frac{p_0^{\gamma}(x)}{\mu(B_{\sigma})} dx \\
&\stackrel{\text{(i)}}{\leq} \mu(B_{\sigma}) \left( \int_{B_{\sigma}} \frac{p_0(x)}{\mu(B_{\sigma})} dx \right)^{\gamma} \\
&= \mu(B_{\sigma})^{1 - \gamma} (1 - \sigma)^{\gamma} \leq \mu(S)^{1 - \gamma} (1 - \sigma)^{\gamma}
\end{align*}
where (i) uses H\"older's inequality. This yields the claim. 
For the uniform distribution $u$ on the set $S$ we have that for any set $B_{\sigma}$ of mass $1 - \sigma$,
\begin{align*}
T_{\sigma}^{\gamma}(u) &= \int_{B_{\sigma}}\frac{1}{\mu^{\gamma}(S)} dx  \\
&= \mu(S)^{1 - \gamma} (1 - \sigma),
\end{align*}
which matches the  result of~\eqref{eqn:claim} up to constant factors involving $\sigma$ and $\gamma$. In particular, our interest is in the regime when $\sigma \rightarrow 0$, and $\gamma$ is a constant, in which case the two quantities are equal.

\section{Technical Results for Lipschitz Testing}
\label{app:lip}
In this section we provide the remaining technical proofs related the Theorem~\ref{thm:main}. We begin with the preliminary Lemmas~\ref{lemma:ingster} and~\ref{lem:rho}.

\subsection{Preliminaries}

\subsubsection{Proof of Lemma~\ref{lemma:ingster}}
Let ${\cal A}$ be all sets $A$ such that $P_0^n(A) \leq \alpha$.
Now
\begin{align*}
\zeta_n({\cal P}) & \geq \inf_\phi Q(\phi=0) \geq 1-\alpha - \sup_{A\in {\cal A}}|Q(A)-P_0^n(A)|\\
& \geq 1-\alpha - \sup_{A}|Q(A)-P_0^n(A)|\\
&= 1-\alpha - \frac{1}{2} \|Q- P_0^n\|_1.
\end{align*}
Note that
\begin{align*}
\|Q- P_0^n\|_1 &=
\mathbb{E}_0 | \likrat(Z_1,\ldots,Z_n) - 1| \leq \sqrt{\mathbb{E}_0 [\likrat^2(Z_1,\ldots,Z_n)] - 1}.
\end{align*}
The result then follows from (\ref{eq:L}).

\subsubsection{Proof of Lemma~\ref{lem:rho}}
We divide the proof into several claims.

\noindent {\bf Claim 1: Each $p_\eta$ is a density function.}
Note that \begin{align*}
\int p_\eta(x) dx = 1.
\end{align*}
Now we show it is non-negative.
Let $x\in A_j$.
Then
\begin{align*}
p_\eta(x) &= p_0(x) + \rho_j  \eta_j \psi_j(x) \geq 
p_0(x) -\rho_j  \psi_j(x)\\
&\geq p_0(x) -  \frac{\rho_j}{c_j^{d/2}h_j^{d/2}} \|\psi\|_{\infty}.
\end{align*}
Now, we observe that for each piece of our partition we have that,
\begin{align*}
p_0(x) \geq \frac{p_0(x_j)}{2} \geq \frac{\lipcons \alpha \beta \sqrt{\dimension} h_j}{2} \geq \lipcons\sqrt{\dimension} h_j, 
\end{align*}
where we use the fact that $\alpha \beta \geq 2$. We then obtain that it suffices to choose,
\begin{align*}
\rho_j \leq  \frac{\lipcons c_j^{d/2}}{\coninfty}  h_j^{d/2 + 1},
\end{align*}
which is ensured by the condition in Equation~\eqref{eqn:condone}.

\vspace{1cm}

\noindent {\bf Claim 2:
Each $p_\eta \in \mathcal{L}(\lipcons)$.}
Let $x,y$ be two points, and that $x \in A_j, y \in A_k$. We consider two
cases: when neither of $j,k$ are $\infty$, and when at least one of them is. Noting that we do not perturb $A_{\infty}$ the second case follows from a similar argument to that of the first case. 
In the first case, we have that:
\begin{align*}
|p_\eta(y) - p_\eta(x)| &\leq |p_0(x) - p_0(y)| + \left\vert \frac{\rho_k \eta_k}{c_k^{d/2} h_k^{d/2}} \psi \left( \frac{y - x_k}{c_k h_k} \right) -   \frac{\rho_j \eta_j}{c_j^{d/2} h_j^{d/2}} \psi \left( \frac{x - x_j}{c_j h_j} \right)\right\vert \\
& \leq \intconst \lipcons \|x - y\| + \left\vert \frac{\rho_k \eta_k}{c_k^{d/2} h_k^{d/2}} \psi \left( \frac{y - x_k}{c_k h_k} \right) -   \frac{\rho_k \eta_k}{c_k^{d/2} h_k^{d/2}} \psi \left( \frac{x - x_k}{c_k h_k} \right)\right\vert \\
& + \left\vert \frac{\rho_j \eta_j}{c_j^{d/2} h_j^{d/2}} \psi \left( \frac{y - x_j}{c_j h_j} \right) -   \frac{\rho_j \eta_j}{c_j^{d/2} h_j^{d/2}} \psi \left( \frac{x - x_j}{c_j h_j} \right)\right\vert \\
&\leq \intconst \lipcons \|x - y\| + \frac{\rho_k \|\psi^{\prime}\|_{\infty} \|x - y\|}{c_k^{d/2 + 1} h_k^{d/2 + 1}} +
\frac{\rho_j \|\psi^{\prime}\|_{\infty} \|x - y\|}{c_j^{d/2 + 1} h_j^{d/2 + 1}},
\end{align*}
so that it suffices to ensure that for $i \in \{1,\ldots,\smallN\}$,
\begin{align*}
\rho_i \leq \frac{(1 - \intconst) \lipcons c_i^{d/2 + 1} h_i^{d/2 + 1}}{ 2 \|\psi^{\prime}\|_{\infty}},
\end{align*}
which is ensured by the condition in~\eqref{eqn:condone}.

\vspace{1cm}

\noindent {\bf Claim 3.
$\int |p_0 - p_\eta| \geq \epsilon$.}
We have
\begin{align*}
\int | p_0 - p_\eta| &= \sum_j \int_{A_j} |p_0 - p_\eta| =
\sum_j \int_{A_j} |\rho_j \eta_j \psi_j|\\
&=
\sum_j \rho_j \int_{A_j} |\psi_j| =
\sum_j \rho_j \int_{A_j} \frac{1}{c_j^{d/2} h_j^{d/2}} \left\vert\psi \left(\frac{x-x_j}{c_j h_j}\right)\right\vert\\
&=
\sum_j \rho_j c_j^{d/2} h_j^{d/2}\int_{[-1/2,1/2]^d} |\psi|=
c_1 \sum_j \rho_j c_j^{d/2} h_j^{d/2} \geq \epsilon,
\end{align*}
where we use the condition in~\eqref{eqn:condtwo}.
Taken together claims 1, 2 and 3 show that $p_\eta\in \mathcal{L}(\lipcons)$ and that
$\|p_\eta- p_0\|_1 \geq \critradn$.

\vspace{1cm}

\noindent {\bf Claim 4: Likelihood ratio bound.}
For observations $\{Z_1,\ldots,Z_n\}$ the likelihood ratio is given as
$$
\likrat(Z_1,\ldots,Z_n) = \frac{1}{2^N}\sum_{\eta \in \{-1,1\}^N} \prod_i \frac{p_\eta(Z_i)}{p_0(Z_i)}
$$
and
\begin{align*}
\likrat^2(Z_1,\ldots,Z_n) &= \frac{1}{2^{2\smallN}}\sum_{\eta \in \{-1,1\}^N} \sum_{\nu \in \{-1,1\}^N}   \prod_i \frac{p_\eta(Z_i)p_\nu(Z_i)}{p_0(Z_i)p_0(Z_i)}\\
&= \frac{1}{2^{2\smallN}}\sum_{\eta \in \{-1,1\}^N} \sum_{\nu \in \{-1,1\}^N}    \prod_i 
\left(1 + \frac{\sum_{j=1}^{\smallN} \rho_j \eta_j \psi_j(Z_i)}{p_0(Z_i)}\right)
\left(1 + \frac{\sum_{j=1}^{\smallN} \rho_j \nu_j \psi_j(Z_i)}{p_0(Z_i)}\right).
\end{align*}
Taking the expected value over $Z_1,\ldots, Z_n$, and using the fact that the $\psi_j$s have disjoint support we obtain
\begin{align*}
E_0[\likrat^2(Z_1,\ldots,Z_n)] &=
 \frac{1}{2^{2\smallN}}\sum_{\eta \in \{-1,1\}^N} \sum_{\nu \in \{-1,1\}^N} \left(1+  \sum_{j=1}^{\smallN} \rho_j^2 \eta_j \nu_j a_j \right)^n \\
 &\leq
\frac{1}{2^{2\smallN}}\sum_{\eta \in \{-1,1\}^N} \sum_{\nu \in \{-1,1\}^N}
\exp\left( n  \sum_j \rho_j^2 \eta_j \nu_j a_j\right)
\end{align*}
where
\begin{align*}
a_j &= \int_{A_j} \frac{\psi_j^2(z)}{p_0(z)} dz =
\frac{1}{p_0(z_j)}\int_{A_j} \psi_j^2(z) \frac{p_0(z_j)}{p_0(z)} dz \\
& \leq \frac{2}{p_0(z_j)}.
\end{align*}
Thus
$E_0[\likrat^2(Z_1,\ldots,Z_n)] \leq E_{\eta,\nu} e^{n \langle \eta,\nu\rangle}$
where we use the weighted inner product defined as:
\begin{align*}
\langle \eta,\nu\rangle := \sum_j \rho_j^2 \eta_j \nu_j a_j.
\end{align*}
Hence,
\begin{align*}
E_0[\likrat^2(Z_1,\ldots,Z_n)] & \leq E_{\eta,\nu} e^{n \langle \eta,\nu\rangle} = 
\prod_j E e^{n \eta_j \nu_j}\\
&=
\prod_j {\rm cosh}(n \rho_j^2 a_j) \leq
\prod_j ( 1+ n^2 \rho_j^4 a_j^2) \leq
\prod_j \exp (n^2 \rho_j^4 a_j^2)\\
&= \exp\left\{ \sum_j n^2 \rho_j^4 a_j^2 \right\} \leq
\exp\left\{ 4 n^2 \sum_j \frac{\rho_j^4}{p_0^2(x_j)}\right\} \leq C_0,
\end{align*}
where the final inequality uses the condition in~\eqref{eqn:condthree}.
From Lemma \ref{lemma:ingster},
it follows that the Type II error of any test is at least $\delta$.

\subsection{Further technical preliminaries}
Our analysis of the pruning in Algorithm~\ref{alg:two} uses various results that we provide in this section.
\begin{lemma}
\label{lemma::smallest}
Let $P$ be a distribution with density $p$
and let $\gamma\in [0,1)$.
Let
$$
A = \{x:\ p(x) \geq t\}
$$
for some $t$.
Define $\theta = P(A)$.
Finally, let
${\cal B} = \{B:\ P(B) \geq \theta\}$.
Then, for every $B\in {\cal B}$,
$$
\int_A p^\gamma(x) dx \leq \int_B p^\gamma(x) dx.
$$
\end{lemma}

\begin{proof}
Let
$$
S_1 = A \bigcap B^c,\ \ \ S_2 = A^c \bigcap B.
$$
Then
$$
\int_A p^\gamma(x) dx - \int_B p^\gamma(x) dx = \int_{S_1} p^\gamma(x) dx - \int_{S_2} p^\gamma(x) dx.
$$
So it suffices to show that
$\int_{S_1} p^\gamma(x) dx \leq \int_{S_2} p^\gamma(x) dx.$
Note that:
\begin{enumerate}
\item $S_1$ and $S_2$ are disjoint,
\item $\inf_{y\in S_1}p(y)   \geq \sup_{y\in S_2}p(y)$ and
\item $\int_{S_1} p(x) dx \leq \int_{S_2} p(x) dx$.
\end{enumerate}
where the last fact follows since $\int_A p(x) dx \leq \int_{B} p(x) dx$.
Thus,
letting $g(x) = 1/p^{1-\gamma}(x)$, we have that
$$
g(x) \leq g(y)
$$
for all $x\in S_1$ and $y\in S_2$.
So
\begin{align*}
\int_{S_1} p^\gamma(x) dx &= \int_{S_1} p(x) g(x) dx \leq \sup_{x\in S_1}g(x) \int_{x\in S_1} p(x) dx \\
& \leq \sup_{x\in S_1}g(x) \int_{x\in S_2} p(x) dx \leq \inf_{x\in S_2}g(x) \int_{x\in S_2} p(x) dx\\
& \leq \int_{S_2} p(x) g(x) dx = \int_{S_2} p^\gamma(x) dx.
\end{align*}
\end{proof}

\noindent The following lemma concerns the optimal truncation of a piecewise constant function. Suppose we have a piecewise constant positive function $f$, which is constant on the partition
$\{A_1,\ldots,A_N\}$. Without loss of generality suppose that $A_1,\ldots,A_N$ are arranged 
in decreasing order of the value of $f$ on the cell $A_i$.
The lemma follows from lemma \ref{lemma::smallest}.

\begin{lemma}
\label{lem::piececonst}
With the notation introduced above suppose that we construct a set $A = \bigcup_{i=1}^t A_i$
and let $\theta = \int_A f(x)$ then 
we have that for any $\gamma \leq 1$
\begin{align*}
\int_{A} f^{\gamma}(x) dx \leq \inf_{B, \int_{B} f(x) \geq \theta} \int_B f^{\gamma}(x) dx.
\end{align*}
\end{lemma}

\noindent The following result is the discrete analogue of the one above. Suppose that 
we have a sequence $\{p_1,\ldots,p_d\}$ of positive numbers sorted as $p_1 \geq p_2 \geq \ldots p_d$.
By replacing Lebesgue measure in Lemma \ref{lemma::smallest} by the counting measure we get:

\begin{lemma}
\label{lem::flat}
Suppose we construct a set of indices $A = \{1,\ldots,t\}$ and let $\theta = \sum_{i=1}^t p_i$, then we
have that,
\begin{align*}
\sum_{i=1}^t p_i^{2/3} \leq 
\min_{\mathcal{J}} \sum_{j \in \mathcal{J}, \sum_{k \in \mathcal{{J}}} p_k \geq \theta} p_j^{2/3}.
\end{align*}
\end{lemma}

\subsection{Proof of Lemma~\ref{lem:uppermain}}
We divide the proof into two steps: the first step analyzes the output of Algorithm~\ref{alg:one},
and the second step analyzes the pruning of Algorithm~\ref{alg:two}.

\subsubsection{Analysis of Algorithm~\ref{alg:one}}
We analyze Algorithm~\ref{alg:one}, with the paramters: $\theta_1 = 1/(2\lipcons)$ and
$a, b = \critradn/1024$. We allow $\theta_2 > 0$ to be arbitrary.

Before turning our attention to the main properties, we verify that the partition created by Algorithm~\ref{alg:one} is indeed finite.
It is immediate to check that the partition $\mathcal{P}^{\dagger} = \{A_1,\ldots,A_{\bigN},A_{\infty}\}$ has the property that $P(A_{\infty}) \leq a + b,$ which yields the upper bound of property~\eqref{eqn:prop6f}.
We claim that no cell $A_i$
has very small diameter.  Recall that Algorithm~\ref{alg:one} is run on $S_{a}$ a set of probability content $1 - a$ (centered around the mean of $p_0$).
Define, 
\begin{align*}
p_{\min} = \frac{b}{\text{vol}(S_a)},
\end{align*}
Suppose that,
\begin{equation}
\text{diam}(A_i) < \frac{1}{4}\min \left\{ \theta_1 p_{\min}, \theta_2 p_{\min}^{\gamma} \right\},
\end{equation}
then let us denote the parent cell of $A_i$ by $U_i$ and its centroid by $y_i$. The parent cell $U_i$,
satisfies the condition that:
\begin{align*}
\text{diam}(U_i) <  \frac{1}{2} \min \left\{ \theta_1 p_{\min}, \theta_2 p_{\min}^{\gamma} \right\}.
\end{align*}
Since this cell was split, we must have that neither stopping rule~\eqref{eq:stop1} or~\eqref{eq:stop2} was satisfied. We claim that if the second stopping rule was not satisfied it must be the case that,
\begin{align*}
\nul(y_i) \leq \frac{p_{\min}}{2}.
\end{align*}
Indeed, if the second rule is not satisfied we obtain that:
\begin{align*}
\min \left\{ \theta_1p_0(y_i), \theta_2 p_0^{\gamma}(y_i) \right\} \leq \frac{1}{2} \min \left\{ \theta_1 p_{\min}, \theta_2 p_{\min}^{\gamma} \right\},
\end{align*}
which via some simple case analysis of the min's, together with the fact that $\gamma < 1$ yields the desired claim. Now using the Lipschitz property 
and the fact that $\theta_1 = 1/(2\lipcons)$,
we have that:
\begin{align*}
\sup_{x \in U_i} \nul(x) \leq \nul(y_i) + \lipcons \text{diam}(U_i) < \frac{p_{\min}}{2} + \frac{p_{\min}}{4} < p_{\min}.
\end{align*}
This means that the first stopping rule was in fact satisfied and we could not have split $U_i$.
This in turn means that every cell in our partition (excluding $A_{\infty}$) has diameter at least:
\begin{align*}
\text{diam}(A_i) > \frac{1}{4}\min \left\{ \theta_1 p_{\min}, \theta_2 p^{\gamma}_{\min} \right\}.
\end{align*}
This yields that our produced partition is finite and in turn that algorithm terminates in a finite number of steps.

\noindent {\bf Proof of Claim~\ref{eqn:prop1f}: } 
Our final task is to show that the partition satisfies the condition that,
\begin{align*}
\frac{1}{4} \min \left\{ \theta_1 \nul(x_i),\theta_2 \nul^\gamma(x_i) \right\} \leq \text{diam}(A_i) \leq \min \left\{ \theta_1 \nul(x_i),\theta_2 \nul^\gamma(x_i) \right\}.
\end{align*}
The upper bound is straightforward since it is enforced by our stopping rule. To observe that the lower bound is always satisfied we note that if
\begin{align*}
\text{diam}(A_i) < \frac{1}{4} \min \left\{ \theta_1 \nul(x_i),\theta_2 \nul^\gamma(x_i) \right\},
\end{align*}
then denoting the parent cell of $A_i$ to be $U_i$ (with centroid $y_i$) we obtain that,
\begin{align*}
\text{diam}(U_i) < \frac{1}{2} \min \left\{ \theta_1 \nul(x_i),\theta_2 \nul^\gamma(x_i) \right\}.
\end{align*}
Using this we obtain that,
\begin{align*}
\nul(y_i) \geq \nul(x_i) - \lipcons \text{diam}(U_i) \geq \frac{3}{4} \nul(x_i).
\end{align*}
This yields that,
\begin{align*}
\text{diam}(U_i) < \frac{1}{2(3/4)^\gamma} \min \left\{ \theta_1\nul(y_i), \theta_2 \nul^\gamma(y_i) \right\} <   \min \left\{ \theta_1\nul(y_i), \theta_2 \nul^\gamma(y_i) \right\},
\end{align*}
where in our final step we use the fact that $\gamma < 1$. This results in a contradiction since this means that $U_i$ satisfies our stopping rule and would not have been split.

\noindent {\bf Proof of Claim~\eqref{eqn:prop2f}: } This is a straightforward consequence of the previous property. In particular, we have that $\text{diam}(A_i) \leq \theta_1 p_0(x_i)$, with $\theta_1 = 1/(2\lipcons)$ so that,
\begin{align*}
\sup_{x \in A_i} p_0(x) \leq p_0(x_i) + \lipcons \frac{\theta_1 p_0(x_i)}{2} \leq \frac{5}{4} p_0(x_i). 
\end{align*}
Similarly,
\begin{align*}
\inf_{x \in A_i} p_0(x) \leq p_0(x_i) - \lipcons \frac{\theta_1 p_0(x_i)}{2} \leq \frac{3}{4} p_0(x_i),
\end{align*}
which yields the desired claim.

\subsubsection{Analysis of Algorithm~\ref{alg:two}}
We now turn our attention to studying the properties of the pruned partition $\mathcal{P} = \{A_1,\ldots,A_{\smallN},A_{\infty}\}$. For this algorithm, we choose $\theta_2 = \critradn/(8\lipcons \mu(1))$ and take $c = \critradn/512$.


\noindent {\bf Proof of Claim~\eqref{eqn:prop1f}: } The pruning algorithm completely eliminates some cells, adding them to $A_{\infty}$. 
In the case when $\mathcal{Q}(j^*) \leq c/5$ we change the diameter of the final cell $A_{\smallN},$
shrinking it by a $1 - \alpha$ factor. By definition $\alpha \leq 1/5$, and this yields Claim~\eqref{eqn:prop1f}.

\noindent {\bf Proof of Claim~\eqref{eqn:prop2f}: } Since the pruning step either eliminates cells, adding them to $A_{\infty}$, or reduces their diameter this claim follows directly from the fact that this property holds for $\mathcal{P}^{\dagger}$.

\noindent {\bf Proof of Claim~\eqref{eqn:prop6f}: } The pruning eliminates cells of total additional mass at most $c$ so we obtain that, $P(A_{\infty}) \leq a + b + c \leq \critradn/256$ verifying the upper bound in~\eqref{eqn:prop6f}. To verify the lower bound, we claim that the difference in the probability 
mass of the unpruned partition, $\{A_1,\ldots,A_{\bigN}\}$ and the pruned partition $\{A_1,\ldots,A_{\smallN}\}$ is at least $c/5$, i.e.
\begin{align*}
P_0\Big( \bigcup_{j=1}^{\tilde N}A_j\Big)-P_0\Big( \bigcup_{j=1}^{N}A_j\Big) \geq c/5.
\end{align*}
In the case when $\mathcal{Q}(j^*) \geq c/5$ the claim is direct. When this is not the case then the cell $A_N$ was too large, so that $\mathcal{Q}(j^*) + P_0(A_N) \geq c$, which implies that,
$P_0(A_N) \geq 4c/5$.
Let $x_N$ be the center of $A_N$.
Using property~\eqref{eqn:prop2f} and the fact that $(1 - \alpha))^d \leq (1 - \alpha)$ verify that,
\begin{align*}
P_0(D_1) \leq 4(1 - \alpha) P_0(A_N).
\end{align*}
Using the definition of $\alpha$ we obtain that $P_0(D_1) \geq c/5$ as desired.

\noindent {\bf Proof of Claim~\eqref{eqn:prop3f}: } We claim that the partition satisfies the property that,
\begin{align}
\label{eqn:randclaim}
\lipcons \sum_{i=1}^{\smallN} \text{diam}(A_i) \text{vol}(A_i) \leq \frac{\critradn}{4}.
\end{align}
Taking this claim as given we verify the property~\eqref{eqn:prop3f}.
We divide the proof into two cases:
\begin{enumerate}
\item $P(A_{\infty}) \geq \critradn/4$: In this case we obtain that,
\begin{align*}
\sum_{i=1}^N | P_0(A_i) - P(A_i) | + | P_0(A_{\infty}) - P(A_{\infty})| \geq 
| P_0(A_{\infty}) - P(A_{\infty})| \geq \critradn/8,
\end{align*}
using the upper bound in property~\eqref{eqn:prop6f}.

\item $P(A_{\infty}) \leq \critradn/4$: In this case we observe that,
\begin{align*}
\int_{A_{\infty}} |p_0(x) - p(x)| dx \leq \int_{A_{\infty}} p_0(x) dx +  \int_{A_{\infty}} p(x) dx \leq \frac{3\critradn}{\contwo},
\end{align*}
and this yields that,
\begin{align*}
\int_{\mathbb{R}^d \backslash A_{\infty}} | p_0(x) - p(x)| dx \geq \critradn (1 - 3/\contwo).
\end{align*}
Now denoting by $\bar{p}$ the approximation of $p$ by a density equal to the average of $p$ on each cell of the partition we have that,
\begin{align*}
\int_{\mathbb{R}^d \backslash A_{\infty}} | p_0(x) - p(x)| dx
\leq & \int_{\mathbb{R}^d \backslash A_{\infty}}| p_0(x) - \bar{p_0}(x)| dx + 
\int_{\mathbb{R}^d \backslash A_{\infty}}| p(x) - \bar{p}(x)| dx \\
& + \sum_{i=1}^N | p_0(A_i) - p(A_i) | dx.
\end{align*}
For any $\lipcons$-Lipschitz density we have that,
\begin{align*}
\int_{\mathbb{R}^d \backslash A_{\infty}}| p(x) - \bar{p}(x)| dx \leq 
 \lipcons \sum_{i=1}^N \text{diam}(A_i) \text{vol}(A_i) \leq \frac{\critradn}{4},
\end{align*}
using claim~\eqref{eqn:randclaim}.
This yields that,
\begin{align*}
\sum_{i=1}^N | p_0(A_i) - p(A_i) | + | p_0(A_{\infty}) - p(A_{\infty})| \geq
\sum_{i=1}^N | p_0(A_i) - p(A_i) | \geq  \critradn(1 - 7/\contwo) = \critradn/\contwo,
\end{align*}
as desired.
\end{enumerate}

\noindent It remains to prove claim~\eqref{eqn:randclaim}. Notice that,
\begin{align*}
\lipcons \sum_{i=1}^{\smallN} \text{diam}(A_i) \text{vol}(A_i) &\leq 
\sum_{i=1}^N \min \left\{ \theta_1 p_0(x_i), \theta_2 p_0^{\gamma}(x_i)\right\} \text{vol}(A_i) \stackrel{\text{(i)}}{\leq} 2 \int_{\mathbb{R}^d}  \left\{ \theta_1 p_0(x),\theta_2 p_0^{\gamma}(x)\right\} dx \\
&= 2 \int_{\mathbb{R}^d}   \left\{ \frac{p_0(x)}{2\lipcons}, \frac{\critradn p_0^{\gamma}(x)}{8 \lipcons \mu(1/4)} \right\} dx \\
&\stackrel{\text{(ii)}}{=} \frac{\critradn}{4},
\end{align*}
where step (i) uses property~\eqref{eqn:prop2f} and (ii) uses the definition of $\mu$ in~\eqref{eqn:mudef}.

\noindent {\bf Proof of Claim~\eqref{eqn:prop4f}: } Recall that we have chosen $c = \critradn/512$. In order to prove this claim we need to use properties of the pruning step. Let us define $\widetilde{p}_0(x)$  
as the piecewise constant density formed by replacing $p_0(x)$ by its maximum value over the cell containing $x$, and $0$ outside the support of $\{A_1,\ldots,A_{\bigN}\}$. We note that,
\begin{align}
\label{eqn::one}
 \int_{K} p_0^{\gamma}(x) dx  \leq \int_{K} \widetilde{p}_0^{\gamma}(x) dx.
\end{align}
Now, abusing notation slightly and ignoring the set $A_{\infty}$ we denote the original partition as $\{A_1,\ldots,A_{\bigN}\}$, which we take as sorted by the values in $g_0$, 
and the pruned partition as $\{A_1,\ldots,A_{\smallN}\}$, noting that we might potentially have split the last cell $A_{\smallN}$ into two cells. We let $A = \bigcup_{i=1}^{\bigN} A_i$.
Let us denote,
\begin{align*}
\mathcal{B} = \left\{B: B \subset A, \int_{B^c} \widetilde{p}_0(x) dx \leq \int_{K^c} \widetilde{p}_0(x) dx \right\}.
\end{align*}
Using Lemma~\ref{lem::piececonst} we obtain that,
\begin{align}
\label{eqn::two}
\int_{K} \widetilde{p}_0^{\gamma}(x) dx \leq \inf_{B \in \mathcal{B}} \int_B \widetilde{p}_0^{\gamma}(x) dx.
\end{align}
Noting, that 
\begin{align*}
\int_{K^c} \widetilde{p}_0(x) dx \geq \int_{K^c} p_0(x) dx \geq \frac{c}{5},
\end{align*}
and defining,
\begin{align*}
\mathcal{C} = \left\{C: C \subset A, \int_{C^c} \widetilde{p}_0(x) dx \leq \frac{c}{5} \right\}.
\end{align*}
we obtain that,
\begin{align}
\label{eqn::three}
 \inf_{B \in \mathcal{B}} \int_B \widetilde{p}_0^{\gamma}(x) dx \leq \inf_{C \in \mathcal{C}} \int_C \widetilde{p}_0^{\gamma}(x) dx.
\end{align}
Defining, 
\begin{align*}
\mathcal{D} = \left\{D: D \subset A, \int_{D^c} p_0(x) dx \leq \frac{c}{10} \right\},
\end{align*}
we see that
$$
{\cal D}\subset {\cal C} \subset {\cal B}
$$
so that
\begin{align}
\label{eqn::four}
 \inf_{C \in \mathcal{C}} \int_C \widetilde{p}_0^{\gamma}(x) dx \leq  2^\gamma \inf_{D \in \mathcal{D}} \int_D p_0^{\gamma}(x) dx \leq 2^\gamma T^{\gamma}_{c/10}.
\end{align} 
Putting together Equations~\eqref{eqn::one}, \eqref{eqn::two}, \eqref{eqn::three} and~\eqref{eqn::four} we obtain the desired result.

\noindent {\bf Proof of Claim~\eqref{eqn:prop5f}: } In order to lower bound the density over the pruned partition we 
will show that our pruning step is approximately a level set truncation.
We have that for any point $x$ that is removed and any point $y$ that is retained it must be the case that,
\begin{align*}
p_0(x) \leq 2p_0(y).
\end{align*}
where we used property~\eqref{eqn:prop2f}.
Let $K$ denote the set of points retained by the pruning. The above observation yields that,
there exists some $t \geq 0$ such that,
\begin{align*}
\{p_0 \geq t \} \subseteq K \subseteq \{ p_0 \geq t/2\}. 
\end{align*}
We know that $\int_{K} p_0(x) dx \leq 1 - c/10$. 
Consider, the set 
\begin{align*}
G(u) = \left\{ x:  p_0(x) \geq u \right\}.
\end{align*}
Suppose that for some $u$ we can show that,
\begin{align*}
\mathbb{P}(K) \leq \mathbb{P}(G(u)),
\end{align*}
then we can conclude that $t \geq u$, and further that the density on $K$ is at least $u/2$. It thus only remains to find a value $u$ such that $\mathbb{P}(G(u)) \geq 1 - c/10$.
Suppose we choose $u = \left(\frac{c}{10 \mu(c/(10\epsilon))}\right)^{1/(1-\gamma)},$ and recall that,
\begin{align*}
\epsilon = \int \min \left\{ \frac{p_0(x)}{c/(10\epsilon)}, \frac{\epsilon p_0^{\gamma}(x)}{\mu(c/(10\epsilon))} \right\} dx. 
\end{align*}
Over the set $G^{c}$ the minimizer is always the first term above which yields,
\begin{align*}
\epsilon \geq \int_{G^c} \frac{p_0(x)}{c/(10\epsilon)} dx,
\end{align*}
i.e. that $\mathbb{P}(G^c) \leq c/10$, as desired. This in turn yields the claim.

\subsection{Proof of Lemma~\ref{lem:lipvfunc}}
To show this, it suffices to show that
more mass is truncated from $q$ than is truncated from $p$, i.e.
that
\begin{equation}\label{eq::more}
\sum_{s+1}^{\smallN + 1} q_i \geq P(A_{\infty}),
\end{equation}
and then we apply
Lemma~\ref{lem::flat}.
To show (\ref{eq::more}) we proceed as follows.
Note that $P(A_{\infty}) = q_a$ for some $a$.
If $s\leq a$ then 
$\sum_{s+1}^{\smallN + 1} q_i \geq P(A_{\infty})$ follows immediately.
Now suppose that
$s > a$.
From the definition of $s$ we know that
$q_s + \sum_{s+1}^{N+1}q_i \geq \epsilon/128$
so that
$\sum_{s+1}^{N+1}q_i \geq \epsilon/128 - q_s$.
Since $s>a$,
$q_s \leq P(A_{\infty}) \leq \epsilon/256$ and so
$\sum_{s+1}^{N+1}q_i \geq \epsilon/256 \geq P(A_{\infty})$ so that
$\sum_{s+1}^{\smallN + 1} q_i \geq P(A_{\infty})$
as required.
Thus (\ref{eq::more}) holds.

\newcommand{\sigman}{\widetilde{\sigma}}
\newcommand{\critradf}{w_n}
\newcommand{\critradfa}{w^a_n}

\section{Adapting to Unknown  Parameters}
\label{app:adapt}
In this section, we consider ways to choose the parameter $\sigma$ for the max test, and for the test in \cite{valiant14}, and then consider tests that are adaptive to the typically unknown smoothness parameter $\lipcons$.

\subsection{Choice of $\sigma$}
The max test and the test from \cite{valiant14} require choosing the truncation parameter $\sigma = \critradn/8$. 
In typical settings, we do not assume that $\critradn$ is known. We consider the case of the test from \cite{valiant14} though our ideas generalize to the max test in a straightforward way. 

Perhaps the most natural way to choose the parameter $\sigma$ is to solve the critical equation and choose $\sigma$ accordingly, i.e.
we find $\sigman$ that satisfies:
\begin{align}
\label{eqn:sigman}
\sigman = \max\left\{ \frac{1}{n}, \sqrt{ \frac{V_{\sigman/16}(p_0)}{n}} \right\},
\end{align}
and then 
we choose the tuning parameter
$\sigma := C \max\{1/\alpha,1/\zeta\} \sigman$, for a sufficiently large constant $C > 0$. 

When the unknown $\critradn \geq 8 C \max\{1/\alpha,1/\zeta\} \sigman,$ then it is clear that our choice
guarantees that the tuning parameter $\sigma$ is chosen sufficiently small, i.e. $\sigma \leq \critradn/8$ as desired. It is also clear that the test has size at most $\alpha$. It remains to understand the Type II error.
Inverting the above relationship we see that,
\begin{align*}
n = \max \left\{ \frac{1}{\sigman},  \frac{ V_{\sigman/16} (p_0)}{ \sigman^2 }\right\}.
\end{align*}
Noting that $\sigma/2 \geq \sigman/16$, and that $C \max\{1/\alpha,1/\zeta\} \geq 1$ we obtain that,
\begin{align*}
n \geq C \max \{1/\alpha, 1/\zeta \} \max \left\{ \frac{1}{\sigma}, \frac{ V_{\sigma/2} (p_0)}{ \sigma^2 }\right\} \geq  C \max \{1/\alpha, 1/\zeta \} \max \left\{ \frac{1}{\sigma},\frac{ V_{\sigma/2} (p_0)}{ \critradn^2 }\right\}.
\end{align*}
An application of Lemma~\ref{lem:valinter} shows that the Type II error of the test is at most $\zeta$ as desired. Thus, we see that this test provides the same result as the test in Theorem~\ref{thm:valup} without knowledge of $\critradn$.

Although adequate from a theoretical perspective, the previous choice of the parameter depends on an unknown (albeit universal) constant. An alternative is to consider a range of possible values for the parameter $\sigma$ and appropriately adjust the threshold $\alpha$ via a Bonferroni correction. One natural range is to consider scalings of the parameter $\sigman$ in~\eqref{eqn:sigman}. More generally, if we considered $\Sigma = \{\sigma_1,\ldots,\sigma_K\}$, a natural goal would be to compare the risk of the Bonferroni corrected test to the oracle test which minimizes the risk over the set $\Sigma$ of possible tuning parameters. We leave a more detailed analysis of this test to future work.

\subsection{Adapting to unknown $\lipcons$}
Our tests for Lipschitz testing, in addition to assuming knowledge of $\critradn$ use knowledge of the Lipschitz constant $\lipcons$ in constructing the binning. The techniques from the previous section can be used to construct tests without knowledge of $\critradn$. Constructing 
tests which are adaptive to unknown smoothness parameters is a problem which has received much attention in classical works. 
We focus only on establishing upper bounds. Some lower bounds follow from standard arguments and we highlight important open questions in the sequel.

In order to define precisely the notion of an adaptive test, we follow the prescription of \citet{spokoiny96} (see also \cite{ingster1997adaptive,gine15}). 
As in~\eqref{eqn:liprad} define a sequence of critical radii $\critradf(p_0,L)$ as the solutions to the critical equations:
\begin{align*}
\critradf(p_0,L) = \left( \frac{L^{d/2}  T_{c \critradf(p_0,L)}(p_0)}{n} \right)^{2/(4+d)}
\end{align*}
for a sufficiently small constant $c > 0$. We now define the adaptive upper critical radii as the solutions to the critical equations:
\begin{align}
\label{eqn:adapliprad}
\critradfa(p_0,L) =  \left( \frac{L^{d/2}  \log(n)T_{c \critradfa(p_0,L)}(p_0)}{n} \right)^{2/(4+d)}.
\end{align}
We can upper bound the ratio:
\begin{align*}
\frac{\critradfa(p_0,L)}{\critradf(p_0,L)} \leq (\log(n))^{2/(4+d)}.
\end{align*}
This ratio upper bounds the price for adaptivity. It will be necessary to distinguish the (known) smoothness parameter of the null from the possibly unknown parameter $\lipcons$ in~\eqref{test:lip}.
We will denote the smoothness parameter of $p_0$ by $L_0$. We note that in the setting where $\lipcons$ was known, we assumed that both $p_0, p \in \mathcal{L}(\lipcons)$ and this in turn requires that $\lipcons \geq L_0$. 

We take $\alpha, \zeta > 0$ to be fixed constants.
For a sufficiently large constant $C > 0$ we define the class of densities:
\begin{align*}
\mathcal{L}(\lipcons, \critradfa) = \{ p: p \in \mathcal{L}(\lipcons),   \|p - p_0\|_1 \geq C \max\{1/\alpha,1/\zeta\} \critradfa(p_0, \lipcons) \}.
\end{align*}
For some $p_0 \in \mathcal{L}(L_0)$, consider 
the hypothesis testing problem of distinguishing:
\begin{align}
\label{eqn:adapttest}
H_0: p = p_0, p_0 \in \mathcal{L}(L_0) ~~~\text{versus}~~~H_1: p \in \bigcup_{\lipcons \geq L_0} 
\mathcal{L}(\lipcons,\critradfa).
\end{align}
In order to precisely define our testing procedure we first show that there are natural upper bounds
on $\lipcons$. In particular, we claim that when $\lipcons \gg n^{2/d} L_0$ then the critical radius
remains lower bounded by a constant.

We have the following lemma. We let $C_{\ell},c > 0$ denote universal constants. 
\begin{lemma}
\label{lem:aux}
If $\lipcons \geq C_{\ell} n^{2/d} L_0$, then 
\begin{align*}
\critradn(p_0,\lipcons) \geq c. 
\end{align*}
\end{lemma}
Thus we restrict our attention to the regime where $\lipcons \in [L_0, Cn^{2/d} L_0]$, for a sufficiently large constant $C > 0$. A natural strategy is then to consider a discretization of the set of possible values for $L_n$, 
\begin{align*}
\mathfrak{L} = \{L_0, 2L_0, \ldots, 2^{ \log_2 (Cn^{2/d})} L_0\}. 
\end{align*} 
Our adaptive test then simply performs the binning test described in Theorem~\ref{thm:main} for each choice of $\lipcons \in \mathfrak{L}$, with the threshold $\alpha$ reduced by a factor of $\lceil \log_2 (Cn^{2/d}) + 1 \rceil$.
We refer to this test as the adaptive Lipschitz test. We have the following result:
\begin{theorem}
\label{thm:adapt}
Consider the testing problem in~\eqref{eqn:adapttest}. The adaptive Lipschitz test has Type I error at most $\alpha$, and has Type II error at most $\zeta$.
\end{theorem}
\noindent {\bf Remarks: }
\begin{itemize}
\item Comparing the non-adaptive critical radii in~\eqref{eqn:liprad} and the adaptive critical radii in~\eqref{eqn:adapliprad} we see that we lose a factor of $(\log(n))^{2/(4+d)}.$ A natural question 
is whether such a loss is necessary.
\item Classical results \citep{ingster1997adaptive} consider adapting to an unknown H\"{o}lder exponent $s$ and show that for testing uniformity (with deviations in the $\ell_2$ metric) a loss of a factor $(\sqrt{\log\log(n)})^{2s/(4s+d)}$ is necessary and sufficient. In our setting, the loss is of a logarithmic factor instead of a $\log \log$ factor and this is a consequence of the fact that the high-dimensional 
multinomial tests we build on \cite{valiant14} are not in general exponentially consistent (i.e. their power and size do not tend to zero at an exponential rate). We hope to develop a more precise understanding of this situation in future work.
\end{itemize}

\begin{proof}
The proof follows almost directly from our previous analysis of Theorem~\ref{thm:main} so we only provide a brief sketch.
It is straightforward to check that the Bonferroni correction controls the size of the adaptive Lipschitz test at $\alpha$. Let $j^*$ denote the smallest integer such that, $2^{j^*} L_0 \geq \lipcons.$
In order to bound the Type II error, it is sufficient to show that under the alternate, the test
corresponding to the index $j^*$ rejects the null hypothesis with probability at least $1 - \zeta$. Noting that the ratio $2^{j^*} L_0 /\lipcons \leq 2$ this follows directly from the proof of Theorem~\ref{thm:main}. 
\end{proof}

\subsubsection{Proof of Lemma~\ref{lem:aux}}
In order to establish this claim, it suffices to show that the lower bound on the critical radius in~\eqref{eqn:liprad} is at least a constant. By the monotonicity of the critical equation, it suffices to show that for some small constant $c > 0$ we have that,
\begin{align*}
c \leq \left( \frac{\lipcons^{d/2} T_{Cc}(p_0)}{n} \right)^{2/(4+d)},
\end{align*}
where $C > 0$ is the universal constant in~\eqref{eqn:liprad}.
We choose $c < C/2$ so we obtain that it suffices to show,
\begin{align*}
c \leq \left( \frac{\lipcons^{d/2} T_{1/2}(p_0)}{n} \right)^{2/(4+d)}.
\end{align*}
We claim that for any $p_0 \in \mathcal{L}(L_0)$ there is a universal constant $C_1 > 0$ such that,
\begin{align}
\label{eqn:lbclaim}
T_{1/2}(p_0) \geq \frac{C_1}{L_0^{d/2}}.
\end{align}
Taking this claim as given for now we see that,
\begin{align*}
\left( \frac{\lipcons^{d/2} T_{1/2}(p_0)}{n} \right)^{2/(4+d)} \geq 
\left( \frac{C_1 \lipcons^{d/2}}{L_0^{d/2} n} \right)^{2/(4+d)} \geq \left(\frac{C_1}{C_{\ell}^{d/2}}\right)^{2/(4+d)} \geq c,
\end{align*}
as desired. 

\noindent {\bf Proof of Claim~\eqref{eqn:lbclaim}: } 
As a preliminary we first produce an upper bound on 
any Lipschitz density.
We claim that, there exists a constant $C > 0$ depending only on the dimension such that any $L_0$-Lipschitz density $p_0$ is upper bounded as $\|p_0\|_{\infty} \leq C L_0^{d/(d+1)}.$

Without loss of generality let us suppose the density $p_0$ is maximized at $x = 0$. The density $p_0$ is then lower bounded by the function,
\begin{align*}
g_0(x) = \left(\|p_0\|_{\infty} - L_0 \|x\|\right) \mathbb{I}( \|p_0\|_{\infty} - L_0 \|x\| \geq 0).
\end{align*}
The integral of this function is straightforward to compute, and since $p_0$ must integrate to 1 we obtain that,
\begin{align*}
1 = \int_x p_0(x) dx \geq \int_x g_0(x) dx = \frac{v_d}{d+1}  \frac{\|p\|_{\infty}^{d+1}}{L_0^d},
\end{align*}
where $v_d$ denotes the volume of the $d$-dimensional unit ball.
This in turn yields the upper bound,
\begin{align*}
\|p\|_{\infty} \leq \left( \frac{d+1}{v_d} \right)^{1/(d+1)} L_0^{d/(d+1)},
\end{align*}
as desired.
With this result in place we can lower bound the truncated $T$-functional. In particular, letting $B_{\sigma}$ denote a set of probability content $1 - \sigma$ that (nearly) minimizes
the truncated $T$-functional we have that,
\begin{align*}
T_{\sigma}^{\gamma}(p_0) =  \int_{B_{\sigma}} p_0^{\gamma}(x) dx \geq \int_{B_{\sigma}} \frac{p_0(x)}{\|p\|_{\infty}^{1 - \gamma}} dx \geq \left( \frac{v_d}{d+1}\right)^{1/(3+d)}  \frac{1 - \sigma}{L_0^{d/(3+d)}},
\end{align*}
which gives the bound,
\begin{align*}
T_{\sigma}(p_0) \geq \left( \frac{v_d}{d+1}\right)^{1/2}  \frac{ (1 - \sigma)^{(3+d)/2}} {L_0^{d/2}},
\end{align*}
as desired. Taking $\sigma = 1/2$ yields the desired claim.

\clearpage

\section{Additional Simulations}
\label{app:extra_sims}

\begin{center}
\begin{figure}
\begin{tabular}{cc}
\includegraphics[scale=0.4]{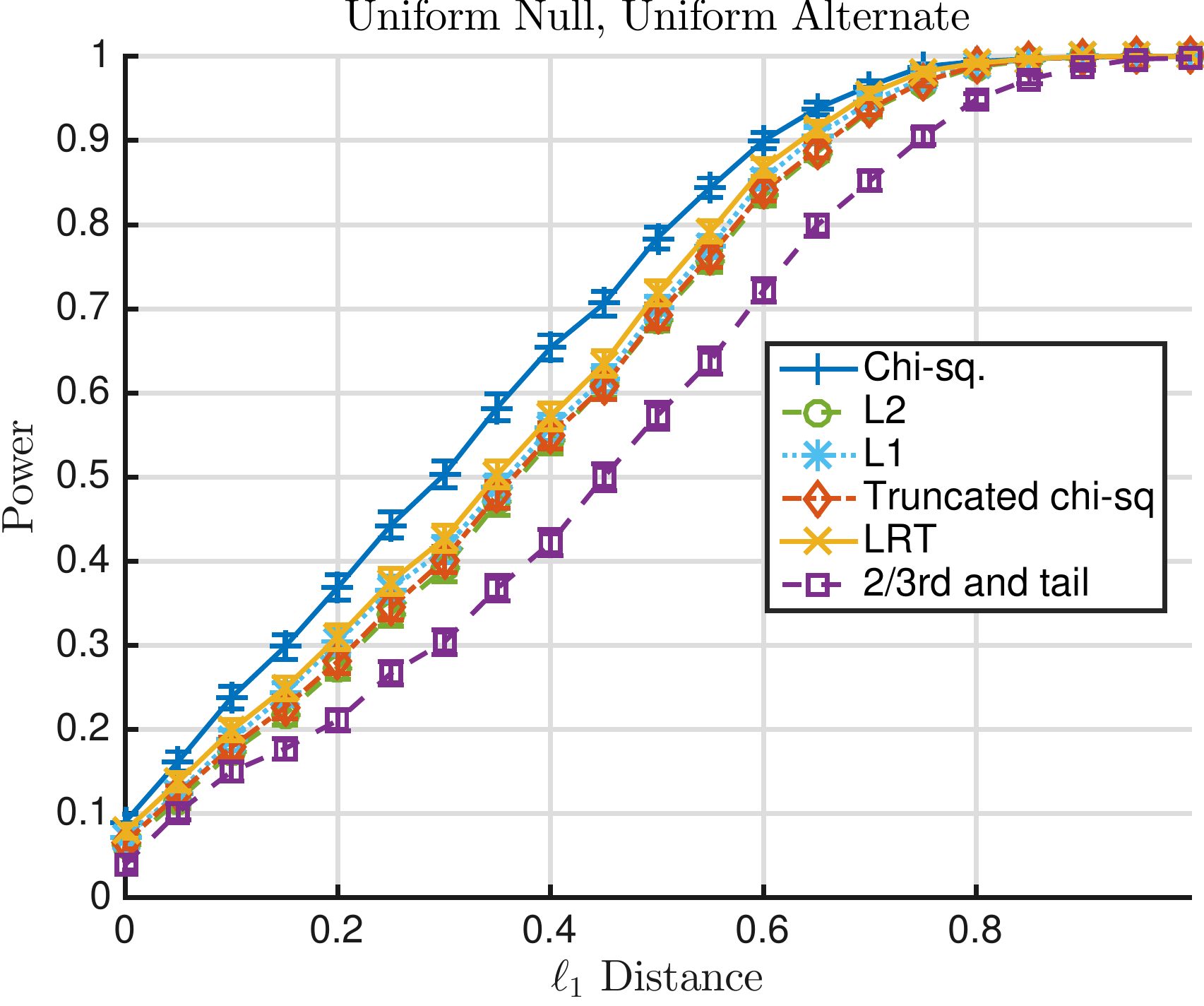} & ~~~~~\includegraphics[scale=0.4]{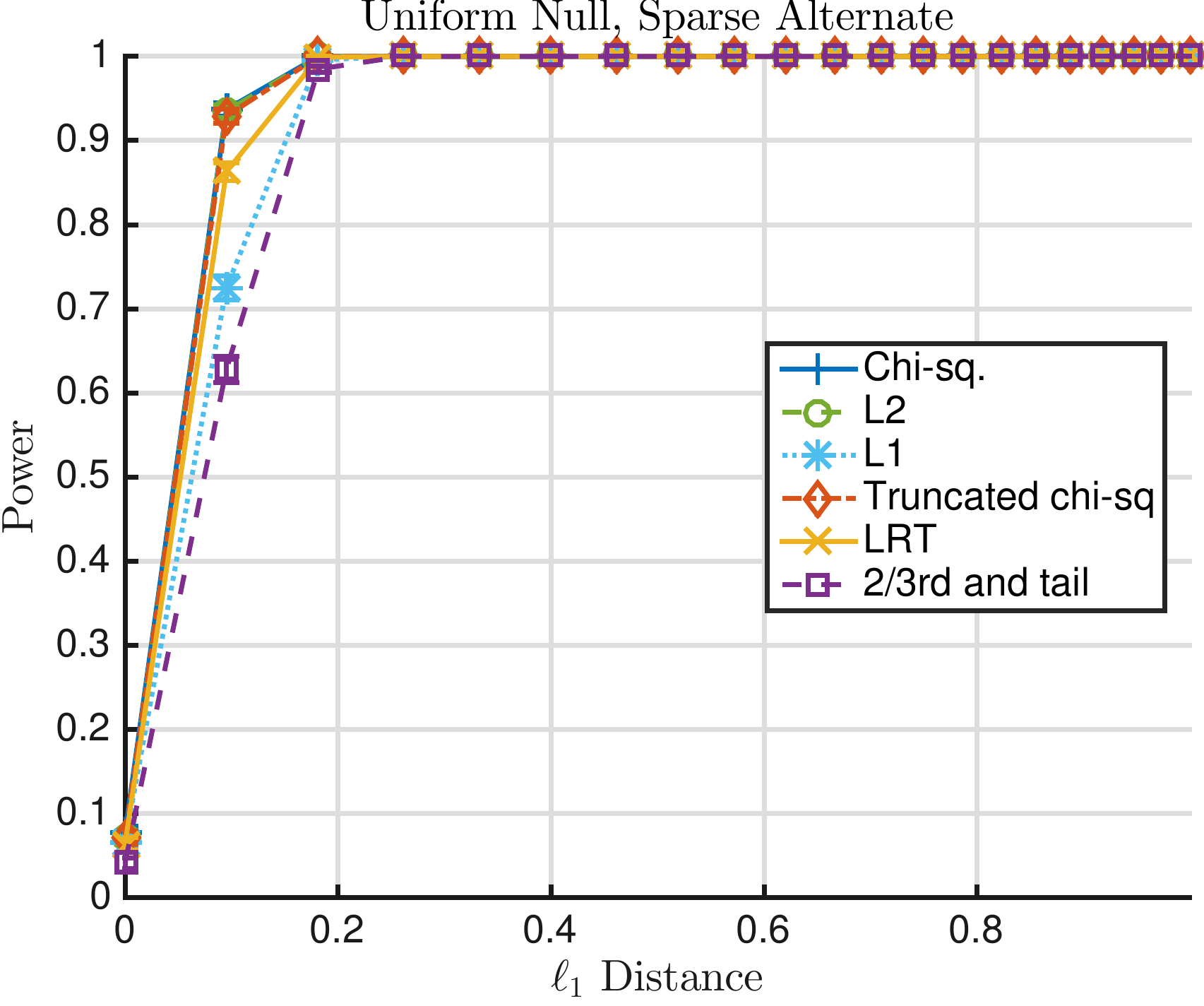}
\end{tabular}
\caption{A comparison between the truncated $\chi^2$ test, the 2/3rd + tail test \cite{valiant14}, 
the $\chi^2$-test, the likelihood ratio test, the $\ell_1$ test and the $\ell_2$ test. 
The null is chosen to be uniform, and the alternate is either a dense or sparse perturbation of the null. The power of the tests are plotted against the $\ell_1$ distance between the null and alternate. Each point in the graph is an average over 1000 trials. Despite the high-dimensionality (i.e. $n = 200, d = 2000$) the tests have high-power, and perform comparably.}
\label{fig:unifapp}
\end{figure}
\end{center}
In this section we re-visit the simulations for multinomials. 
We add comparisons to tests based on the $\ell_1$ and $\ell_2$ statistics which are given as
\begin{align*}
T_{\ellone} = \sum_{i=1}^d |X_i - np_0(i)|,
\end{align*}
and 
\begin{align*}
T_{\elltwo} = \sum_{i=1}^d  (X_i - np_0(i))^2.
\end{align*}
In addition to the alternatives that are created by dense and sparse perturbations of the null we also consider two other perturbations: one where we perturb each coordinate of the null by an amount proportional to the entry $p_0(i)$, and one where we perturb each coordinate by $p_0(i)^{2/3}$, in magnitude with a Rademacher sign. The latter perturbation is close to the worst-case perturbation considered by \cite{valiant14} in their proof of local minimax lower bounds. We take $n = 200, d = 2000$ and each point in the graph is an average over 1000 trials.

Once again we observe that the truncated $\chi^2$ test we propose, and the 2/3rd + tail test from \cite{valiant14} are remarkably robust. All tests are comparable when the null is uniform, while distinctions are clearer for the power law null. 
The $\ell_2$ test appears to have high-power against sparse alternatives suggesting potential avenues for future investigation.

\begin{center}
\begin{figure}
\begin{tabular}{cc}
\includegraphics[scale=0.4]{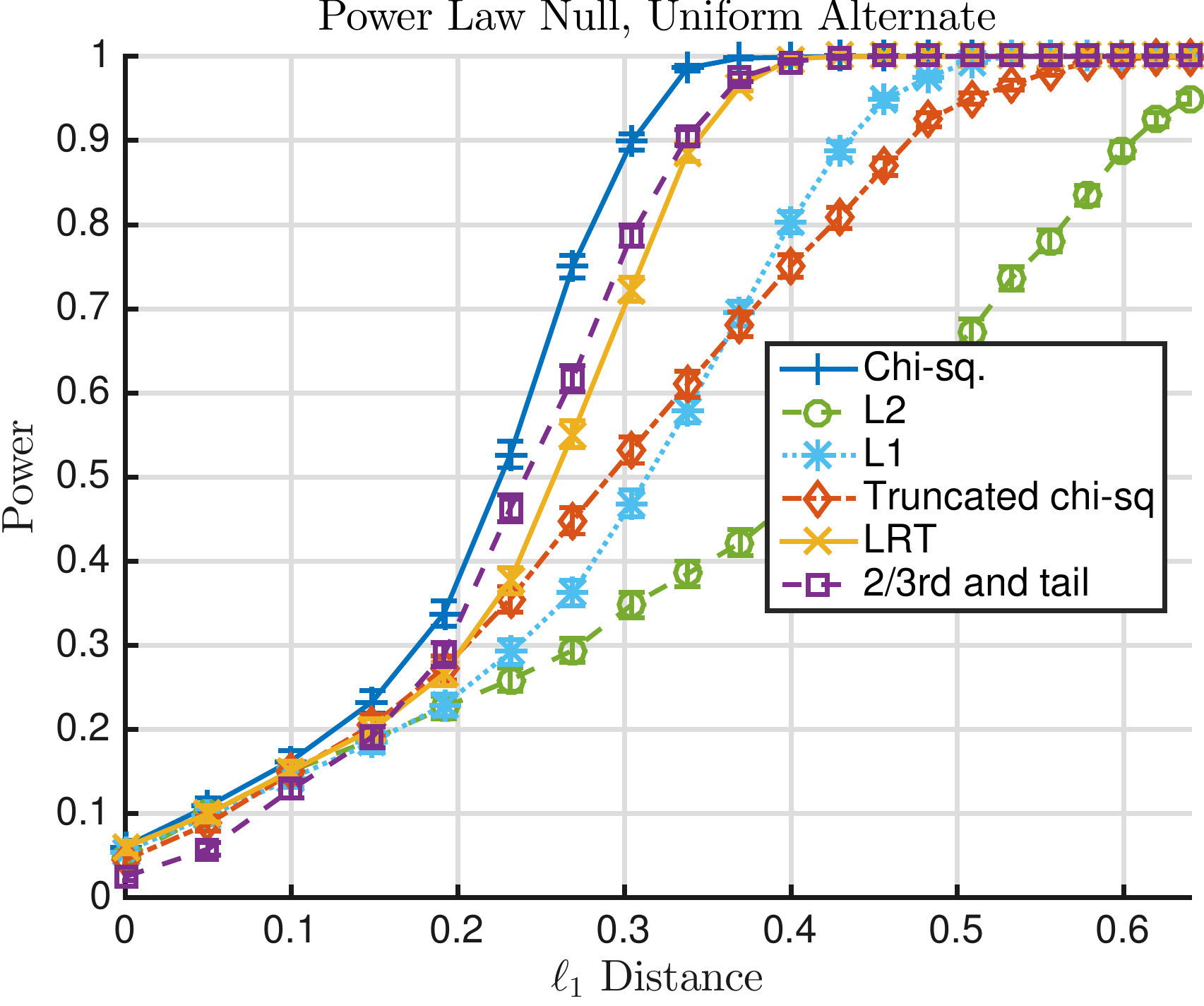} & ~~~~~\includegraphics[scale=0.4]{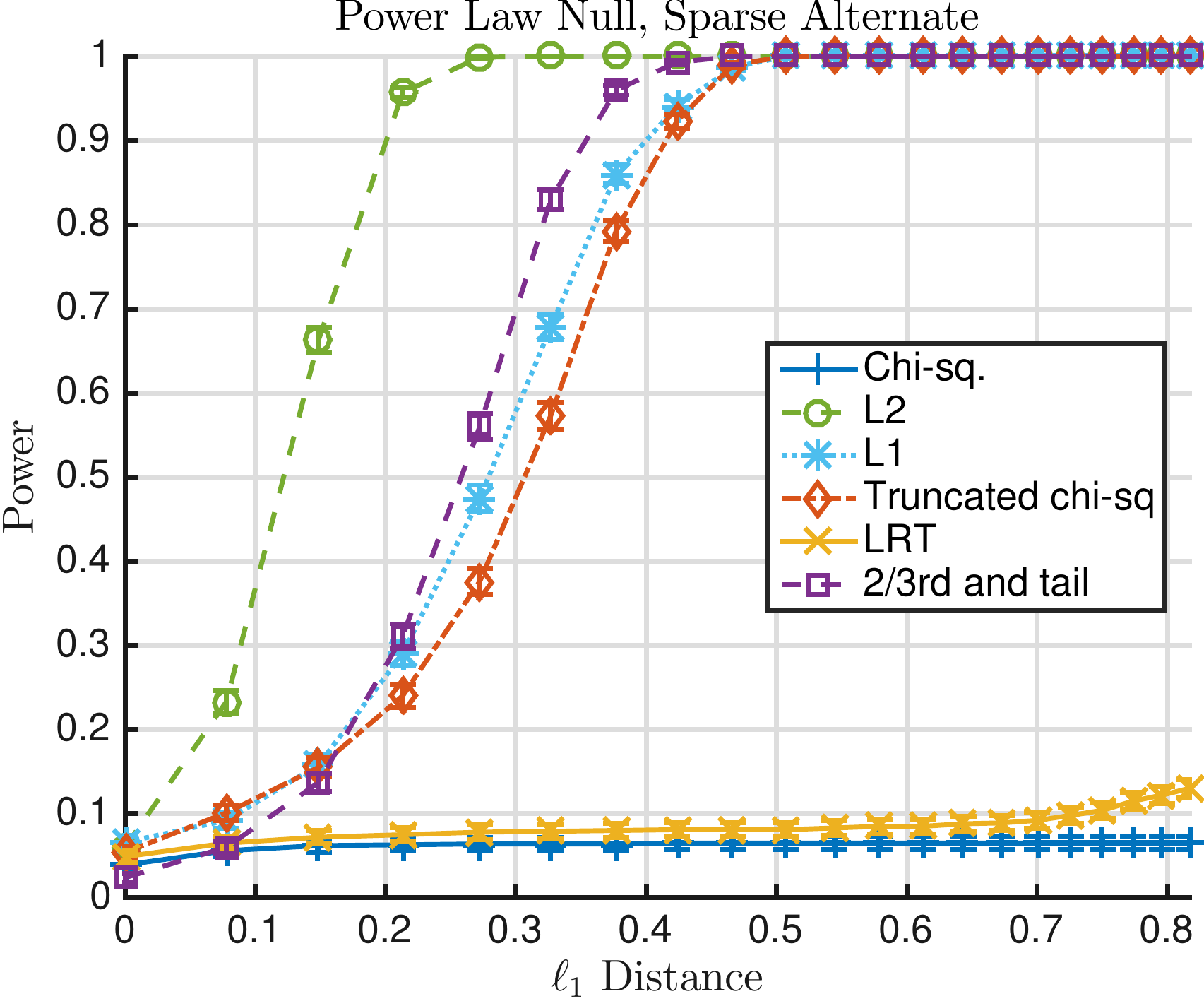} \\
\includegraphics[scale=0.4]{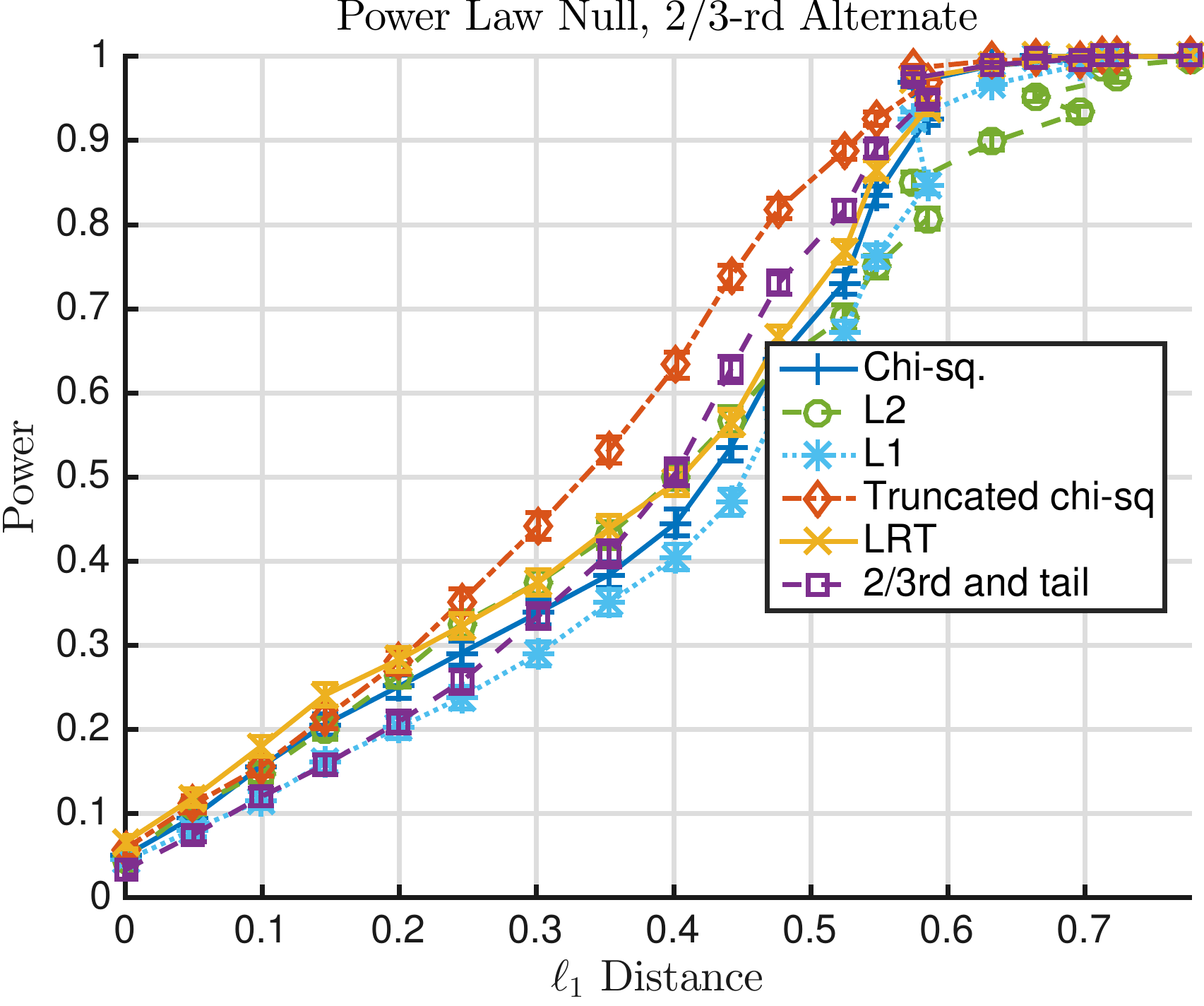} & ~~~~~\includegraphics[scale=0.4]{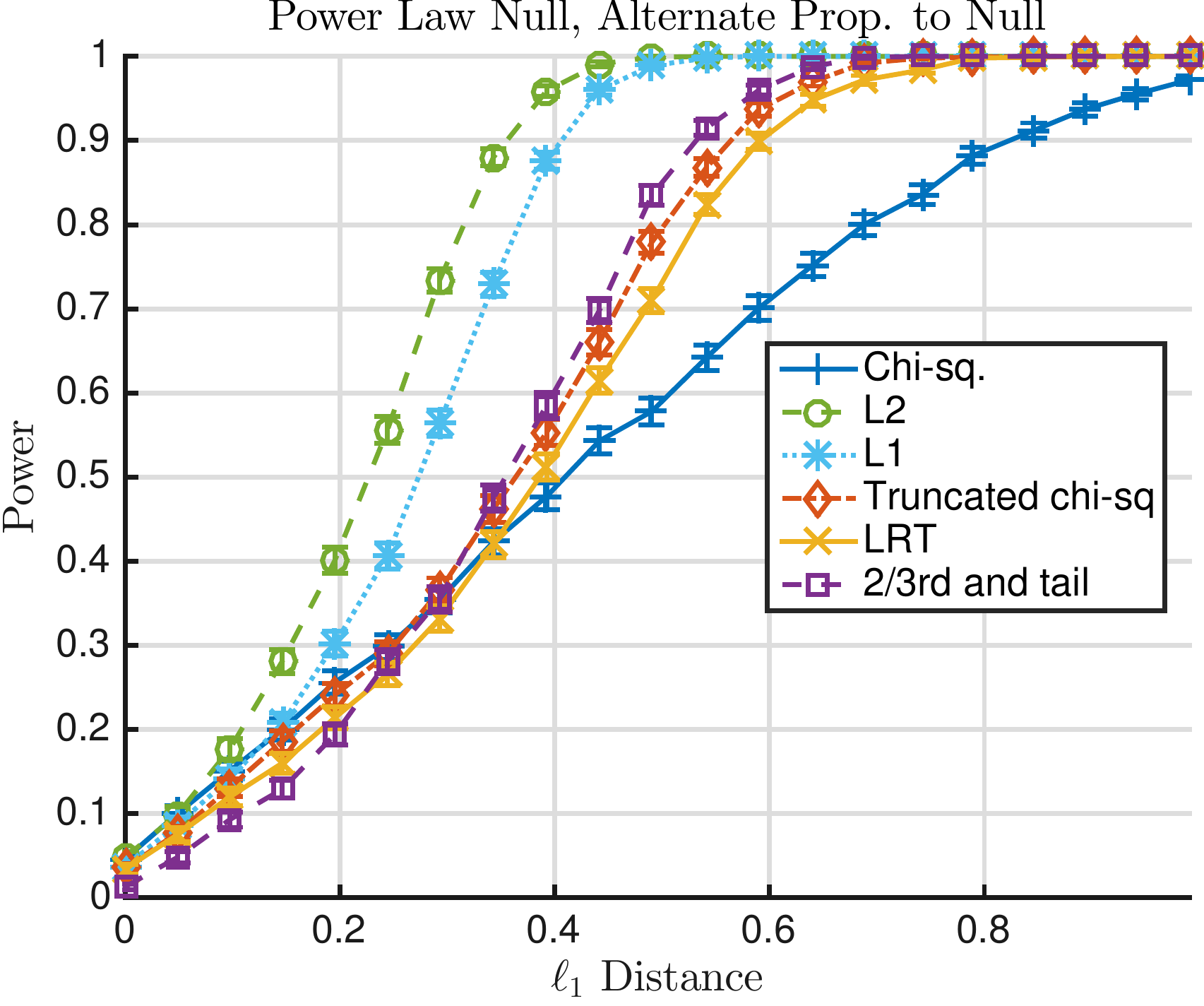} \\
\end{tabular}
\caption{A comparison between the truncated $\chi^2$ test, the 2/3rd + tail test \cite{valiant14}, 
the $\chi^2$-test, the likelihood ratio test, the $\ell_1$ test and the $\ell_2$ test. 
The null is chosen to be a power law with $p_0(i) \propto 1/i$.
We consider four possible alternates, the first uniformly perturbs the coordinates, the second
is a sparse perturbation only perturbing the first two coordinates, the third perturbs each co-ordinate proportional to $p_0(i)^{2/3}$ and the final setting perturbs each coordinate proportional to $p_0(i)$. The power of the tests are plotted against the $\ell_1$ distance between the null and alternate. Each point in the graph is an average over 1000 trials. }
\label{fig:powerapp}
\end{figure}
\end{center}

\clearpage

\end{document}